\documentclass[12pt,a4paper]{article}
\usepackage[top=2.5cm, bottom=2.5cm, left=2.5cm, right=2.5cm]{geometry}

\usepackage{amsmath,amssymb,amsthm,mathrsfs,mathtools}
\usepackage{graphicx,tikz}
\usepackage{empheq}
\usepackage{braket}
\usepackage{comment}
\usepackage{physics} 

\usetikzlibrary{decorations.pathreplacing,calligraphy,cd}

\theoremstyle{definition}
\newtheorem{dfn}{Definition}[section] 
\newtheorem{ex}[dfn]{Example}
\newtheorem{hyp}[dfn]{Hypothesis}
\newtheorem{rem}[dfn]{Remark}

\theoremstyle{plain}
\newtheorem{maintheorem}{Theorem}

\newtheorem{mainconjecture}[maintheorem]{Conjecture}

\newtheorem{thm}[dfn]{Theorem}
\newtheorem{lem}[dfn]{Lemma}
\newtheorem{prp}[dfn]{Proposition}
\newtheorem{cor}[dfn]{Corollary}

\numberwithin{equation}{section}

\DeclareMathOperator{\Ind}{Ind}
\DeclareMathOperator{\ind}{ind}
\DeclareMathOperator{\res}{res}
\DeclareMathOperator{\con}{con}
\DeclareMathOperator{\Cok}{Coker}
\DeclareMathOperator{\Ker}{Ker}
\DeclareMathOperator{\Ad}{Ad}
\DeclareMathOperator{\pullback}{pb}
\DeclareMathOperator{\Aut}{Aut}
\DeclareMathOperator{\id}{id}
\DeclareMathOperator{\Tor}{Tor}

\newcommand{\ve}{\varepsilon}
\newcommand{\vp}{\varphi}

\newcommand{\C}{\mathbb{C}}
\newcommand{\R}{\mathbb{R}}
\newcommand{\Q}{\mathbb{Q}}
\newcommand{\Z}{\mathbb{Z}}
\newcommand{\Oi}{\mathcal{O}_\infty}

\newcommand{\Zn}{\mathbb{Z}/n}
\newcommand{\cupprod}{\hat{\otimes}}
\newcommand{\Cstar}{\mathrm{C}^*}

\begin{document}
\title{Toward the classification of strongly self-absorbing $\mathrm{C}^*$-dynamical systems of compact groups}

\author{Masaki Izumi 
\thanks{Supported in part by JSPS KAKENHI Grant Number 25K00912}\\  Keiya Ohara\\
Graduate School of Science \\
Kyoto University \\
Sakyo-ku, Kyoto 606-8502, Japan} 
\date{}
\maketitle

\begin{abstract}
Strongly self-absorbing $\mathrm{C}^*$-algebras play a distinguished role in the classification of nuclear $\Cstar$-algebras. 
Their dynamical analogues were introduced and extensively studied by Szab\'o. 
In this paper, we propose a conjecture regarding the equivariant $KK$-theory of strongly self-absorbing $\mathrm{C}^*$-dynamical 
systems of compact groups in the equivariant bootstrap category; an affirmative answer to this conjecture 
would lead to classification results.  
We settle this conjecture for all finite EPPO (every element has a prime-power order) groups. 
In the course of our proof, we establish a K\"unneth-type formula for the equivariant $K$-theory of $\mathrm{C}^*$-algebras equipped 
with finite cyclic group actions---more precisely, for the cyclotomic part of the equivariant $K$-groups introduced by 
Meyer and Nadareishvili---under a certain unique divisibility assumption. 
\end{abstract}
\tableofcontents

\section*{Introduction}
\addcontentsline{toc}{section}{\numberline{}Introduction}

The notion of strongly self-absorbing $\Cstar$-algebras was introduced by Toms and Winter 
\cite{Toms and Winter} in order to axiomatize the class of $\Cstar$-algebras for which a McDuff-type argument applies.    
A unital separable $\Cstar$-algebra $D$ is said to be \textit{strongly self-absorbing} if the first-factor  embedding $D\to D\otimes D$, 
given by $x\mapsto x\otimes 1_D$, is approximately unitarily equivalent to an isomorphism. 
The following list exhausts all the known examples of strongly self-absorbing $\Cstar$-algebras: 
the Jiang-Su algebra $\mathcal{Z}$, the UHF-algebras of infinite type $M_{\mathfrak{P}^\infty}$, 
where $\mathfrak{P}$ is a set of prime numbers, 
the Cuntz algebras $\mathcal{O}_2$ and $\mathcal{O}_\infty$, and their tensor products.  

It is known that $\mathcal{O}_2$ and $\mathcal{O}_\infty$ play a central role in the Kirchberg--Phillips classification 
theory of purely infinite nuclear $\Cstar$-algebras (see \cite{Kirchberg and Phillips, Phillips}),  
and that $\mathcal{Z}$-absorption is the precise regularity condition required for the classification 
of nuclear $\Cstar$-algebras (see, for example, \cite{Gong Lin Niu, CETWW}). 
The formulation of Toms and Winter offers a unified framework in which these algebras can be treated simultaneously. 
As a consequence of the classification theory of nuclear $\Cstar$-algebras, the above list exhausts all isomorphism classes 
of strongly self-absorbing $\Cstar$-algebras provided that they satisfy the Universal Coefficient Theorem 
\cite{Rosenberg and Schochet 1997} (see \cite{Toms and Winter, Winter2011, Dadarlat Rordam} for the proof). 
In the present paper, we initiate a project to extend this classification of the strongly self-absorbing $\Cstar$-algebras to the equivariant setting involving compact group actions.

Let $G$ be a locally compact group. 
By a $G$-$\Cstar$-algebra, we mean a $\Cstar$-algebra equipped with a fixed $G$-action throughout this paper.  
A $G$-equivariant analogue of the notion of strongly self-absorbing $\Cstar$-algebras, namely 
\textit{strongly self-absorbing $G$-actions} or \textit{strongly self-absorbing $G$-$\Cstar$-algebras}, was introduced by Szab\'o \cite{Szabo}, and their properties have been extensively studied in the series of works \cite{Szabo, Szabo 2, Szabo 3}. 
In the recent striking work of Gabe and Szab\'o \cite{Gabe and Szabo} on the classification of isometrically shift-absorbing 
actions on Kirchberg algebras by equivariant $KK$-theory, an equivariant analogue of the $\mathcal{O}_\infty$-absorption 
plays an essential role. 
Strongly self-absorbing actions are of fundamental importance in the classification of group actions. 

The significance of the notion of strongly self-absorbing $\Cstar$-algebras (and their equivariant analogues) 
is not limited to the classification theory. 
Indeed, in a series of works \cite{Dadarlat and Pennig I, Dadarlat and Pennig II, Dadarlat and Pennig U}, Dadarlat and Pennig 
developed a far-reaching generalization of the Dixmier--Douady theory. 
They demonstrated that the classifying space $B\Aut(D\otimes \mathbb{K})$ of the automorphism group of the stabilization of a strongly self-absorbing $\Cstar$-algebra $D$ is an infinite loop space, and hence gives rise to a generalized cohomology theory. 
Moreover, they elucidated the relationship between these resulting infinite loop spaces and the ring spectra associated with $K$-theory.
An equivariant version of this theory involving compact group actions has not been fully developed yet, but a few examples coming from 
the infinite tensor product actions on the UHF algebras have been worked out in 
\cite{Evans and Pennig SU(n), Evans and Pennig T, Bianchi and Pennig}. 
An important lesson from these works is that the equivariant $K_0$-groups of these examples are localizations of 
the representation ring $R(G)$ of the acting group $G$.  
  
The $K_0$-group of a strongly self-absorbing $\Cstar$-algebra $D$ has a commutative ring structure, which is in fact a localization 
of the ring of integers $\mathbb{Z}$. 
This fact was implicitly used in \cite{Toms and Winter} to determine the structure of $D$, and explicitly stated 
in \cite{Dadarlat and Pennig I}. 
On the other hand, when $G$ is a compact group, we have $K^G_0(\C)=KK^G(\C,\C)\cong R(G)$, and the equivariant $KK$-group 
$KK^G(A,B)$ for any two $G$-$\Cstar$-algebras $A,B$ has a natural $R(G)$-module structure. 
Moreover, when $D$ is a strongly self-absorbing $G$-$\Cstar$-algebra, $K^G_0(D)$ has a commutative ring structure,  
as in the non-equivariant case, and the embedding map $\nu:\C \hookrightarrow D$ induces a ring homomorphism 
\[\nu_*:R(G)=K_0^G(\C)\to K_0^G(D).\]    
 
Let 
\[S_D=\{x\in R(G);\; \nu_*(x) \text{ is invertible in } K_0^G(D)\}.\]
Then $S_D$ is a multiplicative subset of $R(G)$. 
Thus for every $R(G)$-module $\mathcal{M}$, its localization $\mathcal{M}_{S_D}$ at $S_D$ makes sense. 
Although it is customary to exclude the case $S=R(G)$ to define the localization $\mathcal{M}_S$, we include the case 
$S=R(G)$ by putting $\mathcal{M}_S=\{0\}$ for convenience in stating our main conjecture.   

More generally, we denote by $\nu^B_*:KK_i^G(B,\C)\to KK_i^G(B,D)$ the map induced by $\nu$ for any $G$-$\Cstar$-algebra $B$. 

\begin{dfn}\label{loc condition}
We say that a strongly self-absorbing $G$-$\Cstar$-algebra $D$ satisfies the \textit{localization condition with respect to 
a $G$-$\Cstar$-algebra $B$} if   
there exist $R(G)$-module isomorphisms $h_i^B\colon KK_i^G(B,\C)_{S_D}\to KK_i^G(B,D)$ for $i=0,1$  
such that $h_i^B$ composed with the natural localization maps $KK^G_i(B,\mathbb{C})\to KK^G_i(B,\mathbb{C})_{S_D}$ are $\nu^B_\ast\colon KK_i^G(B,\C)\to KK_i^G(B,D)$:   
\[
\begin{tikzcd}[row sep=large, column sep=large]
	KK_i^G(B,\mathbb{C})
	\arrow[r]
	\arrow[dr, "\nu^B_*"']
	&
	KK_i^G(B,\mathbb{C})_{S_D}
	\arrow[d, "h^B_i"', "\cong"]
	\\
	{}
	&
	KK_i^G(B,D)
	\arrow[phantom, from=1-1, to=2-2,
	"\circlearrowright", pos=0.5,
	xshift=22pt, yshift=6pt]
\end{tikzcd}. 
\]
\end{dfn}
 
 Now we state our main conjecture. We say that a separable $G$-$\Cstar$-algebra $B$ is \textit{compact} if $B$ is a compact object of the $G$-equivariant Kasparov category $KK^G$, that is, for every countable family $\{C_n\}_n$ of separable $G$-$\Cstar$-algebras, $\bigoplus_n KK^G(B,C_n)$ is canonically isomorphic to $KK^G(B, \bigoplus_n C_n)$. 
 
\begin{mainconjecture}\label{conjecture}Let $G$ be a compact group and $D$ be a strongly self-absorbing $G$-$\Cstar$-algebra that belongs to the $G$-equivariant bootstrap class $\mathcal{B}_G$. 
Then $D$ satisfies the localization condition with respect to every separable compact $G$-$\Cstar$-algebra $B$. 
\end{mainconjecture}

\begin{rem}\label{rem of main conjecture} Several remarks are in order. 
\begin{itemize}
\item[(1)] For every saturated multiplicative subset $S\subset R(G)$, we construct in Proposition \ref{model action} 
a model action $(M^S,\mu^S)$ belonging to $\mathcal{B}_G$ such that $S_{M^S}=S$ and $(M^S,\mu^S)$ satisfies 
the localization condition with respect to every compact $B$. 
Furthermore, we can take $(M^S,\mu^S)$ so that $M^S$ is a Kirchberg algebra and $\mu^S$ is isometrically shift absorbing (Corollary \ref{model remark}).  
\item[(2)] For a closed subgroup $H\subset G$, $D$ satisfies the localization condition with respect to $C(G/H)$  
if and only if $K_0^H(D)\cong R(H)_{S_D}$ and $K_1^H(D)\cong \{0\}$, where $R(H)$ is regarded as a $R(G)$-module 
through the restriction homomorphism $R(G)\to R(H)$.  
\item[(3)] Let $\{B\}_{B\in \mathcal{B}_G^{0}}$ be a family of compact objects in $\mathcal{B}_G$ generating $\mathcal{B}_G$. 
Then Conjecture A is true for $G$ if and only if every strongly self-absorbing $D\in \mathcal{B}_G$ satisfies 
the localization condition with respect to every $B\in \mathcal{B}_G^{0}$. 
Indeed, we show in Proposition \ref{loc KK} that the latter condition implies that $D$ is $KK^G$-equivalent to the model 
$(M^{S_D},\mu^{S_D})$ with the equivalence sending $[1_D]\in K_0^G(D)$ to $[1_{M^{S_D}}]\in K_0^G(M^{S_D})$.    
For the same reason, we also see that if Conjecture A is true for $G$, the multiplicative set $S_D$ is a complete invariant 
of $D$ up to $KK^G$-equivalence preserving identity.  
If moreover we consider only isometrically shift absorbing actions on Kirchberg algebras, the set $S_D$ is a complete 
invariant up to conjugacy thanks to the Gabe--Szab\'o theorem.  
\item[(4)] When $G$ is a finite group, Meyer and Nadareishvili \cite[Corollary 3.3]{Meyer and Nadareishvili} showed, based on 
Arano and Kubota's work \cite{Arano and Kubota}, that the objects $C(G/H)$ for cyclic subgroups $H\subset G$ generate $\mathcal{B}_G$. 
Thus the proof of Conjecture A is reduced to showing $K_0^H(D)\cong R(H)_{S_D}$ and $K_1^H(D)=\{0\}$ for every cyclic 
subgroup $H\subset G$. 
\end{itemize}
\end{rem}

It is too ambitious a goal to verify Conjecture A in full generality at present.  
In this paper, we content ourselves with verifying it for a class of finite groups slightly larger than the class of $p$-groups, 
which still contains a few non-abelian simple groups. 
We believe that an affirmative answer in such a class of groups demonstrates the relevance of the conjecture. 

\begin{dfn}A finite group $G$ is called an \textit{EPPO-group} (short for \textit{every element has prime-power order}) if every element of $G$ has prime-power order. \end{dfn}

\begin{ex}\begin{itemize}
		\item[(1)]All $p$-groups are EPPO-groups. 
		\item[(2)]A dihedral group $D_n=\mathbb{Z}/n\rtimes\mathbb{Z}/2$ is an EPPO-group if $n=p^m$ for some prime number $p$. Indeed, if $p=2$, then $\abs{D_n}=2^{m+1}$. If $p$ is odd, then each element of $D_n$ has order $1$, $2$ or a power of $p$. 
		\item[(3)]The symmetric groups $S_3, S_4$ and the alternating groups $A_4, A_5, A_6$ are EPPO-groups. \end{itemize}
	\end{ex}

\begin{maintheorem}[Theorem \ref{EPPO case}]\label{intro strongly self-absorbing} Conjecture A is true for every EPPO-group. 
\end{maintheorem}%

Toms and Winter's computation of $K_\ast(D)$ in \cite{Toms and Winter} relied crucially on the K\"unneth formula established 
by Schochet \cite{Schochet}. 
To generalize their computation to the equivariant case, an equivariant version of the K\"unneth formula is necessary. 
However, it is well known that the K\"unneth formula does not hold in equivariant $K$-theory in a naive way, 
and several variants have been proposed in the literature (see \cite{Rosenberg and Schochet 1986, freeness, Rosenberg}). 
We formulate another variant suitable for our purposes following ideas of Meyer and Nadareishvili \cite{Meyer and Nadareishvili}. 

Let $\mathbb{Z}/n$ denote the cyclic group of order $n$. 
For a $\Zn$-$\Cstar$-algebra $A$, its equivariant $K$-theory $K^{\Zn}_\ast(A)$ is a module over the representation ring $R(\Zn)\cong\mathbb{Z}[t]/(t^n-1)$. 
Let $\Phi_n(t)$ denote the $n$-th cyclotomic polynomial, regarded as an element of $\mathbb{Z}[t]/(t^n-1)$. 
Then %
%
\[F^{\Zn}_\ast(A)=\left\{x\in K^{\Zn}_\ast(A)\mid \Phi_n\cdot x=0\right\}\]
is a module over $\mathbb{Z}[t]/(\Phi_n(t))\cong\mathbb{Z}[\zeta_n]$, where $\zeta_n$ is a primitive $n$-th root of unity. 
Using this module, Meyer and Nadareishvili formulated in \cite{Meyer and Nadareishvili} a version of the universal coefficient theorem for actions of finite groups. 
The following theorem may be thought of as a tensor analogue of that theorem in the case of finite cyclic groups. 

\begin{maintheorem}[Theorem \ref{Kunneth}]\label{intro Kunneth}Let $A$ and $B$ be $\mathbb{Z}/n$-$\Cstar$-algebras. Suppose that $K^H_\ast(A)$ is uniquely $n$-divisible for all subgroups $H\subset \Zn$ and $B$ belongs to the $\mathbb{Z}/n$-equivariant 
bootstrap class $\mathcal{B}_{\mathbb{Z}/n}$. 
Then there is a short exact sequence %
\[F^{\mathbb{Z}/n}_\ast(A)\otimes_{\mathbb{Z}[\zeta_{n}]} F^{\mathbb{Z}/n}_\ast(B)\rightarrowtail F^{\mathbb{Z}/n}_\ast(A\otimes B)\twoheadrightarrow {\Tor}^{\mathbb{Z}[\zeta_{n}]}_1(F^{\mathbb{Z}/n}_\ast(A),F^{\mathbb{Z}/n}_{\ast-1}(B)).\]
\end{maintheorem}%

This paper is organized as follows. 
In Section 1, we recall preliminary definitions and results that will be used throughout the paper. 
We begin with the definitions of equivariant $K$-theory, equivariant $KK$-theory, and the equivariant bootstrap class $\mathcal{B}_G$. 
We also review several facts from commutative ring theory and homological algebra that will be used in later sections. 

In Section 2, we provide an introduction to strongly self-absorbing actions. 
This section contains an explanation of the multiplicative structure of their equivariant $K$-theory. 

In Section 3, we construct model actions and discuss their applications mentioned in Remark \ref{rem of main conjecture}. 

In Section 4, we prove an equivariant K\"unneth formula (Theorem \ref{intro Kunneth}). 
Its proof follows the strategy of Schochet \cite{Schochet}: we construct a geometric realization of a projective resolution and derive the K\"unneth formula from the case where one of the two $\Cstar$-algebras has a projective homology.

In Section 5, we compute the ring structure of the $K$-theory of strongly self-absorbing actions. 
As a first step, we discuss the case where the acting group is a cyclic $p$-group for a prime $p$ as an application of 
Theorem \ref{intro Kunneth}. 
This part heavily relies on Meyer' work \cite{Meyer} on $\mathbb{Z}/p$-actions in the bootstrap category. 
We then extend the argument to the case of general $p$-groups. 
Finally, we deduce the result for EPPO-groups (Theorem \ref{intro strongly self-absorbing}) from the case of $p$-groups.

In Section 6, we attempt to verify our main conjecture for the Lie groups $U(1)$ and $SU(2)$ by using the K\"{u}nneth spectral sequence of Rosenberg and Schochet \cite{Rosenberg and Schochet 1986}. 
We show that the conjecture holds provided that $K_1^G(D)$ is trivial. 
In fact, we prove the conjecture after tensoring with the universal UHF algebra, which merely implies that $K_1^G(D)$ is a torsion group in general.

This paper is based on the second-named author's master's thesis conducted under the supervision of the first-named author.

\paragraph{Acknowledgements.}
The authors would like to thank Yosuke Kubota, Taro Sogabe, and Hiroro Kamikawa for stimulating discussions, 
and Takumi Nishihara for bringing Theorem \ref{overring} to our attention. 
The authors used Gemini and ChatGPT for English language editing and mathematical information retrieval during the preparation of this manuscript.


\section{Preliminaries}

\subsection{$K$-theory and $KK$-theory}
In this subsection, we recall basic definitions and properties of equivariant $K$-thoery and $KK$-theory that will be used in the hole of this paper. Standard references are \cite{Kasparov} and the book \cite{Blackadar}. Throughout this subsection, $G$ denotes a second countable locally compact group unless stated otherwise. In the latter part of this subsection, we will further assume that $G$ is finite. \par %
A \textit{$G$-$\Cstar$-algebra} means  a $\Cstar$-algebra $A$ with a point-norm continuous homomorphism $\alpha\colon G\to\text{Aut}(A)$.
$\Sigma A$ denotes the suspension of $A$, in other words, the minimal tensor product of $C_0(\R)$ (equipped with the trivial action of $G$) and $A$. For two $G$-$\Cstar$-algebras $(A,\alpha)$ and $(B,\beta)$, $A\otimes B$ means the minimal tensor product of $A$ and $B$ with the action $\alpha\otimes\beta$. \par %
For a $G$-$\Cstar$-algebra $A$, a \textit{$G$-Hilbert $A$-module} means a right Hilbert $A$-module $E$ with a norm-continuous action of $G$ satisfying %
\[g(e+\lambda e^\prime)=g(e)+\lambda g(e^\prime),\ g(e\cdot a)=g(e)\cdot g(a),\ g(\langle e,e^\prime\rangle)=\langle g(e),g(e^\prime)\rangle.\]%
for $e,e^\prime\in E,\ \lambda\in\C,\ a\in A$. A \textit{$G$-Hilbert space} means a $G$-Hilbert $\C$-module. \par%
$\tilde{A}$ and $\mathcal{M}(A)$ denotes the unitalization and the multiplier algebra of $A$, respectively. For a Hilbert module $E$, we write the $\Cstar$-algebra of all bounded adjointable operators as $\mathcal{L}(E)$. Its norm-closed ideal generated by finite rank operators is denoted by $\mathcal{K}(E)$.\par%
We begin by recalling the definition of equivariant $K$-theory.

\begin{dfn}[cf. {\cite[Definitions 11.5.1, 11.5.3 and 11.9.3]{Blackadar}}]Let $G$ be a compact group and $A$ be a $G$-$\Cstar$-algebra. \begin{itemize}
\item[(1)]When $A$ is unital, let $V,W$ be two finite dimensional $G$-Hilbert spaces. For $G$-invariant projections $p\in B(V)\otimes A$ and $q\in B(W)\otimes A$, they are \textit{equivalent} if there is a $G$-invariant partial isometry $v\in B(V\oplus W)\otimes A$ such that $vv^\ast=p$ and $v^\ast v=q$ hold. The set of equivalence classes of such projections forms an abelian semigroup by $[p]+[q]=[\text{diag}(p,q)]$. Its Grothendieck group is denoted by $K^G_0(A)$.
\item[(2)]When $A$ is nonunital, let $\pi\colon \tilde{A}\to\C$ be the natural projection. $\pi$ induces a map $\pi_\ast\colon K^G_0(\tilde{A})\to K^G_0(\C)$ given by $\pi_\ast([p])=[\pi(p)]$. The group $K^G_0(A)$ is defined as the kernel of $\pi_\ast$. 
\item[(3)]For $i\in\mathbb{Z}_{>0}$, the group $K^G_i(A)$ is defined as $K^G_0(\Sigma^i A)$. \end{itemize}\end{dfn}

There is a bilinear pairing %
\[K^G_i(A)\times K^G_j(B)\ni (x,y)\mapsto x\cupprod y\in K^G_{i+j}(A\otimes B)\]%
called the \textit{cup product}. If $p\in B(V)\otimes A$ and $q\in B(W)\otimes B$ are projections, then the cup product of $[p]\in K^G_0(A)$ and $[q]\in K^G_0(B)$ is $[p]\cupprod [q]=[p\otimes q]\in K^G_0(A\otimes B)$. \par %
When $G$ is compact, $K^G_0(\C)$ is the representation ring $R(G)$. Thus, $K^G_i(A)$ is a left $R(G)$-module by the cup product $R(G)\times K^G_i(A)\to K^G_i(A)$. \par%
We now pass to the equivariant $KK$-thoery. 

\begin{dfn}[cf. {\cite[Definitions 2.2 and 2.3]{Kasparov}}]Let $A,B$ be two separable $G$-$\Cstar$-algebras. A \textit{Kasparov $A$-$B$ bimodule} is a triple $(E,\phi, F)$, where $E$ is a countably generated $\mathbb{Z}/2$-graded $G$-Hilbert $B$-module, $\phi$ is a $G$-equivariant graded $\ast$-homomorphism $A\to \mathcal{L}(E)$ and $F\in\mathcal{L}(E)$ is an odd element such that %
\[\phi(a)F-F\phi(a),\ \phi(a)(F^2-1),\ \phi(a)(F-F^\ast)\ \text{and } \phi(a)(g\cdot F-F)\]%
are in $\mathcal{K}(E)$ for any $a\in A,\ g\in G$. We say that two Kasparov $A$-$B$ bimodules $\mathcal{E},\mathcal{E}^\prime$ are equivalent if there are Kasparov $A$-$B$ bimodules $\mathcal{E}=\mathcal{E}_0, \mathcal{E}_1, \dots ,\mathcal{E}_n=\mathcal{E}^\prime$ such that $\mathcal{E}_i$ and $\mathcal{E}_{i+1}$ are unitarily equivalent or homotopy equivalent for each $i=0,\dots ,n-1$. $KK^G(A,B)$, also denoted by $KK^G_0(A,B)$, is the set of equivalence classes of Kasparov $A$-$B$ bimodules. It forms an abelian group by $[(E,\phi ,F)]+[(E^\prime, \phi^\prime, F^\prime)]=[(E\oplus E^\prime, \phi\oplus\phi^\prime, \text{diag}(F,F^\prime))]$. We also write $KK^G(A,\Sigma B)$ as $KK^G_1(A,B)$. \end{dfn}

 The direct sums $K^G_0(A)\oplus K^G_1(A)$ and $KK^G_0(A,B)\oplus KK^G_1(A,B)$ are denoted by $K^G_\ast(A)$ and $KK^G_\ast(A,B)$, respectively. There is a natural isomorphism $K^G_\ast(A)\cong KK^G_\ast(\C ,A)$. If $p\in B(V)\otimes A$ is a projection, then this isomorphism maps $[p]\in K^G_0(A)$ to $[(E,\phi ,0)]\in KK^G_0(\C ,A)$, where $E=p(V\otimes A)$ and $\phi\colon \C\to\mathcal{L}(E)$ is the unital inclusion. 

\begin{prp}[cf. {\cite[Theorems 2.11 and 2.14]{Kasparov}}] There is an associative and functorial bilinear pairing %
\[KK^G(A,B)\times KK^G(B,C)\ni (x,y)\mapsto x\otimes y\in KK^G(A,C),\]%
called the \textit{Kasparov product}. \end{prp}

A $G$-equivariant $\ast$-homomorphism $f\colon A\to B$ gives $[(B,f,0)]\in KK^G(A,B)$. We also write it as $f$. If $A=B$ and $f=\id_A$, then $x\otimes \id_A=x$ and $\id_A\otimes y=y$ for any $x\in KK^G(C,A),\ y\in KK^G(A,D)$. For another $\ast$-homomorphism $g\colon B\to C$, the Kasparov product $f\otimes g\in KK^G(A,C)$ equals $g\circ f$. Given $x\in KK^G(A,B)$, the following maps %
\begin{align*}& KK^G(C,A)\ni y\mapsto y\otimes x\in KK^G(C,B)\quad \text{and} \\
& KK^G(B,C)\ni y\mapsto x\otimes y\in KK^G(A,C)\end{align*}%
are denoted by $x_\ast$ and $x^\ast$, respectively.

\begin{dfn}Let $A,B$ be two separable $G$-$\Cstar$-algebras. An element $x\in KK^G(A,B)$ is called \textit{invertible} if there exists $y\in KK^G(B,A)$ with $x\otimes y=\id_A$ and $y\otimes x=\id_B$. In this case $A$ and $B$ are said to be \textit{$KK^G$-equivalent}. \end{dfn}

\begin{ex}\begin{itemize}
\item[(1)]Two $G$-$\Cstar$-algebras $A,B$ are \textit{Morita equivalent} if there exists a full $G$-Hilbert $B$-module with an isomorphism $\phi\colon A\to\mathcal{K}(E)$. The Kasparov $A$-$B$ bimodule $(E,\phi ,0)$ gives an invertible element of $KK^G(A,B)$. For details, see \cite[Theorem 2.18]{Kasparov}. 
\item[(2)]There is a $KK^G$-equivalence between $\C$ and $C_0(\R^2)$ with the trivial $G$-action. Equivalently, $\Sigma^2 A$ is $KK^G$-equivalent to $A$ for any $A\in KK^G$. This is called the \textit{Bott periodicity}. The proof is found in \cite[Corollary 19.2.3]{Blackadar}, for example. 
\end{itemize}\end{ex}

We next recall the triangulated categorical aspects of $KK$-theory. 

\begin{dfn}[cf. {\cite[1.1]{Meyer and Nest}}]The \textit{$G$-equivariant Kasparov category $KK^G$} is the category whose objects are all separable 
$G$-$\Cstar$-algebras and whose set of morphisms $A\to B$ is $KK^G(A,B)$. \end{dfn}

Meyer and Nest showed in \cite{Meyer and Nest} that $KK^G$ is (equivalent to) a triangulated category; an additive category $\mathcal{T}$ 
with an autofunctor $\Sigma$ and a collection of diagrams of the form %
\begin{align}\label{ex trgl}\Sigma Z\to X\to Y\to Z\quad\quad (X,Y,Z\in\mathcal{T}), \end{align}%
called \textit{exact triangles}, satisfying some properties. An exact triangle in $KK^G$ is a diagram of the form as above having a commutative diagram %
\[\begin{tikzcd} \Sigma Z \arrow[d,"\Sigma z",] \arrow[r] & X \arrow[d,"x"] \arrow[r] & Y \arrow[d,"y"] \arrow[r] & Z \arrow[d,"z"] \\ 
	\Sigma B \arrow[r] & C_f \arrow[r] & A \arrow[r,"f"] & B,\end{tikzcd}\]%
where the bottom row is the mapping cone sequence of a $G$-equivariant $\ast$-homomorphism $f$ and vertical morphisms $x\in KK^G(X,C_f),\ y\in KK^G(Y,A)$ and $z\in KK^G(Z,B)$ are all invertible. \par%
Let $\mathcal{T}$ be a triangulated category, $\text{\bf{Ab}}$ be the category of abelian groups and $F\colon\mathcal{T}\to\text{\bf{Ab}}$ be an additive functor. Then $F$ is said to be \textit{homological} if for any exact triangle (\ref{ex trgl}), the diagram $F(X)\to F(Y)\to F(Z)$ is exact. In this case we have a long exact sequence %
\[\cdots\to F_{n+1}(X)\to F_{n+1}(Y)\to F_{n+1}(Z)\to F_n(X)\to F_n(Y)\to F_n(Z)\to\cdots ,\]%
here $F_n(X)\coloneqq F(\Sigma^nX)$. A \textit{cohomological} functor is defined dualy. 

\begin{prp}[cf. {\cite[1.1 Theorem]{Cuntz and Skandalis}}]Let $D$ be an object in $KK^G$. Then for any mapping cone sequence $\Sigma B\to C_f\to A\to B$, we have two exact sequences %
\[KK^G(D,C_f)\to KK^G(D,A)\to KK^G(D,B),\]
\[KK^G(B,D)\to KK^G(A,D)\to KK^G(C_f,D).\]
Thus $KK^G(D, \cdot)\colon KK^G\to\text{\bf{Ab}}$ is homological and $KK^G(\cdot ,D)\colon KK^G\to\text{\bf{Ab}}$ is cohomological.  \end{prp}

\begin{dfn}Let $H$ be a closed subgroup of $G$. \begin{itemize}
\item[(1)] \textit{The restriction functor $\Res^H_G\colon KK^G\to KK^H$} is the functor defined by restricting the actions of $G$ to $H$.
\item[(2)]\textit{The induction functor $\Ind^G_H\colon KK^H\to KK^G$} is defined as follows. For any object $A\in KK^H$, we define the object in $KK^G$ %
\[\Ind^G_H A\coloneqq\qty{\widetilde{a}\in C_b(G,A)\middle| \begin{split}&h(\widetilde{a}(g))=\widetilde{a}(gh^{-1})\ \text{for all $g\in G,\ h\in H$} \\ &\text{and $\qty(gH\to \|\widetilde{a}(g)\|)\in C_0(G/H)$}\end{split}}\] %
with the $G$-action given by left translation. For any morphism $f=[(E,\phi ,F)]\in KK^H(A,B)$, we define the morphism \[\Ind^G_H f=[(\Ind^G_H E,\widetilde{\phi}, \widetilde{F})]\in KK^G(\Ind^G_H A,\Ind^G_H B)\] by %
\[\Ind^G_H E\coloneqq\qty{\widetilde{e}\in C_b(G,E)\middle| \begin{split}&h(\widetilde{e}(g))=\widetilde{e}(gh^{-1})\ \text{for all $g\in G,\ h\in H$} \\ &\text{and $\qty(gH\to \|\widetilde{e}(g)\|)\in C_0(G/H)$}\end{split}},\]
\[\qty(\widetilde{\phi}(a)e)(g)\coloneqq \phi(a(g))e(g)\quad (a\in\Ind^G_H A,\ e\in\Ind^G_H E,\ g\in G),\]
\[(\widetilde{F}e)(g)\coloneqq Fe(g)\quad (e\in\Ind^G_H E,\ g\in G).\]
\end{itemize}\end{dfn}

Let $A,B$ be objects in $KK^G$ and $H$ be a closed subgroup of $G$. We often write $KK^H_\ast(\Res^H_G A,\Res^H_G B)$ and $K^H_\ast(\Res^H_G A)$ simply as $KK^H_\ast(A,B)$ and $K^H_\ast(A)$, respectively. If $K$ is a subgroup of $H$, then the restriction functor and the induction functor satisfy %
\[\Res^K_H\circ \Res^H_G=\Res^K_G,\]
\[\Ind^G_H\circ \Ind^H_K=\Ind^G_K.\]

\begin{lem}\label{IndRes}Let $H$ be a closed subgroup of $G$ and $A$ be an object in $KK^G$. Then There is a $G$-equivariant $\ast$-isomorphism %
\[\Ind^G_H\Res^H_G A\cong C_0(G/H)\otimes A.\]\end{lem}
\begin{proof}The map %
\[\Ind^G_H\Res^H_G A\ni a\mapsto\qty(gH\mapsto g(a(g)))\in C_0(G/H)\otimes A\]%
gives the desired $\ast$-isomorphism. \end{proof}

\begin{prp}\label{reciprocity}Let $H$ be a closed subgroup of $G$, $A$ be an object in $KK^G$ and $B$ be an object in $KK^H$. \begin{itemize}
\item[(1)](cf. \cite[(19)]{Meyer and Nest}) If $H$ is cocompact, then there is a natural isomorphism %
\[KK^G(A,\Ind^G_H B)\cong KK^H(\Res^H_G A,B).\]%
\item[(2)](cf. \cite[(20)]{Meyer and Nest}) If $H$ is open, then there is a natural isomorphism %
\[KK^G(\Ind^G_H B,A)\cong KK^H(B,\Res^H_G A).\]
\end{itemize}\end{prp}
\begin{proof}The proof can be found in \cite[p.230-231]{Meyer and Nest}, but we briefly describe the homomorphisms that give the isomorphism of (2). Define the map $\iota_B\colon B\to \Res^H_G\ind^G_H B$ by %
\[\iota_B(b)(g)=\begin{dcases*} g^{-1}(b) & if $g\in H$, \\ 0 & otherwise. \end{dcases*}\quad (b\in B,\ g\in G)\]%
Then we obtain the composition map %
\begin{align}\label{adjunction}KK^G(\Ind^G_H B,A)\overset{\Res^H_G}\to KK^H(\Res^H_G\Ind^G_H B,\Res^H_G A)\overset{\iota_B^\ast}\to KK^H(B,\Res^H_G A).\end{align}%
By Lemma \ref{IndRes}, $\Ind^G_H\Res^H_G A$ is isomorphic to $C_0(G/H)\otimes A$. Thus, the composition of the diagonal inclusion $C_0(G/H)\hookrightarrow\mathcal{K}(\ell_2(G/H))$ with the Morita equivalence $\mathcal{K}(\ell_2(G/H))\sim\C$ gives an element $\pi_A\in KK^G(\Ind^G_H\Res^H_G A, A)$. Now the map %
\[KK^H(B,\Res^H_G A)\overset{\Ind^G_H}\to KK^H(\Ind^G_H B, \Ind^G_H\Res^H_G A)\overset{(\pi_A)_\ast}\to KK^G(\ind^G_H B,A)\]%
is the inverse of (\ref{adjunction}). \end{proof}

\begin{rem}\label{counit}Proposition \ref{reciprocity} (1) states that the functor $\Res^H_G$ is the left adjoint of $\Ind^G_H$. Hence, there is an element $\ve_B\in KK^G(\Res^H_G\Ind^G_H B,B)$ such that the isomorphism of Proposition \ref{reciprocity} (1) is equal to %
\[KK^G(A,\Ind^G_H B)\overset{\Res^H_G}{\to}KK^H(\Res^H_G A,\Res^H_G\Ind^G_H B)\overset{(\ve_B)_\ast}{\to}KK^H(\Res^H_G A,B)\]%
(see \cite[Theorem 2.2.5]{basic category}).\end{rem}

Let $\mathcal{T}$ be a triangulated category with countable direct sums and $\mathcal{G}$ be a class of objects in $\mathcal{T}$. The \textit{($\aleph_0$-)localising subcategory} generated by $\mathcal{G}$ is the minimal triangulated subcategory of $\mathcal{T}$ that contains $\mathcal{G}$ and closed under countable direct sums. If $\mathcal{T}=KK^G$, then it is the minimal full subcategory $\mathcal{B}$ containing $\mathcal{G}$ with the following properties: \begin{itemize}
\item[(1)]For any $A,B\in \mathcal{B}$ and any $G$-equivariant $\ast$-homomorphism $f\colon A\to B$, its mapping cone $C_f$ is in $\mathcal{B}$. 
\item[(2)]For any $A\in\mathcal{B}$, its suspention $\Sigma A$ is in $\mathcal{B}$. 
\item[(3)]For any countably many $A_n\in\mathcal{B}\ (n\in\mathbb{N})$, their direct sum $\oplus_n A_n$ is in $\mathcal{B}$. 
\item[(4)]If $A\in\mathcal{B}$ and $B\in KK^G$ are $KK^G$-equivalent, then $B\in\mathcal{B}$. 
\end{itemize}

\begin{dfn}[cf. {\cite[Definition 3.9]{Lefschetz}}] Let $G$ be a compact group. The \textit{$G$-equivariant bootstrap class} $\mathcal{B}_G$ is the localising subcategory of $KK^G$ generated by the collection of $G$-$\Cstar$-algebras of the form $\Ind^G_H M_n$, where $H$ is a closed subgroup of $G$ and $M_n$ is a matrix algebra with some action of $H$. \end{dfn}

The invertibility of a morphism between two objects in the bootstrap class can be easily checked. 

\begin{prp}\label{bootstrap}Let $G$ be a compact group, $A,B$ be objects in the $G$-equivariant bootstrap class $\mathcal{B}_G$ and $\mathcal{B}_G^0$ be a subclass of $\mathcal{B}_G$ generating $\mathcal{B}_G$. Then an element $f\in KK^G(A,B)$ is invertible if and only if $f_\ast\colon KK^G_\ast(C,A)\to KK^G_\ast(C,B)$ is invertible for each object $C$ in $\mathcal{B}_G^0$. In particular, when $G$ is finite, $f$ is invertible if and only if $f_\ast\colon K^H_\ast(A)\to K^H_\ast(B)$ is invertible for each cyclic subgroup $H\subset G$. \end{prp}
\begin{proof}Assume that $f_\ast\colon KK^H_\ast(C,A)\to KK^H_\ast(C,B)$ is invertible for each object $C\in\mathcal{B}_G^0$. Let $\mathcal{C}$ be the full subcategory of objects $C$ in $KK^G$ such that $f_\ast\colon KK^G_\ast(C,A)\to KK^G_\ast(C,B)$ is invertible. Clearly $\mathcal{C}$ includes $\mathcal{B}_G^0$. It is easy to see that $\mathcal{C}$ is a localising subcategory of $KK^G$. Hence, $\mathcal{C}$ includes $\mathcal{B}_G$. Since $A,B\in\mathcal{B}_G$, we have two isomorphisms %
\begin{align}&KK^G(B,A)\overset{\cong}{\to} KK^G(B,B),\quad x\mapsto x\otimes f, \\
\label{second isom}&KK^G(A,A)\overset{\cong}{\to} KK^G(A,B),\quad y\mapsto y\otimes f. \end{align}%
The first isomorphism implies that there exits an element $x\in KK^G(B,A)$ with $x\otimes f=\id_B$. For this $x$, %
\[f_\ast\circ x_\ast=(x\otimes f)_\ast=(\id_B)_\ast\colon KK^G(A,B)\to KK^G(A,B). \]%
The map $f_\ast\colon KK^G(A,A)\to KK^G(A,B)$ is isomorphic by (\ref{second isom}). Hence, $x_\ast\colon KK^G(A,B)\to KK^G(A,A)$ is the inverse of $f_\ast$ and %
\[f\otimes x=x_\ast\circ f_\ast(\id_A)=\id_A. \]%
Thus $f\in KK^G(A,B)$ is invertible with the inverse $x\in KK^G(B,A)$. Conversely, if $f$ is invertible, then clearly $f$ gives an isomorphism $KK^H_\ast(C,A)\cong KK^H_\ast(C,B)$ for each object $C\in\mathcal{B}_G^0$. 

When $G$ is finite, by \cite[Corollary 3.3]{Meyer and Nadareishvili}, $\mathcal{B}_G$ is generated by the objects $C(G/H)$ for cyclic subgroups $H\subset G$. Since $KK^G_\ast(C(G/H),A)\cong K^H_\ast(A)$ by Proposition \ref{reciprocity} (2), the above argument implies that $f$ is invertible if $f$ induces an isomorphism $K^H_\ast(A)\cong K^H_\ast(B)$ for each cyclic subgroup $H$. \end{proof}

Henceforth, $G$ is assumed to be finite. We finally describe the Mackey module structure of $(K^H_\ast(A))_{H\subset G}$. 

\begin{dfn}\label{Mackey} Let $G$ be a finite group and $A$ be an object in $KK^G$. \begin{itemize}
\item[(1)]For any subgroups $K\subset H\subset G$, \textit{the restriction map $\res^K_H$} is the restriction homomorphism $\Res^K_H\colon K^H_\ast(A)\to K^K_\ast(A)$. 
\item[(2)]For any element $g\in G$ and any subgroup $H\subset G$, \textit{the conjugation map $\con_{g,H}$} is the homomorphism $K^H_\ast(A)\to K^{gHg^{-1}}_\ast(A)$ defined by regarding the $H$-action as the $gHg^{-1}$-action via the isomorphism $H\to gHg^{-1},\ h\mapsto ghg^{-1}$.
\item[(3)]For any subgroups $K\subset H\subset G$, \textit{the induction map $I^H_K\colon K^K_\ast(A)\to K^H_\ast(A)$} is defined as follows. Let $\iota\colon\C\to C(H/K)=\Ind^H_K\C$ be the unital inclusion and $\pi_A\colon\Ind^H_K\Res^K_H(\Res^H_G A)\to\Res^G_H A$ be the map defined in the same way as in the proof of Proposition \ref{reciprocity} (2). Then $I^H_K$ is the composition %
\begin{align*}&K^K_\ast(A)\overset{\Ind^H_K}\to KK^H_\ast(\Ind^H_K\C ,\Ind^H_K\Res^K_H(\Res^H_G A)) \\ &\overset{(\pi_A)_\ast}\to KK^H_\ast(\Ind^H_K\C ,A)\overset{\iota^\ast}\to K^H_\ast(A).\end{align*}
\end{itemize}\end{dfn}

\begin{prp}[cf. {\cite[Proposition 4.5]{Dell'ambrogio}}]\label{Dell'ambrogio} For $g\in G$ and $H\subset G$, we write the subgroups $gHg^{-1}$ and $g^{-1}Hg$ as ${}^gH$ and $H^g$, respectively. Then the maps defined in Definition \ref{Mackey} satisfy the following relations: %
\begin{align*}&\res^K_H\circ \res^H_G=\res^K_G &(K\subset H\subset G), \\ 
&I^G_K\circ I^H_K=I^G_K &(K\subset H\subset G), \\
&\con_{g^\prime ,{}^gH}\circ\con_{g,H}=\con_{g^\prime g,H} &(g^\prime, g\in G,\ H\subset G), \\
&\con_{g,H}\circ I^H_K=I^{{}^gH}_{{}^gK}\circ\con_{g,K} &(g\in G,\ K\subset H\subset G), \\
&\con_{g,K}\circ \res^K_H=\res^{{}^gK}_{{}^gH}\circ\con_{g,H} &(g\in G,\ K\subset H\subset G), \\
&I^H_H=\res^H_H=\con_{h,H}=\id_{K^H_\ast(A)} &(h\in H\subset G), \\
&\res^L_H\circ I^H_K=\sum_{x\in L\backslash H/K}I^L_{L\cap {}^xK}\circ\con_{x,L^x\cap K}\circ \res^{L^x\cap K}_K &(L,K\subset H\subset G).\end{align*}%
Furthermore, for all $A\in KK^G$ and $K\subset H\subset G$, the following formula holds: %
\begin{align*}&x\cdot I^H_K(y)=I^H_K(\res^K_H(x)\cdot y) &(x\in R(H),\ y\in K^K_\ast(A)), \\
& I^H_K(x)\cdot y=I^H_K(x\cdot\res^K_H(y)) &(x\in R(K),\ y\in K^H_\ast(A)).\end{align*}\end{prp}

\subsection{Some remarks on ring theory and homological algebra}

The purpose of this subsection is to collect some elementary facts of ring theory and homological algebra. They are used in the analysis of equivariant $K$-theory and in the proof of the equivariant K\"unneth formula. Most of the material in this subsection is standard and can be found in \cite{Rotman} and \cite{Neukirch}. 

In what follows, a commutative ring is simply called a ring. 

Let $R$ be a ring. For a multiplicative subset $S\subset R$ and a left $R$-module $M$, $M_S$ denotes the localization of $M$ by $S$. More precisely, $M_S$ is the quotient of $M\times S$ by the equivalence relation $\sim$, where $(m,x)\sim (n,y)$ if $s(ym-xn)=0$ for some $s\in S$. The element of $M_S$ represented by $(m,x)\in M\times S$ is denote by $m/x$. For a prime ideal $\mathfrak{p}\triangleleft R$, we write $M_{R\setminus\mathfrak{p}}$ as $M_\mathfrak{p}$. When $R$ is an integral domain, $\text{Frac}(R)$ denotes the field of fractions $R_{R\setminus \{0\}}$. For details on the localization, see \cite[4.7]{Rotman}. 

A multiplicative subset $S\subset R$ is said to be \textit{saturated} if for $x,y\in R$ with $xy\in S$, we have $x,y\in S$. The saturation of a multiplicative subset $S$ is the set %
\[\overline{S}\coloneqq \{x\in R\mid xy\in S\ \text{for some $y\in R$}\}.\]%
It is easy to see that $\overline{S}$ is saturated and the two localizations $R_S$, $R_{\overline{S}}$ are canonically isomorphic. 

\begin{rem}\label{localization}The localization $R_S$ has the universality in the sense as follows. Let $i\colon R\to R_S$ be the natural map $i(r)=r/1$. If $f\colon R\to R^\prime$ is a ring homomorphism such that $f(s)\in R^\prime$ is invertible for all $s\in S$, then there exists a unique homomorphism $\widetilde{f}\colon R_S\to R^\prime$ with $\widetilde{f}\circ i=f$. More precisely, $\widetilde{f}$ is defined by %
\[\widetilde{f}\qty(\frac{x}{y})=f(x)\cdot f(y)^{-1}\quad (x\in R,\ y\in S).\]%
Thus, $\widetilde{f}$ is surjective if and only if for any $a\in A$, there are $x\in R$ and $y\in S$ with $a=f(x)\cdot f(y)^{-1}$. The above formula also implies that $\widetilde{f}$ is injective if $f$ is injective. \end{rem}

\begin{dfn}A ring $R$ is said to be \textit{hereditary} if every ideal of $R$ is projective as an $R$-module. A hereditary domain is called a \textit{Dedekind domain}. \end{dfn}

If $R$ is a principal ideal domain (PID), then every ideal of $R$ is a free $R$-module. Hence, PIDs are Dedekind domains. 

\begin{rem}\label{hereditary}If $R$ is hereditary, then every submodule of a projective module over $R$ is also projective (see \cite[Theorem 4.19]{Rotman}, for example). In particular, every module over a hereditary ring has a projective resolution of length at most one. \end{rem}
\begin{rem}\label{Dedekind}Let $R$ be a Dedekind domain. Then $R_S$ is also a Dedekind domain for any multiplicative subset $S\subset R$, and $R_\mathfrak{p}$ is a PID for any prime ideal $\mathfrak{p}\triangleleft R$. The proof is found in \cite[(11.4) Proposition and (11.5) Proposition]{Neukirch}, for example.\end{rem}

The following statement will be used in the proof of Theorem \ref{F is a localization}.

\begin{lem}\label{Tor(M,M)}Let $R$ be a Dedekind domain and $M$ be a left $R$-module. If $M$ satisfies $\text{Tor}^R_1(M,M)=0$, then $M$ is flat. \end{lem}
\begin{proof}$M$ is flat if and only if $M_\mathfrak{p}$ is a flat $R_\mathfrak{p}$-module for every maximal ideal $\mathfrak{p}\triangleleft R$ (see \cite[Corollary 7.18]{Rotman}, for example). Thus, it suffices to show the flatness of $M_\mathfrak{p}$. Observe that %
\[0=\qty(\text{Tor}^R_1(M,M))_{\mathfrak{p}}\cong\text{Tor}^{R_\mathfrak{p}}_1(M_\mathfrak{p},M_\mathfrak{p}).\]%
Take any nonzero element $x\in M_\mathfrak{p}$. We will show that $x$ is not a torsion element. By applying the torsion long exact sequence for the short exact sequence %
\[R_\mathfrak{p}x\rightarrowtail M_\mathfrak{p}\twoheadrightarrow M_\mathfrak{p}/R_\mathfrak{p}x,\]%
we get an exact sequence %
\[\text{Tor}^{R_\mathfrak{p}}_2(M_\mathfrak{p}/R_\mathfrak{p}x,M_\mathfrak{p})\to\text{Tor}^{R_\mathfrak{p}}_1(R_\mathfrak{p}x,M_\mathfrak{p})\to\text{Tor}^{R_\mathfrak{p}}_1(M_\mathfrak{p},M_\mathfrak{p}).\]%
Since $R_\mathfrak{p}$ is a PID by Remark \ref{Dedekind}, the leftmost module in the above sequence vanishes. We already know that the rightmost one vanishes. Thus, the middle $\text{Tor}^{R_\mathfrak{p}}_1(R_\mathfrak{p}x,M_\mathfrak{p})$ also vanishes. By repeating a similar argument, we get $\text{Tor}^{R_\mathfrak{p}}_1(R_\mathfrak{p}x,R_\mathfrak{p}x)=0$. Now remember that $R_\mathfrak{p}x$ is a module over a PID and generated by one element $x$. Hence there exists $a\in R_{\mathfrak{p}}$ such that $R_\mathfrak{p}x$ is isomorphic to $R_\mathfrak{p}/(a)$. Suppose $a\neq 0$ and consider the following part of the torsion exact sequence %
\[\text{Tor}^{R_\mathfrak{p}}_1(R_\mathfrak{p}/(a),R_\mathfrak{p}/(a))\to R_\mathfrak{p}\otimes_{R_\mathfrak{p}}R_\mathfrak{p}/(a)\overset{m_a\otimes \id}\to R_\mathfrak{p}\otimes_{R_\mathfrak{p}}R_\mathfrak{p}/(a),\]%
where $m_a\colon R_\mathfrak{p}\to R_\mathfrak{p}$ denotes the multiplication by $a$. The map $m_a\otimes \id$ is injective because $\text{Tor}^{R_\mathfrak{p}}_1(R_\mathfrak{p}/(a),R_\mathfrak{p}/(a))=0$. This map is identified with the multiplication by $a$ on $R_\mathfrak{p}/(a)$ and it is injective. However, it occurs only when $R_\mathfrak{p}/(a)=0$, and it contradicts $R_\mathfrak{p}x\neq 0$. Thus, we have $a=0$ and $R_\mathfrak{p}x\cong R_\mathfrak{p}$, that is, $x$ is not a torsion element of $M_{\mathfrak{p}}$. Hence, $M_{\mathfrak{p}}$ is torsion-free. Since a torsion-free module over a PID is flat (see \cite[Corollary 3.51]{Rotman}), $M_{\mathfrak{p}}$ is a flat $R_\mathfrak{p}$-module. \end{proof}

To state the theorem that is needed in the fourth section, we now turn to the definition of algebraic integers.

\begin{dfn}Let $K$ be a number field, in other words, an extension field of $\mathbb{Q}$ such that $K$ is finite dimensional as a vector space over $\mathbb{Q}$. An element $a\in K$ is called an \textit{algebraic integer} of $K$ if there exists a monic $f\in\mathbb{Z}[x]$ such that $f(a)=0$ in $K$. The set of algebraic integers of $K$ forms a subring of $K$, called \textit{the ring of integers} of $K$. \end{dfn}

\begin{rem}\label{Z[zeta] is Dedekind}If $K$ is a cyclotomic field $\mathbb{Q}[\zeta_n]$, where $\zeta_n$ is a primitive $n$-th root of unity, then the ring of integers of $\mathbb{Q}[\zeta_n]$ is $\mathbb{Z}[\zeta_n]$. It is known that the ring of integers of a number field is always a Dedekind domain. See \cite[(3.1) Theorem]{Neukirch}, for example. In particular, $\mathbb{Z}[\zeta_n]$ is a Dedekind ring. \end{rem}

\begin{rem}\label{Frac=K}Let $K$ be a number field and $\mathcal{O}_K$ be the ring of integers of $K$. Then $\text{Frac}(\mathcal{O}_K)$ equals $K$. Indeed, every element $a\in K$ has a monic $f\in\mathbb{Q}[x]$ with $f(a)=0$ because all finite extension of $\mathbb{Q}$ are algebraic. Write %
\[f(x)=x^n+\frac{p_1}{q_1}x^{n-1}+\dots +\frac{p_n}{q_n}\]%
for some $p_1, \dots ,p_n\in\mathbb{Z},\ q_1, \dots ,q_n\in\mathbb{Z}\setminus\{0\}$. Put $d=q_1\cdots q_n$. Then we can easily check that $da$ is a root of the monic %
\[g(x)\coloneqq x^n+\frac{p_1}{q_1}dx^{n-1}+\dots +\frac{p_n}{q_n}d^n\in\mathbb{Z}[x].\]%
Thus, we can write $a=da/d$ with $da,d\in\mathcal{O}_K$. \end{rem}

\begin{dfn}Let $R$ be a domain. An \textit{overring} of $R$ is a subring of $\text{Frac}(R)$ that contains $R$ as a subring. \end{dfn}

The following theorem plays a key role in the discussion of Section 4. Let $R$ be a Dedekind domain. In general, the set of all nonzero, finitely generated $R$-submodules of $\text{Frac}(R)$ modulo ones generated by one element forms a group, called the \textit{class group}. 

\begin{thm}[cf. {\cite[Theorem 2]{Davis}}]\label{overring}Let $R$ be a Dedekind domain. Then the class group of $R$ is torsion if and only if for any overring $\mathcal{A}$ of $R$, there exists a multiplicative subset $S$ of $R$ such that $\mathcal{A}$ equals $R_S$. \end{thm}

It is known that if $K$ is a number field, then the class group of the ring of integers $\mathcal{O}_K$ is finite (see \cite[(6.3) Theorem]{Neukirch}). Thus, the above theorem implies that all overrings of $\mathcal{O}_K$ are localizations of $\mathcal{O}_K$. 

\section{Strongly self-absorbing actions}

In this section, we review the notion of strongly self-absorbing $G$-$\Cstar$-algebras and 
equip their $K$-theory with a ring structure. 

\begin{dfn}[cf. {\cite[Definition 1.15]{Szabo}}]\label{app G-uni eq}Let $G$ be a second countable locally compact group, $(A,\alpha),\ (B,\beta)$ be two $G$-$\Cstar$-algebras and $\phi, \psi\colon A\to B$ be two $G$-equivariant $\ast$-homomorphisms. Then $\phi$ and $\psi$ are called \textit{approximately $G$-unitarily equivalent} if there exist a sequence of unitaries $(u_n)_{n=1}^\infty$ in $\mathcal{M}(B)$ satisfying %
\begin{align}\label{app eq}&\lim_{n\to\infty}\|\text{Ad}(u_n)\circ \phi(a)-\psi(a)\|=0\quad (a\in A), \\
&\label{app inv}\lim_{n\to\infty}(\beta_g(u_n)-u_n)=0\quad (\text{in the strict topology, }g\in G),\end{align}%
where the convergence of (\ref{app inv}) is uniform on compact subsets of $G$. \end{dfn}

\begin{prp}\label{app G-uni eq}Let $G,A,B,\phi,\psi$ be as above. If $G$ is compact and $B$ is unital, then the homomorphisms %
\[\phi_\ast, \psi_\ast\colon K^G_\ast(A)\to K^G_\ast(B)\]%
induced by $\phi$ and $\psi$ are the same maps. \end{prp}
\begin{proof}We first show that in this case, the sequence of unitaries that gives the equivalence of $\phi$ and $\psi$ can be taken to be $G$-invariant. To see this, let $\beta$ be the action of $G$ on $B$ and $\mu$ be a Haar measure on $G$ with $\mu(G)=1$. Pick $0<\ve<1/4$ and a sequence of unitaries $(u_n)_n$ of $\mathcal{M}(B)=B$ satisfying (\ref{app eq}) and (\ref{app inv}). Since $B$ is unital, the strict convergence in (\ref{app inv}) coinsides with norm convergence. Thus, for a large $n$ and for all $g\in G$, we have $\|\beta_g(u_n)-u_n\|<\ve$. Then %
\[\|u_n-\int_G \beta_g(u_n)d\mu(g)\|\leq\int_G \|u_n-\beta_g(u_n)\|d\mu(g)<\ve,\]%
where the integral $\int_G \beta_g(u_n)d\mu(g)$ converges in the norm topology. Put $v_n\coloneqq\int_G \beta_g(u_n)d\mu(g)$. By the above inequality, $v_n$ is invertible and $G$-invariant. Moreover, we have %
\begin{align*}&\|v_n\| \leq \int_G\| \beta_g(u_n)\|d\mu(g)=1, \\
&\|1-v_n^\ast v_n\| \leq \|u_n^\ast-v_n^\ast\| \|u_n\|+\|v_n^\ast\| \|u_n-v_n\|<2\ve,\quad \text{and} \\
&\|1-(v_n^\ast v_n)^{-\frac{1}{2}}\|=\sup\qty{\abs{1-(1-x)^{-\frac{1}{2}}}\middle|  x\in\sigma(1-v_n^\ast v_n)}<2\ve. \end{align*}
The last inequality follows from $(1-x)^{-1/2}\leq 1+x$ for $0\leq x\leq1/2$. Thus, the sequence of unitaries defined as $u_n^\prime\coloneqq v_n(v_n^\ast v_n)^{-1/2}$ is $G$-invariant and satisfies (\ref{app eq}). \par
If $A$ is unital, then take any finite dimensional $G$-Hilbert space $V$ and any $G$-invariant projection $p\in B(V)\otimes A$. For large $n$, we have $\|\text{Ad}(u_n^\prime)\circ\phi(p)-\psi(p)\|<1/4$. Now there is a $G$-invariant unitary that gives a unitary equivalence between $\text{Ad}(u_n^\prime)\circ\phi(p)$ and $\psi(p)$. Thus, we obtain $[\phi(p)]=[\text{Ad}(u_n^\prime)\circ\phi(p)]=[\psi(p)]$ in $K^G_0(A)$. It implies that $\phi$ and $\psi$ induce the same map from $K^G_0(A)$ to $K^G_0(B)$ if $A$ is unital. The nonunital case follows from the unital case by considering the unitalization $\tilde{A}$. Finally, since $K^G_1(A)$ equals $K^G_0(\Sigma A)$, $\phi$ and $\psi$ also induce the same map from $K^G_1(A)$ to $K^G_1(B)$. \end{proof}

\begin{dfn}[cf. {\cite[Definition 3.1]{Szabo}}]Let $G$ be a second countable locally compact group, $D\neq\C$ be a separable, unital $G$-$\Cstar$-algebra and $\nu\colon\C\to D$ be the unital inclusion. $D$ is called a \textit{strongly self-absorbing $G$-$\Cstar$-algebra} if the $\ast$-homomorphism $\id\otimes\nu\colon D\to D\otimes D$ is approximately $G$-unitarily equivalent to an isomorphism. \end{dfn}

\begin{prp}[cf. {\cite[Proposition 3.4 (ii)]{Szabo}}]\label{app inner half flip}In the above setting, the two maps $\id\otimes \nu,\ \nu\otimes \id\colon D\to D\otimes D$ are approximately $G$-unitarily equivalent. \end{prp}

\begin{cor}\label{multiplication}Let $G$ be a compact group, $D$ be a strongly self-absorbing $G$-$\Cstar$-algebra and $\nu\colon\C\to D$ be the unital inclusion. \begin{itemize}
\item[$(1)$] $\id\otimes\nu$ and $\nu\otimes \id$ induce the same isomorphism %
\[(\id\otimes\nu)_\ast=(\nu\otimes \id)_\ast\colon K^G_\ast(D)\stackrel{\cong}{\to} K^G_\ast(D\otimes D).\]%
\item[$(2)$] The automorphism on $K^G_\ast(D\otimes D)$ induced by the flip automorphism $D\otimes D\to D\otimes D,\ a\otimes b\mapsto b\otimes a$ is the identity map. 
\item[$(3)$] $K^G_\ast(D\otimes C_\nu)=0$, where $C_\nu$ is the mapping cone of $\nu$.\end{itemize}\end{cor}

\begin{proof}\begin{itemize}
\item[(1)]Since $\id\otimes\nu$ is approximately $G$-unitarily equivalent to an isomorphism, $(\id\otimes\nu)_\ast$ is isomorphic by Proposition \ref{app G-uni eq}. Again by Proposition \ref{app G-uni eq} and Proposition \ref{app inner half flip}, we have $(\id\otimes\nu)_\ast=(\nu\otimes \id)_\ast$. 
\item[(2)]Let $f$ be the flip automorphism. Since $f\circ(\nu\otimes \id)=\id\otimes\nu$, %
\[f_\ast\circ(\nu\otimes \id)_\ast=(\id\otimes\nu)_\ast\colon K^G_\ast(D)\stackrel{\cong}{\to} K^G_\ast(D\otimes D)\]%
holds. (1) implies that $f_\ast$ is the identity map. 
\item[(3)]Tensoring the mapping cone sequence %
\[\Sigma D\to C_\nu\to \C\overset{\nu}{\to} D\]%
with $D$, we get the exact triangle %
\[\Sigma D\otimes D\to D\otimes C_\nu\to D\overset{\id\otimes\nu}{\to}D\otimes D.\]%
It gives the exact sequence %
\[\cdots\to K^G_{\ast+1}(D\otimes D)\to K^G_\ast(D\otimes C_\nu)\to K^G_\ast(D)\overset{(\id\otimes \nu)_\ast}{\to}K^G_\ast(D\otimes D)\to\cdots\]%
Since $(\id\otimes\nu)_\ast$ is an isomorphism by (1), $K^G_\ast(D\otimes C_\nu)=0$. \end{itemize}\end{proof}

\begin{dfn}\label{product}Let $G$ be a compact group, $D$ be a strongly self-absorbing $G$-$\Cstar$-algebra and $\nu\colon\C\to D$ be the unital inclusion. We call the composition of the cup product with the isomorphism $(\id\otimes \nu)_\ast^{-1}$%
\[K^G_\ast(D)\times K^G_\ast(D)\to K^G_\ast(D\otimes D)\to K^G_\ast(D),\quad (x, y)\mapsto x\cupprod y\mapsto (\id\otimes \nu)_\ast^{-1}(x\cupprod y)\]%
the \textit{multiplication} of $K^G_\ast(D)$. For $x,y\in K^G_\ast(D)$, we write $(\id\otimes \nu)_\ast^{-1}(x\cupprod y)$ simply as $xy$. \end{dfn}

\begin{prp}\label{ring structure}Let $G,D$ be as above. Then $K^G_\ast(D)$ is a commutative ring with the multiplication defined in the previous definition. \end{prp}
\begin{proof}Since $x\cupprod y=y\cupprod x$ by Corollary \ref{multiplication} (2), the commutativity follows. Let $1_D$ be the unit of $D$. Then $[1_D]$ is the multiplicative identity of $K^G_\ast(D)$ because we have $(\nu\otimes \id)_\ast(x)=[1_D]\cupprod x$ for any $x\in K^G_\ast(D)$. \par%
Take $x,y,z\in K^G_\ast(D)$. In a similar way to the proof of Corollary \ref{multiplication} (1), the two maps $\id\otimes \id\otimes \nu,\ \id\otimes \nu\otimes \id\colon D\otimes D\to D\otimes D\otimes D$ induces the same isomorphism %
\[(\id\otimes \id\otimes \nu)_\ast=(\id\otimes \nu\otimes \id)_\ast\colon K^G_\ast(D\otimes D)\overset{\cong}{\to}K^G_\ast(D\otimes D\otimes D).\]%
Thus, we get the associativity %
\begin{align*}x(yz)&=(\id\otimes \nu)_\ast^{-1}(x\cupprod((\id\otimes \nu)_\ast^{-1}(y\cupprod z))) \\ &=(\id\otimes \nu)_\ast^{-1}\circ(\id\otimes\id\otimes \nu)_\ast^{-1}(x\cupprod y\cupprod z) \\ &=(\id\otimes \nu)_\ast^{-1}\circ(\id\otimes \nu\otimes \id)_\ast^{-1}(x\cupprod y\cupprod z) \\ &=(\id\otimes \nu)_\ast^{-1}(((\id\otimes \nu)_\ast^{-1}(x\cupprod y))\cupprod z) \\ &=(xy)z.\end{align*}%
It is easy to see the distributivity $x(y+z)=xy+xz$. This finishes the proof. \end{proof}

Next, we see the relation between $K^G_\ast(D)$ and $KK^G_\ast(D)$. Let $\vp\colon D\to D\otimes D$ be a $G$-equivariant isomorphism that is approximately $G$-unitarily equivalent to $\id\otimes\nu$. We also write the element of $KK^G(D,D\otimes D)$ given by $\vp$ as $\vp$. There is a map %
\[K^G_\ast(D)\to KK^G_\ast(D,D\otimes D)\to KK^G_\ast(D,D),\quad x\mapsto x\cupprod\id\mapsto (x\cupprod\id)\otimes\vp^{-1}.\]%
\begin{prp}\label{K0(D) and KK(D,D)}Let $G$ be a finite group and $D$ be a strongly self-absorbing $G$-algebra that belongs to the $G$-equivariant bootstrap class. Then $x\in K^G_\ast(D)$ is invertible if and only if $(x\cupprod\id)\otimes\vp^{-1}\in KK^G_\ast(D,D)$ is invertible. \end{prp}
\begin{proof}Let $H\subset G$ be a subgroup and $y\in K^H_\ast(D)$ be an element. We write $\res^H_G(x)$ simply as $x$. Since %
\[y\otimes(x\cupprod\id)=x\cupprod y\in K^H_\ast(D\otimes D),\]%
we have %
\[y\otimes(x\cupprod\id)\otimes\vp^{-1}=(x\cupprod y)\otimes\vp^{-1}=\vp^{-1}_\ast(x\cupprod y)=(\id\otimes\nu)^{-1}_\ast(x\cupprod y)\]%
in $K^G_\ast(D)$, where the last equality follows from Proposition \ref{app G-uni eq}. It means that the map $((x\cupprod\id)\otimes\vp^{-1})_\ast\colon K^H_\ast(D)\to K^H_\ast(D)$ equals the multiplication by $x$. Thus $x$ is invertible if and only if $((x\cupprod\id)\otimes\vp^{-1})_\ast\colon K^H_\ast(D)\to K^H_\ast(D)$ is invertible for all subgroups $H\subset G$. It occurs if and only if $(x\cupprod\id)\otimes\vp^{-1}$ is invertible by Proposition \ref{bootstrap}. \end{proof}

In fact, the above result holds for every compact $G$ (see \cite{Kamikawa}). 

The following corollary is useful to check the invertibility of elements of $K^G_\ast(D)$.

\begin{cor}\label{res invertible}Let $G$ be a finite group and $D$ be a strongly self-absorbing $G$-$\Cstar$-algebra that belongs to the $G$-equivariant bootstrap class. Then $x\in K^G_\ast(D)$ is invertible if and only if $\res^H_G(x)\in K^H_\ast(D)$ is invertible for all cyclic subgroup $H\subset G$. \end{cor}
\begin{proof}By Theorem \ref{K0(D) and KK(D,D)}, $x$ is invertible if and only if $(x\cupprod\id)\otimes\vp^{-1}\in KK^G_\ast(D,D)$ is invertible. By Proposition \ref{bootstrap}, it occurs if and only if $\Res^H_G((x\cupprod\id)\otimes\vp^{-1})\in KK^H_\ast(D,D)$ is invertible for all cyclic subgroup $H\subset G$. Again by Theorem \ref{K0(D) and KK(D,D)}, $\Res^H_G((x\cupprod\id)\otimes\vp^{-1})$ is invertible if and only if $\res^H_G(x)\in K^H_\ast(D)$ is invertible. \end{proof}

\section{Model actions}

The purpose of this section is to construct models of strongly self-absorbing $G$-$\Cstar$-algebras satisfying the localization condition. 
Now recall Definition \ref{loc condition}. Let $D$ be a strongly self-absorbing $G$-$\Cstar$-algebra and $\nu\colon\mathbb{C}\to D$ be the unital inclusion. Set 
\[S_D=\{x\in R(G);\; \nu_*(x) \text{ is invertible in } K_0^G(D)\}.\]
$D$ satisfies the localization condition with respect to a $G$-$\Cstar$-algebra $B$ if there exists an $R(G)$-module isomorphism $h^B_\ast\colon KK^G_\ast(B,\mathbb{C})_{S_D}\to KK^G_\ast(B,D)$ such that the diagram 
\[\begin{tikzcd}[row sep=large, column sep=large]
	KK_\ast^G(B,\mathbb{C})
	\arrow[r]
	\arrow[dr, "\nu^B_*"']
	&
	KK_\ast^G(B,\mathbb{C})_{S_D}
	\arrow[d, "h^B_\ast"', "\cong"]
	\\
	{}
	&
	KK_\ast^G(B,D)
\end{tikzcd}\]
commutes, where the map $KK_i^G(B,\mathbb{C})\to KK_i^G(B,\mathbb{C})_{S_D}$ is the natural map and $\nu^B_\ast$ is the map induced by $\nu$. 

\begin{rem}\label{uniqueness of h} The $R(G)$-module map $h^B_\ast\colon KK^G_\ast(B,\mathbb{C})_{S_D}\to KK^G_\ast(B,D)$ that makes the above diagram commutative is unique if it exists, regardless of whether it is isomorphic or not. Indeed, if such $h^B_\ast$ exists, then for each $x\in S_D$, the multiplication by $x$ on $KK^G_\ast(B,D)$ is invertible and $h^B_\ast$ is given by %
\[h^B_\ast\left(\frac{f}{x}\right)=\nu_\ast(f)\cdot\nu_\ast(x)^{-1}\quad (f\in KK^G_\ast(B,\mathbb{C}), x\in S_D).\]%
Let $D^\prime$ be another strongly self-absorbing $G$-$\Cstar$-algebra with $S_D=S_{D^\prime}$ satisfying the localization condition with respect to $B$. The uniqueness of $h^B_\ast$ gives the uniqueness of the $R(G)$-module map $KK^G(B,D^\prime) \to KK^G(B,D)$, which is automatically an isomorphism, maiking the following diagram commutative: 
\[\begin{tikzcd}[row sep=large, column sep=large]
	KK_\ast^G(B,\mathbb{C})
	\arrow[r, "\nu^B_*"]
	\arrow[dr, "\nu^B_*"']
	&
	KK_\ast^G(B,D^\prime)
	\arrow[d]
	\\
	{}
	&
	KK_\ast^G(B,D)
\end{tikzcd}\] \end{rem}

For a multiplicative subset $S\subset R(G)$, $\overline{S}$ denotes its saturation. 

\begin{prp}\label{model action}Let $G$ be a compact group and $S$ be a multiplicative subset of the representation ring $R(G)$. Then there is a strongly self-absorbing $G$-$\Cstar$-algebra $M^S$ such that $M^S$ belongs to the $G$-equivariant bootstrap class $\mathcal{B}_G$, $S_{M^S}=\overline{S}$ and $M^S$ satisfies the localization condition with respect to every compact object $B\in KK^G$. \end{prp}

\begin{proof}First we consider the case that $S$ is generated by $1$ and an element $x\in R(G)$. Write $x=[\rho_0]-[\rho_1]$, where $[\rho_0]$ and $[\rho_1]$ are the elements given by finite dimensional representations $\rho_0\colon G\curvearrowright H_0$ and $\rho_1\colon G\curvearrowright H_1$, respectively. For $i=0,1$, put $d_i\coloneqq\dim H_i$. Since the Cuntz algebra $\Oi$ is simple and purely infinite, there exists a family of mutually orthogonal projections $\{p_i^{(k)}\mid i=0,1,\ 1\leq k\leq d_i\}$ of $\Oi$ with %
\[[p_i^{(k)}]=(-1)^i\cdot [1_\infty]\in K_0(\Oi),\] %
where $[1_\infty]\in K_0(\Oi)$ is the element given by the identity $1_\infty\in\Oi$. Since $\Oi$ is purely infinite, $[p_i^{(1)}]=[p_i^{(k)}]$ implies that there is a partial isometry $v_i^{(1k)}$ with $v_i^{(1k)}v_i^{(1k)\ast}=p_i^{(1)}$ and $v_i^{(1k)\ast} v_i^{(1k)}=p_i^{(k)}$ for any $1\leq k\leq d_i$. Thus, setting $v_i^{(kl)}\coloneqq v_i^{(1k)\ast}v_i^{(1l)}$ and $p_i\coloneqq\sum_{k=1}^{d_i}p_i^{(k)}$, there is the isomorphism %
\begin{align}\label{model}B(H_i)\otimes p_i^{(1)}\Oi p_i^{(1)}\cong p_i\Oi p_i\end{align}%
such that for the system of matrix units $\{e_i^{(kl)}\}_{k,l=1}^{d_i}$ of $B(H_i)$, it maps $e_i^{(kl)}\otimes p_i^{(1)}$ to $v_i^{(kl)}$. 
Simplicity of $\Oi$ implies that $p_i^{(1)}$ is a full projection of $\Oi$ and so the inclusion $p_i^{(1)}\Oi p_i^{(1)}\hookrightarrow\Oi$ gives the Morita equivalence. Also, by \cite[Theorem 4.4]{Pimsner}, the unital inclusion $\C\to\mathcal{O}_\infty$ gives a $KK$-equivalence $\mathcal{O}_\infty\sim\C$. Under these equvalence, the element of $KK(\C,p_i^{(1)}\Oi p_i^{(1)})=K_0(p_i^{(1)}\Oi p_i^{(1)})$ given by the unital inclusion $\C\to p_i^{(1)}\Oi p_i^{(1)}$ is identified with $(-1)^i\cdot [1]\in K_0(\C)$. Now the isomorphism (\ref{model}) gives the inner action $\pi_i\colon G\curvearrowright p_i\mathcal{O}_\infty p_i$ induced by $\Ad\rho_i\otimes\id\colon G\curvearrowright B(H_i)\otimes p_i^{(1)}\mathcal{O}_\infty p_i^{(1)}$. Thus, the element of $K^G_0(p_i\Oi p_i)=K^G_0((p_i\Oi p_{i,j},\pi_i))$ given by the unital inclusion $\C\to p_i\Oi p_i$ is identified with $(-1)^i\cdot[\rho_i]\in K^G_0(\C)$. Setting $p\coloneqq p_0+p_1$, we define the inner action $\pi\coloneqq\pi_0\oplus\pi_1\colon G\curvearrowright p\Oi p$. Then the element $[\nu_p]\in K^G_0(p\Oi p)$ given by the unital inclusion $\nu_p\colon\C\to p\Oi p$ is identified with %
\[[\rho_0]-[\rho_1]=x\in K_0^G(\C).\] %
Now consider the inductive system $\qty((p\Oi p)^{\otimes m},\iota_m)_m$, where %
\[\iota_m=\nu_p\otimes\id\colon (p\Oi p)^{\otimes m}\to (p\Oi p)\otimes(p\Oi p)^{\otimes m}=(p\Oi p)^{\otimes m+1}.\] %
The above argument shows that $(\iota_m)_\ast\colon KK^G_\ast(B,(p\Oi p)^{\otimes m})\to KK^G_\ast(B,(p\Oi p)^{\otimes m+1})$ is identified with the multiplication-by-$x$ map on $KK^G_\ast(B,\mathbb{C})$. Hence %
\[\varinjlim_{m}\left(KK^G_\ast(B,(p\Oi p)^{\otimes m}),(\iota_m)_\ast\right)\cong KK^G_\ast(B,\mathbb{C})_S.\] %
Since all $\Cstar$-algebras appearing in $\qty((p\Oi p)^{\otimes m},\iota_m)_m$ are nuclear, this inductive system is admissible and the inductive limit $\bigotimes_{\mathbb{N}}p\Oi p$ is $KK^G$-equivalent to its homotopy limit (for the definition of "admissible" and "homotopy limit", see \cite[2.4]{Meyer and Nest}). Thus, by the compactness of $B$ and by \cite[Lemma 2.4]{Meyer and Nest}, %
\[KK^G_\ast(B, \bigotimes_{\mathbb{N}}p\Oi p)\cong\varinjlim_{m}\left(KK^G_\ast(B, (p\Oi p)^{\otimes m})\right).\] %
$\bigotimes_{\mathbb{N}}p\Oi p$ is in the $G$-equivariant bootstrap class $\mathcal{B}_G$ because all inner actions are Morita equivalent to the trivial action and $p\Oi p$ is $KK$-equivalent to $\C$. 
Since $\mathcal{O}_\infty$ is strongly self-absorbing, the corner $p\Oi p$, as a non-equivariant $\Cstar$-algebra, has approximate 
inner flip. 
Therefore, by \cite[Proposition 6.3]{Szabo 2}, $M^S\coloneqq\bigotimes_{\mathbb{N}}p\Oi p$ is a strongly self-absorbing $G$-$\Cstar$-algebra. It satisfies $S_{M^S}=\overline{S}$ because %
\begin{align*}K^G_i(\bigotimes_{\mathbb{N}}p\Oi p)\cong\varinjlim_{m}(K^G_i((p\Oi p)^{\otimes m}))\cong\begin{cases*}R(G)_S & ($i=0$), \\ 0 & ($i=1$).\end{cases*} \end{align*}\par%
In the case that $S$ is general, for each $x\in S$, let $S_x$ be the multiplicative subset generated by $1$ and $x$, and take $M^{S_x}$ as above. Then $M^S\coloneqq\bigotimes_{x\in S}M^{S_x}$ is the desired one. Note that $S\subset R(G)$ is a countable set because $G$ is compact. \end{proof}

In the proof of the preceding proposition, $M^S$ is given as an infinite tensor product of unital Kirchberg algebras of the form $p\mathcal{O}_\infty p$. Thus, this $M^S$ is also a unital Kirchberg algebra. 
Let $\mu^S$ be the action of $G$ on $M^S$ and $\gamma\colon G\curvearrowright\mathcal{O}_\infty$ be the action defined in \cite[Definition 3.4]{Gabe and Szabo} (in fact, we can choose $\gamma$ to be a quasi-product action thanks to 
\cite[Proposition 7.2]{Izumi} ). 
The unital inclusion $\C\to (\mathcal{O}_\infty, \gamma)$ gives a $KK^G$-equivalence $\C\sim(\mathcal{O}_\infty, \gamma)$ by \cite[Remark 3.3]{Gabe and Szabo}. 
Also, $(\mathcal{O}_\infty, \gamma)$ is strongly self-absorbing by \cite[Corollary 6.8]{Gabe and Szabo}. 
Hence, by \cite[Corollary 6.8]{Gabe and Szabo}, the action $\mu^S\otimes\gamma\colon G\curvearrowright M^S\otimes \mathcal{O}_\infty$ is isometrically shift-absorbing and strongly self-absorbing. Moreover, $(M^S\otimes \mathcal{O}_\infty, \mu^S\otimes\gamma)$ is $KK^G$-equivalent to $(M^S,\mu^S)$. This observation gives the following corollary. 

\begin{cor}\label{model remark} $M^S=(M^S,\mu^S)$ in Proposition \ref{model} can be chosen so that $M^S$ is a unital Kirchberg algebra and the action $\mu^S\colon G\curvearrowright M^S$ is isometrically shift-absorbing. \end{cor}

\begin{rem}\label{model limit} In the argument of Proposition \ref{model} (and Corollary \ref{model remark}), $M^S$ was constructed as a direct limit of an inductive system consisting of $G$-$\Cstar$-algebras that are $KK^G$-equivalent to $\mathbb{C}$. More precisely, let us relabel the elements of $S\subset R(G)$ as $\{x_n\}_{n=1}^\infty$. Set the inductive system $(A_n,\iota_n)_n$ as $A_n=\mathbb{C}$ and $\iota_n=x_1\cdots x_n\in KK^G(A_n,A_{n+1})=R(G)$. Then the homotopy limit of this system is $KK^G$-equivalent to $M^S$. \end{rem}

Our ultimate goal is to prove that every strongly self-absorbing $D\in\mathcal{B}_G$ is $KK^G$-equivalent to one of the models $M^S$ in Proposition \ref{model}. We conclude this section by establishing the following proposition, which serves as a stepping stone toward this goal. 

\begin{prp}\label{loc KK} Let $\mathcal{B}_G^0$ be a subclass of the $G$-equivariant bootstrap class $\mathcal{B}_G$ that generates $\mathcal{B}_G$ and $D$ be a strongly self-absorbing $G$-$\Cstar$-algebra in $\mathcal{B}_G$. 
Suppose that $D$ satisfies the localization condition with respect to every $B\in\mathcal{B}_G^0$. 
Then there is a $KK^G$-equivalence between $M^{S_D}$ of Proposition \ref{model} (or Corollary \ref{model remark}) and $D$ that maps $[1_{M^{S_D}}]\in K^G_0(M^{S_D})$ to $[1_D]\in K^G_0(D)$. \end{prp}

\begin{proof}
Take the inductive system $(A_n,\iota_n)_n$ of Remark \ref{model limit} for $S=S_D$. By \cite[Lemma 2.4]{Meyer and Nest}, there is a short exact sequence
\[\varprojlim{}^1 KK^G_{\ast+1}(A_n,D)\rightarrowtail KK^G_\ast(M^{S_D},D)\twoheadrightarrow \varprojlim KK^G_\ast(A_n,D),\]
where $\varprojlim{}^1 KK^G_{\ast}(A_n,D)$ is %
\begin{align*}\Cok\qty(\id-(\iota_{n-1}^\ast)_n\colon\prod_{n=1}^\infty KK^G_\ast(A_n,D)\to\prod_{n=1}^\infty KK^G_\ast(A_n,D)). \end{align*}
Since $A_n=\mathbb{C}$ and $\iota_{n-1}\in KK^G(A_{n-1},A_n)$ is identified with an element of $S_D$, the map $(\iota_{n-1})^\ast \colon KK^G_\ast(A_n,D)\to KK^G_\ast(A_{n-1},D)$ is invertible. Hence, $\varprojlim{}^1 KK^G_{\ast+1}(A_n,D)$ vanishes and we have the isomorphism 
\[KK^G_\ast(M^{S_D},D) \cong \varprojlim KK^G_\ast(A_n,D) \cong K^G_\ast(D).\] 
Let $f\in KK^G(M^{S_D},D)$ be the element corresponding to $[1_D]\in K^G_0(D)$. Then we have $[1_{M^{S_D}}]\otimes f = [1_D]$. Equivalently, the diagram 
\[\begin{tikzcd}[row sep=large, column sep=large]
	KK_\ast^G(B,\mathbb{C})
	\arrow[r, "\nu^B_*"]
	\arrow[dr, "\nu^B_*"']
	&
	KK_\ast^G(B,M^{S_D})
	\arrow[d,"f_\ast"]
	\\
	{}
	&
	KK_\ast^G(B,D)
\end{tikzcd}\]
commutes for all $B\in KK^G$. Now $D$ and $M^{S_D}$ satisfy the localization condition with respect to every $B\in \mathcal{B}_G^0$. Thus, by Remark \ref{uniqueness of h}, the map $f_\ast\colon KK^G(B,M^{S_D}) \to KK^G(B,D)$ is an isomorphism. Combining this with Proposition \ref{bootstrap}, we conclude that $f\in KK^G(M^{S_D},D)$ is invertible. 
\end{proof}

Recall that $\mathcal{B}_G$ is generated by $G$-$\Cstar$-algebras of the form $\Ind^G_H M_n$, which are compact by \cite[Proposition 3.13]{Lefschetz}. Hence, a subclass $\mathcal{B}_G^0$ as in the previous proposition always exists. In particular, $D$ is $KK^G$-equivalent to $M^{S_D}$ preserving identities if Conjecture \ref{conjecture} holds for $D$.

\section{Equivariant K\"{u}nneth formula}

In this section we prove an equivariant K\"unneth formula that is needed to calculate the equivariant $K$-theory of strongly self-absorbing actions. To state it, we have to define some notions. \par
For a positive integer $n$, an abelian group is said to be \textit{uniquely n-divisible} if the multiplication by $n$ is bijective on it. 

\begin{dfn}Let $n$ be a positive integer and $G$ be a finite group. We say that $A\in KK^G$ {\it has uniquely $n$-divisible $K$-theory} if for each subgroup $H$ of $G$, $K^H_\ast(A)$ is uniquely $n$-divisible.\end{dfn}

For $A\in KK^{\Zn}$, its equivariant $K$-theory $K^{\Zn}_\ast(A)$ is a left $R(\Zn)$-module. The representation ring $R(\Zn)$ is isomorphic to $\mathbb{Z}[t]/(t^n-1)$. Let $\Phi_n(t)\in\mathbb{Z}[t]/(t^n-1)$ be the $n$-th cyclotomic polynomial: the minimal polynomial of the primitive $n$-th root of unity $\zeta_n$. Note that the ring $\mathbb{Z}[t]/(\Phi_n)$ is isomorphic to $\mathbb{Z}[\zeta_n]$.

\begin{dfn}[cf. {\cite[p.2]{Meyer and Nadareishvili}}] For $A\in KK^{\Zn}$, a $\mathbb{Z}[\zeta_n]$-module $F^{\Zn}_\ast(A)$ is defined by %
\[F^{\Zn}_\ast(A)=\qty{x\in K^{\Zn}_\ast(A)\mid \Phi_n\cdot x=0}.\]%
$F^{\Zn}_\ast$ is the functor from $KK^{\Zn}$ to the category of $\mathbb{Z}/2$-graded left $\mathbb{Z}[\zeta_n]$-modules that maps an object $A\in KK^{\Zn}$ to $F^{\Zn}_\ast(A)$. \end{dfn}

The main result of this section is as follows. 

\begin{thm}\label{Kunneth} Let $A,B$ be objects in $KK^{\Zn}$. If $A$ has uniquely $n$-divisible $K$-theory and $B$ is in the $\Zn$-equivariant bootstrap class $\mathcal{B}_{\Zn}$, then there is an short exact sequence of $\mathbb{Z}[\zeta_n]$-modules %
\[F^{\mathbb{Z}/n}_\ast(A)\otimes_{\mathbb{Z}[\zeta_{n}]} F^{\mathbb{Z}/n}_\ast(B)\rightarrowtail F^{\mathbb{Z}/n}_\ast(A\otimes B)\twoheadrightarrow \text{Tor}^{\mathbb{Z}[\zeta_{n}]}_1(F^{\mathbb{Z}/n}_\ast(A),F^{\mathbb{Z}/n}_{\ast-1}(B)),\]%
where the first homomorphism is given by the cup product. \end{thm}

Before starting the proof of the above theorem, we introduce Meyer and Nadareishvili's universal coefficient theorem. Let $G$ be a finite group and $H\subset G$ be a cyclic subgroup. We write the normalizer of $H$ by $N_H$. For any $g\in N_H$, $\con_{g,H}\colon K^H_\ast(\C)\to K^H_\ast(\C)$ fixes $\Phi_\abs{H}$. Thus, $\con_{g,H}$ gives an automorphism of $R(H)/\Phi_\abs{H}\cong \mathbb{Z}[\zeta_\abs{H}]$. Similarly, for any $A\in KK^G$, $\con_{g,H}\colon K^H_\ast(A)\to K^H_\ast(A)$ gives an automorphism of $F^H_\ast(A)$. If $g\in H$, then $\con_{g,H}$ acts trivially on $\mathbb{Z}[\zeta_\abs{H}]$ and on $F^H_\ast(A)$. Hence, $F^H_\ast(A)$ is a $\mathbb{Z}[\zeta_\abs{H}]\rtimes W_H$-module, where $W_H\coloneqq N_H/H$. 

\begin{thm}[cf. {\cite[Theorem 1.1]{Meyer and Nadareishvili}}]\label{UCT}Let $G$ be a finite group and $A,B$ be objects in $KK^G$. If $A$ is in the $G$-equivariant bootstrap class and $\abs{G}\cdot \id\in KK^G(B,B)$ is invertible, then there is a short exact sequence %
\begin{align*}\prod_{H\subset G\ \text{cyclic}} \operatorname{Ext}^1_{\mathbb{Z}[\zeta_\abs{H}]\rtimes W_H}(&F^H_{\ast-1}(A), F^H_\ast(B)) \rightarrowtail KK^G(A,B) \\
& \twoheadrightarrow \prod_{H\subset G\ \text{cyclic}} \operatorname{Hom}_{\mathbb{Z}[\zeta_\abs{H}]\rtimes W_H}(F^H_\ast(A),F^H_\ast(B)).\end{align*}\end{thm}

Although this universal coefficient theorem is not needed to prove Theorem \ref{Kunneth}, the two proofs share common ideas. We will also use Theorem \ref{UCT} to the computation in Subsection 5.3.  

\begin{rem}\label{decomposition by psi}If $A\in KK^{\Zn}$ has uniquely $n$-divisible $K$-theory, then $K^{\Zn}_\ast(A)$ is a $\mathbb{Z}[t,1/n]/(t^n-1)$-module. By \cite[Lemma 5.3]{Meyer and Nadareishvili}, this is decomposed by idempotents $(\psi_k)_{k|n}$ in $\mathbb{Z}[t,1/n]/(t^n-1)$ defined as %
\[\psi_k(t)=\frac{t(t^n-1)}{n\Phi_k(t)}\frac{d\Phi_k(t)}{dt}.\]%
The direct summand $\psi_n K^{\Zn}_\ast(A)$ of $K^{\Zn}_\ast(A)$ is just $F^{\Zn}_\ast(A)$. Indeed, $\psi_n\Phi_n=0$ implies $\psi_nK^{\Zn}_\ast(A)\subset F^{\Zn}_\ast(A)$. If $k$ divides $n$ and $k\neq n$, then $\Phi_{n}$ divides $(t^n-1)/\Phi_k$. Hence, for any $x\in F^{\Zn}_\ast(A)$, %
\[x=\sum_{k|n}\psi_k x=\psi_n x.\]%
This formula means $F^{\Zn}_\ast(A)\subset\psi_nK^{\Zn}_\ast(A)$. \end{rem}

\begin{lem}\label{category n} Let $\mathfrak{N}$ be the full subcategory of $KK^{\Zn}$ consisting of objects that have uniquely $n$-divisible $K$-theory. Then $\mathfrak{N}$ is a localising subcategory of $KK^{\Zn}$, and the restriction of the functor $F^{\Zn}_\ast$ to $\mathfrak{N}$ is a homological functor. \end{lem}
\begin{proof}It is easy to see that $\mathfrak{N}$ is closed under countable direct sums, $KK^{\Zn}$-equivalence and suspension. For any exact triangle %
\[\Sigma C\longrightarrow A\longrightarrow B\longrightarrow C\]%
in $KK^{\Zn}$ and for any subgroup $H$ of $\Zn$, we have a long exact sequence %
\[\cdots\longrightarrow K^H_{\ast+1}(C)\longrightarrow K^H_\ast(A)\longrightarrow K^H_\ast(B)\longrightarrow K^H_\ast(C)\longrightarrow\cdots.\]%
If any two of $A,B,C$ are in $\mathfrak{N}$, then the third one is also in $\mathfrak{N}$ by the five lemma. Thus $\mathfrak{N}$ is a localising subcategory of $KK^{\Zn}$. On $\mathfrak{N}$, the functor $F^{\Zn}_\ast$ equals the direct summand $\psi_nK^{\Zn}_\ast$ of $K^{\Zn}_\ast$ by Remark \ref{decomposition by psi}. Since $K^{\Zn}_\ast$ is a homology, its direct summand $F^{\Zn}_\ast$ is also a homology on $\mathfrak{N}$.
\end{proof}

\begin{lem}\label{tensor n-div} Let $A,B$ be objects in $KK^{\Zn}$. If $A$ has uniquely $n$-divisible $K$-theory and $B$ is in the equivariant bootstrap class $\mathcal{B}_{\Zn}$, then $A\otimes B$ also has uniquely $n$-divisible $K$-theory. \end{lem}
\begin{proof}Let $\mathcal{C}$ be the full subcategory of $KK^{\Zn}$ consisting of objects $B$ with $A\otimes B$ having uniquely $n$-divisible $K$-theory. By considering the functor $K^{\Zn}_\ast(A\otimes\cdot)$, we can check that $\mathcal{C}$ is a localising subcategory of $KK^{\Zn}$ in the same way as the proof of Lemma \ref{category n}. Thus, to see $\mathcal{B}_{\Zn}\subset\mathcal{C}$, it suffices to show that $C((\Zn)/H)$ is an object in $\mathcal{C}$ for each subgroup $H$ of $\Zn$. \par%
Let $H$ be a subgroup of $\Zn$. By Lemma \ref{IndRes} and Proposition \ref{reciprocity}, for any subgroup $K$ of $\Zn$, %
\begin{align*}K^K_\ast(\Res^K_{\Zn}\Ind^{\Zn}_H\Res^H_{\Zn}A)&\cong KK^{\Zn}_\ast(C((\Zn)/K),\Ind^{\Zn}_H\Res^H_{\Zn}A) \\ &\cong KK^H_\ast(\Res^{H}_{\Zn}C((\Zn)/K),\Res^H_{\Zn}A)\end{align*}%
holds. Let $X$ be a complete system of representatives of $H\backslash(\Zn)/K$. Under the equality %
\[(\Zn)/K=\bigsqcup_{h\in H/H\cap K}^{}\bigsqcup_{g\in X}^{}hgK,\]%
we get an $H$-equivariant $\ast$-isomorphism of $\Cstar$-algebras %
\[\Res^H_{\Zn}C((\Zn)/K)\cong (C(H/H\cap K))^\abs{X}.\]%
Again by Proposition \ref{reciprocity}, %
\[KK^H_\ast(C(H/H\cap K)^{\abs{X}}, \Res^H_{\Zn}A)\cong K^{H\cap K}_\ast(\Res^{H\cap K}_{\Zn}A)^{\abs{X}}.\]%
The right hand side is uniquely $n$-divisible. Therefore, $A\otimes C((\Zn)/H)$ has uniquely $n$-divisible $K$-theory and $C((\Zn)/H)$ is in $\mathcal{C}$. \end{proof}

\begin{cor}\label{tensor homology} Let $A$ be an object in $KK^{\Zn}$ having uniquely $n$-divisible $K$-theory. The functor $F^{\Zn}_\ast(A\otimes\cdot)|_{\mathcal{B}_{\Zn}}$ is a homological functor from $\mathcal{B}_{\Zn}$ to the category of $\mathbb{Z}/2$-graded left $\mathbb{Z}[\zeta_n]$-modules. \end{cor}
\begin{proof}It directly follows from Lemma \ref{category n} and Lemma \ref{tensor n-div}.\end{proof}

\begin{prp}\label{F vanish} Let $A$ be an object in $KK^{\Zn}$ having uniquely $n$-divisible $K$-theory. If $F^{\Zn}_\ast(A)=0$, then $F^{\Zn}_\ast(A\otimes B)=0$ for any object $B$ in $\mathcal{B}_{\Zn}$. \end{prp}
\begin{proof}Let $\mathcal{C}$ be the kernel of the functor $F^{\Zn}_\ast(A\otimes\cdot)|_{\mathcal{B}_{\Zn}}$. This functor is a homology by Corollary \ref{tensor homology}, and so $\mathcal{C}$ is a localising subcategory of $\mathcal{B}_{\Zn}$. Clearly $\mathcal{C}$ contains $\mathbb{C}$. Let $H$ be a proper subgroup of $\Zn$ and $i_H$ be the isomorphism $K^{\Zn}_\ast(A\otimes C((\Zn)/H))\cong K^H_\ast(A)$ of Proposition \ref{reciprocity} (1). By Remark \ref{counit}, for any $f\in R(\Zn)$ and $x\in K^{\Zn}_\ast(A\otimes C((\Zn)/H))$, we have $i_H(f\cdot x)=\Res^H_{\Zn}(f)\cdot i_H(x)$. Since $F^{\Zn}_\ast(A\otimes C((\Zn)/H))$ is the image of the multiplication by $(t^n-1)/\Phi_n(t)$ and $t^{\abs{H}}-1$ divides $(t^n-1)/\Phi_n(t)$, we get $\Res^H_{\Zn}((t^n-1)/\Phi_n(t))=0$ and $F^{\Zn}_\ast(A\otimes C((\Zn)/H))=0$. Therefore, $C((\Zn)/H)$ is in $\mathcal{C}$ and $\mathcal{C}$ equals $\mathcal{B}_{\Zn}$. \end{proof}
\begin{thm}\label{geometric resolution} Let $A$ be an object in $KK^{\Zn}$ having uniquely $n$-divisible $K$-theory. Then there is a $\Cstar$-algebra $\mathscr{F}$ with the trivial $\Zn$-action and an element $\alpha\in KK^{\Zn}(\mathscr{F}\otimes M_{n^\infty},A)$ such that $\mathscr{F}$ is an at most countable sum of $\C$ and $C_0(\R)$, and $\alpha$ induces a surjection $\alpha_\ast$ from $F^{\Zn}_\ast(\mathscr{F}\otimes M_{n^\infty})$ to $F^{\Zn}_\ast(A)$. Moreover, if $F^{\Zn}_\ast(A)$ is a free $\mathbb{Z}[\zeta_n,1/n]$-module, then $\mathscr{F}$ and $\alpha$ can be taken so that $\alpha_\ast$ is an isomorphism. \end{thm}
\begin{proof}Consider the inductive system $(M_{n^k},\iota_k)_{k=1}^{\infty}$, where $\iota_k\colon M_{n^k}\to M_{n^{k+1}}$ is the unital inclusion. It gives a projective system $(KK^{\Zn}_\ast(M_{n^k},A),\iota^{\ast}_k)$. Then its projective limit is isomorphic to $KK^{\Zn}_\ast(M_{n^\infty},A)$. Indeed, by \cite[Lemma 2.4, Proposition 2.6 and Lemma 2.7]{Meyer and Nest}, there is a short exact sequence %
\[\varprojlim{}^1 KK^G_{\ast+1}(M_{n^k},A)\rightarrowtail KK^G_\ast(M_{n^\infty},A)\twoheadrightarrow \varprojlim KK^G_\ast(M_{n^k},A),\]%
where $\varprojlim{}^1 KK^G_{\ast}(M_{n^k},A)$ is %
\begin{align*}\Cok\qty(\id-(\iota_{k-1}^\ast)_k\colon\prod_{k=1}^\infty KK^G_\ast(M_{n^k},A)\to\prod_{k=1}^\infty KK^G_\ast(M_{n^k},A)).\end{align*}%
Under the Morita equivalence $M_{n^k}\sim\C$, the map $\iota^\ast_k$ is identified with the multiplication-by-$n$ map on $K^{\Zn}_\ast(A)$. By combining it with the assumption that $A$ has uniquely $n$-divisible $K$-theory, $\varprojlim{}^1 KK^G_{\ast}(M_{n^k},A)$ vanishes. \par%
Pick a generating set (respectively, a basis if $F^{\Zn}_\ast(A)$ is free) $\{f_i\}_{i\in I}$ of the $\mathbb{Z}[\zeta_n,1/n]$-module $F^{\Zn}_0(A)$. Since $F^{\Zn}_0(A)\subset K^{\Zn}_0(A)$ is at most countable, the index set $I$ can be chosen to be at most countable. For each $i\in I$, let $\widetilde{f_i}$ be an element of $KK^{\Zn}(M_{n^\infty},A)$ corresponding to the element%
\[[(n^{-k}f_i)_k]\in\varprojlim K^{\Zn}(A)\cong \varprojlim KK^{\Zn}(M_{n^k},A).\]%
This $\widetilde{f_i}$ induces a surjection from $F^{\Zn}_0(M_{n^\infty})$ to the submodule of $F^{\Zn}_0(A)$ generated by $f_i$ (respectively, an isomorphism). Also, pick a generating set (respectively, a basis) $\{h_j\}_{j\in J}$ of $F^{\Zn}_1(A)$ and take $\widetilde{h_j}\in KK^{\Zn}(\Sigma M_{n^\infty},A)$ in a similar way. Now we define %
\[\mathscr{F}\coloneqq\C^I\oplus C_0(\R)^J,\]
\begin{align*}\alpha\coloneqq\qty((\widetilde{f_i})_i,(\widetilde{h_j})_j)\in &\qty(\prod_{i\in I}KK^{\Zn}(M_{n^\infty},A))\oplus \qty(\prod_{j\in J}KK^{\Zn}(\Sigma M_{n^\infty},A)) \\ &\cong KK^{\Zn}(F\otimes M_{n^\infty},A).\end{align*}%
By construction, $\alpha_\ast\colon F^{\Zn}_\ast(\mathscr{F}\otimes M_{n^\infty})\to F^{\Zn}_\ast(A)$ is a surjection (respectively, an isomorphism). \end{proof}

\begin{lem}\label{F(M tensor B)}For any object $B$ in $\mathcal{B}_{\Zn}$, the cup product gives an isomorphism of $\mathbb{Z}[\zeta_n]$-modules $F^{\Zn}_\ast(M_{n^\infty})\otimes_{\mathbb{Z}[\zeta_n]}F^{\Zn}_\ast(B)\cong F^{\Zn}_\ast(M_{n^\infty}\otimes B)$. \end{lem}
\begin{proof}By Lemma \ref{tensor n-div}, $M_{n^\infty}\otimes B$ has uniquely $n$-divisible $K$-theory. Using the idempotent $\psi_n$ defined in Remark \ref{decomposition by psi}, $F^{\Zn}_\ast(M_{n^\infty}\otimes B)=\psi_n K^{\Zn}_\ast(M_{n^\infty}\otimes B)$ holds. $M_{n^\infty}\otimes B$ is the inductive limit of the system $\qty(M_{n^k}\otimes B, \iota_k\otimes id_B)_k$, where $\iota_k\colon M_{n^k}\to M_{n^{k+1}}$ is the unital inclusion. Under the Morita equivalence $M_{n^k}\sim\C$, this inductive system is identified with the one that consists of copies of $B$ and the multiplication-by-$n$ map from $B$ to $B$. Thus, $K^{\Zn}_\ast(M_{n^\infty}\otimes B)$ is isomorphic to $\mathbb{Z}[1/n]\otimes_{\mathbb{Z}}K^{\Zn}_\ast(B)$, and $F^{\Zn}_\ast(M_{n^\infty}\otimes B)$ is isomorphic to $\mathbb{Z}[1/n]\otimes_{\mathbb{Z}}F^{\Zn}_\ast(B)$. On the other hand, $F^{\Zn}_\ast(M_{n^\infty})$ is isomorphic to $\mathbb{Z}[\zeta_n, 1/n]$. Therefore, the cup product map $K^{\Zn}_\ast(M_{n^\infty})\otimes_{R(\Zn)}K^{\Zn}_\ast(B)\to K^{\Zn}_\ast(M_{n^\infty}\otimes B)$ gives an isomorphism from $F^{\Zn}_\ast(M_{n^\infty})\otimes_{\mathbb{Z}[\zeta_n]}F^{\Zn}_\ast(B)\cong\mathbb{Z}[\zeta_n, 1/n]\otimes_{\mathbb{Z}[\zeta_n]}F^{\Zn}_\ast(B)$ to $F^{\Zn}_\ast(M_{n^\infty}\otimes B)\cong\mathbb{Z}[1/n]\otimes_{\mathbb{Z}}F^{\Zn}_\ast(B)$. \end{proof}

\begin{thm}\label{projective case} Let $A$ be an object in $KK^{\Zn}$ having uniquely $n$-divisible $K$-theory. If $F^{\Zn}_\ast(A)$ is a projective $\mathbb{Z}[\zeta_n, 1/n]$-module, then for any $B$ in $\mathcal{B}_{\Zn}$, two modules $F^{\Zn}_\ast(A)\otimes_{\mathbb{Z}[\zeta_n]}F^{\Zn}_\ast(B)$ and $F^{\Zn}_\ast(A\otimes B)$ are isomorphic by the cup product. \end{thm}
\begin{proof}Take $\mathscr{F}$ and $\alpha$ as in Theorem \ref{geometric resolution}. There is an exact trinagle %
\begin{align}\label{resolution triangle}\Sigma A\longrightarrow C\longrightarrow \mathscr{F}\otimes M_{n^\infty}\overset{\alpha}\longrightarrow A\end{align}%
in $KK^{\Zn}$. By lemma \ref{category n} and the choice of $\alpha$, the object $C$ has uniquely $n$-divisible $K$-theory and there is a short exact sequence of $\mathbb{Z}[\zeta_n,1/n]$-modules %
\begin{align}\label{resolution} F^{\Zn}_\ast(C)\rightarrowtail F^{\Zn}_\ast(\mathscr{F}\otimes M_{n^\infty})\overset{\alpha_\ast}\twoheadrightarrow F^{\Zn}_\ast(A).\end{align}%
Since $\alpha_\ast$ is surjective and $F^{\Zn}_\ast(A)$ is projective, we have $F^{\Zn}_\ast(\mathscr{F}\otimes M_{n^\infty})\cong F^{\Zn}_\ast(C\oplus A)$. If $F^{\Zn}_\ast(C\oplus A)\otimes_{\mathbb{Z}[\zeta_n]}F^{\Zn}_\ast(B)$ and $F^{\Zn}_\ast((C\oplus A)\otimes B)$ are isomorphic, then their direct summands $F^{\Zn}_\ast(A)\otimes_{\mathbb{Z}[\zeta_n]}F^{\Zn}_\ast(B)$ and $F^{\Zn}_\ast(A\otimes B)$ are also isomorphic. Thus, we can reduce the problem to the case that $F^{\Zn}_\ast(A)$ is a free $\mathbb{Z}[\zeta_n,1/n]$-module. \par%
Again by Theorem \ref{geometric resolution}, we can take an exact trinagle of the form (\ref{resolution triangle}) so that $\alpha_\ast\colon F^{\Zn}_\ast(\mathscr{F}\otimes M_{n^\infty})\to F^{\Zn}_\ast(A)$ is an isomorphism. Then $F^{\Zn}_\ast(C)$ vanishes by Lemma \ref{category n}, and so $F^{\Zn}_\ast(C\otimes B)$ vanishes by Lemma \ref{F vanish}. We also have an exact trinagle %
\[\Sigma A\otimes B\longrightarrow C\otimes B\longrightarrow \mathscr{F}\otimes M_{n^\infty}\otimes B\overset{\alpha\otimes id_B}\longrightarrow A\otimes B.\]%
Now $F^{\Zn}_\ast(C\otimes B)=0$ and Lemma \ref{category n} implies that $F^{\Zn}(\mathscr{F}\otimes M_{n^\infty}\otimes B)$ and $F^{\Zn}_\ast(A\otimes B)$ are isomorphic. By the construction of $\mathscr{F}$ and  Lemma \ref{F(M tensor B)}, $F^{\Zn}_\ast(\mathscr{F}\otimes M_{n^\infty})\otimes_{\mathbb{Z}[\zeta_n]}F^{\Zn}_\ast(B)$ is isomorphic to $F^{\Zn}(\mathscr{F}\otimes M_{n^\infty}\otimes B)$. Hence %
\begin{align*}F^{\Zn}_\ast(A)\otimes_{\mathbb{Z}[\zeta_n]}F^{\Zn}_\ast(B) &\cong F^{\Zn}_\ast(\mathscr{F}\otimes M_{n^\infty})\otimes_{\mathbb{Z}[\zeta_n]}F^{\Zn}_\ast(B) \\ &\cong F^{\Zn}_\ast(\mathscr{F}\otimes M_{n^\infty}\otimes B) \\ &\cong F^{\Zn}_\ast(A\otimes B).\end{align*}\end{proof}

We are now ready to prove the main result of this section. 

\begin{proof}[Proof of Theorem \ref{Kunneth}]Take $\mathscr{F}$ and $\alpha$ as in Theorem \ref{geometric resolution}. By considering an exact triangle of the form (\ref{resolution triangle}), we get a short exact sequence (\ref{resolution}) of $\mathbb{Z}[\zeta_n, 1/n]$-modules. In this sequence, the module $F^{\Zn}_\ast(C)$ is a submodule of the free $\mathbb{Z}[\zeta_n, 1/n]$-module $F^{\Zn}_\ast(\mathscr{F}\otimes M_{n^\infty})$. By Remark \ref{Dedekind} and Remark \ref{Z[zeta] is Dedekind}, $\mathbb{Z}[\zeta_n,1/n]$ is a Dedekind domain. Thus, $F^{\Zn}_\ast(C)$ is projective by Remark \ref{hereditary}. Let $\pi\colon C\to\mathscr{F}\otimes M_{n^\infty}$ be the morphism appearing in the triangle (\ref{resolution triangle}). Now we have a commutative diagram %
\[\begin{tikzcd} \text{Tor}^{\mathbb{Z}[\zeta_{n}]}_1(F^{\Zn}_\ast(\mathscr{F}\otimes M_{n^\infty}),F^{\Zn}_\ast(B)) \arrow[d] \\ 
\text{Tor}^{\mathbb{Z}[\zeta_{n}]}_1(F^{\Zn}_\ast(A),F^{\Zn}_\ast(B)) \arrow[d] & F^{\Zn}_{\ast+1}(A\otimes B) \arrow[d]\\ 
F^{\Zn}_\ast(C)\otimes_{\mathbb{Z}[\zeta_{n}]}F^{\Zn}_\ast(B) \arrow[r,"\omega_C","\cong"'] \arrow[d,"\pi_\ast\otimes \id"] 
& F^{\Zn}_\ast(C\otimes B) \arrow[d,"(\pi\otimes \id)_\ast"] \\ 
F^{\Zn}_\ast(\mathscr{F}\otimes M_{n^\infty})\otimes_{\mathbb{Z}[\zeta_{n}]}F^{\Zn}_\ast(B) \arrow[r,"\omega_{\mathscr{F}}","\cong"'] 
\arrow[d,twoheadrightarrow,"\alpha_\ast\otimes \id"] 
& F^{\Zn}_\ast(\mathscr{F}\otimes M_{n^\infty}\otimes B) \arrow[d,"(\alpha\otimes \id)_\ast"] \\ 
F^{\Zn}_\ast(A)\otimes_{\mathbb{Z}[\zeta_{n}]}F^{\Zn}_\ast(B) \arrow[r,"\omega_A"] & F^{\Zn}_\ast(A\otimes B) .\end{tikzcd}\]%
In the above diagram, two columns are exact and the horizontal maps $\omega_C,\ \omega_{\mathscr{F}}$ and $\omega_A$ are all cup product maps. By theorem \ref{projective case}, $\omega_C$ and $\omega_{\mathscr{F}}$ are isomorphisms. \par%
The remainder of the proof is a diagram chase. If $x\in\Ker\omega_A$, then there are $y\in F^{\Zn}_\ast(\mathscr{F}\otimes M_{n^\infty})\otimes_{\mathbb{Z}[\zeta_{n}]}F^{\Zn}_\ast(B)$ and $z\in F^{\Zn}_\ast(C)\otimes_{\mathbb{Z}[\zeta_{n}]}F^{\Zn}_\ast(B)$ such that $(\alpha_\ast\otimes \id)(y)=x$ and $(\pi\otimes \id)_\ast\circ\omega_C(z)=\omega_{\mathscr{F}}(y)$. Thus %
\[x=(\alpha_\ast\otimes \id)\circ (\pi_\ast\otimes \id)(z)=0\] %
and $\omega_A$ is injective. The above diagram also implies $\Cok\omega_A=\Cok(\alpha\otimes \id)_\ast$. It is isomorphic to %
\[\Ker\qty((\pi\otimes \id)_\ast\colon F^{\Zn}_{\ast-1}(C\otimes B)\to F^{\Zn}_{\ast-1}(\mathscr{F}\otimes M_{n^\infty}\otimes B))\]%
via the map $F^{\Zn}_\ast(A\otimes B)\to F^{\Zn}_{\ast-1}(C\otimes B)$, and this is isomorphic to %
\[\Ker\qty(\pi_\ast\otimes \id\colon F^{\Zn}_\ast(C)\otimes_{\mathbb{Z}[\zeta_{n}]}F^{\Zn}_{\ast-1}(B)\to F^{\Zn}_\ast(\mathscr{F}\otimes M_{n^\infty})\otimes_{\mathbb{Z}[\zeta_{n}]}F^{\Zn}_{\ast-1}(B))\] %
via $\omega_C$. Since $F^{\Zn}_\ast(\mathscr{F}\otimes M_{n^\infty})$ is flat, $\text{Tor}^{\mathbb{Z}[\zeta_{n}]}_1(F^{\Zn}_\ast(\mathscr{F}\otimes M_{n^\infty}),F^{\Zn}_{\ast-1}(B))$ vanishes and $\Ker(\pi_\ast\otimes \id)$ is isomorphic to $\text{Tor}^{\mathbb{Z}[\zeta_{n}]}_1(F^{\Zn}_\ast(A),F^{\Zn}_{\ast-1}(B))$. Hence, 
\[\Cok\omega_A\cong\text{Tor}^{\mathbb{Z}[\zeta_{n}]}_1(F^{\Zn}_\ast(A),F^{\Zn}_{\ast-1}(B)).\]
This completes the proof. \end{proof}

\section{$K$-theory of strongly self-absorbing actions}

The purpose of this section is to compute the ring structure of the equivariant $K$-theory
of strongly self-absorbing action. Before starting the computation, we see the following theorem. It will play a crucial role in the discussion of Subsection 5.1 and 5.3. 

\begin{thm}\label{p-div}Let $G$ be a $p$-group and $A$ be an object in $KK^G$. Then $K^G_\ast(A)$ is uniquely $p$-divisible if $K_\ast(A)$ is uniquely $p$-divisible. \end{thm}

\begin{proof}Assume $\abs{G}=p^m$. We prove it by induction of $m$. The case $m=0$ is trivial. \par
Suppose that the statement holds for groups of order $p^{m-1}$. Since the center of a $p$-group is nontrivial, $G$ has a normal subgroup $N$ that is isomorphic to $\mathbb{Z}/p$. By \cite[Proposition 1]{Green}, there is a twisted action of $G/N$ on $A\rtimes N$ with $A\rtimes G\cong (A\rtimes N)\rtimes G/N$. This twisted action is Morita equivalent to an action of $G/N$ on $(A\rtimes N)\otimes\mathcal{K}$ by \cite[Theorem 3.4]{Twisted}. Thus %
\[K^G_\ast(A)\cong K_\ast(A\rtimes G)\cong K_\ast(((A\rtimes N)\otimes\mathcal{K})\rtimes G/N)\cong K^{G/N}_\ast((A\rtimes N)\otimes\mathcal{K}).\]%
Since $\abs{G/N}=p^{m-1}$, the induction hypothesis implies that $K^{G/N}_\ast((A\rtimes N)\otimes\mathcal{K})$ is uniquely $p$-divisible if $K_\ast((A\rtimes N)\otimes\mathcal{K})$ is uniquely $p$-divisible. It is isomorphic to $K^N_\ast(A)$. By \cite[Theorem 7.2]{Meyer}, this group is uniquely $p$-divisible if $K_\ast(A)$ is uniquely $p$-divisible. \end{proof}

\begin{cor}\label{eq of p-div}Let $G$ be a $p$-group and $A$ be an object in the $G$-equivariant bootstrap class. Then $K_\ast(A)$ is uniquely $p$-divisible if and only if $A$ has uniquely $p$-divisible $K$-thoery if and only if $p\cdot\id\in KK^G(A,A)$ is invertible. \end{cor}
\begin{proof}By Theorem \ref{p-div}, $K_\ast(A)$ is uniquely $p$-divisible if and only if $A$ has uniquely $p$-divisible $K$-thoery. By Proposition \ref{bootstrap}, $A$ has uniquely $p$-divisible $K$-theory if and only if $p\cdot\id\in KK^G(A,A)$ is invertible. \end{proof}

The following corollary is a powerful tool to compute $K^G_\ast(D)$ when $K_\ast(D)$ is not uniquely $p$-divisible. 

\begin{cor}\label{embedding}Let $G$ be a $p$-group and $D$ be a strongly self-absorbing $G$-$\Cstar$-algebra. Suppose that $D$ is in the $G$-equivariant bootstrap class $\mathcal{B}_G$ and $K_\ast(D)$ is not uniquely $p$-divisible. \begin{itemize}
\item[$(1)$] For the unital inclusion $\nu\colon\C\to D$, its mapping cone $C_\nu$ has uniquely $p$-divisible $K$-theory. 
\item[$(2)$] The multiplication by $p$ on $K^G_\ast(D)$ is injective. Consequently, the unital inclusion $\mathbb{C}\hookrightarrow M_{p^\infty}$ induces an injection $K^G_\ast(D)\hookrightarrow K^G_\ast(D\otimes M_{p^\infty})$. \end{itemize}\end{cor}
\begin{proof}\begin{itemize}
\item[(1)]By \cite[Proposition 5.1]{Toms and Winter}, $K_0(D)$ is a localization $\mathbb{Z}_T$ of $\mathbb{Z}$ by the set $T$ consisting of all invertible integers in $K_0(D)$. The map $\nu_\ast\colon K_0(\mathbb{C})\to K_0(D)$ is identified with the natural inclusion $\mathbb{Z}\hookrightarrow\mathbb{Z}_T$. The 6-term exact sequence obtained from the exact triangle %
\[C_\nu\rightarrow \mathbb{C}\overset{\nu}{\rightarrow} D\rightarrow\Sigma^{-1}C_\nu\]%
shows that $K_0(C_\nu)=0$ and $K_1(C_\nu)\cong\mathbb{Z}_T/\mathbb{Z}$. Take $a/b\in\mathbb{Z}_T$ $(a\in\mathbb{Z},\ b\in T)$. Since $\mathbb{Z}_T$ is not uniquely $p$-divisible, $b$ is coprime to $p$. Thus there are integers $k,l$ such that $k$ is positive and $p^k-1=lb$ holds. Now we have the formula %
\[p^k\frac{a}{b}=\frac{a}{b}+la.\]%
It implies that $\mathbb{Z}_T/\mathbb{Z}\cong K_1(C_\nu)$ is uniquely $p$-divisible. Since $\C$ and $D$ are in $\mathcal{B}_G$, so is $C_\nu$. Therefore, $C_\nu$ has uniquely $p$-divisible $K$-theory by Corollary \ref{eq of p-div}. 
\item[(2)]Consider the exact sequence %
\[K^G_\ast(C_\nu)\to K^G_\ast(\C)\to K^G_\ast(D)\to K^G_{\ast-1}(C_\nu).\]%
$K^G_\ast(C_\nu)$ is uniquely $p$-divisible by (1). Since the multiplication by $p$ is injective on $K^G_\ast(\mathbb{C})$,  it is also injective on $K^G_\ast(D)$ by the four lemma.  \end{itemize}\end{proof}

\subsection{Actions of cyclic $p$-groups}

In this subsection we prove the following theorem. 

\begin{thm}\label{cyclic case}Let $G$ be a cyclic $p$-group and $D$ be a strongly self-absorbing $G$-$\Cstar$-algebra that belongs to the $G$-equivariant bootstrap class. Then $K^G_1(D)=0$, and $K^G_0(D)$ is ring isomorphic to $0$ or a localization of the representation ring $R(G)$. Moreover, if $K_0(D)$ is not uniquely $p$-divisible, then the unital inclusion $\C\hookrightarrow D$ induces an injective ring homomorphism $R(G)\hookrightarrow K^G_0(D)$.\end{thm}

Throughout this subsection $G$ denotes the group $\mathbb{Z}/p^m$ ($p$ is a prime number and $m$ is a non-negative integer), $H$ denotes the subgroup $pG\cong\mathbb{Z}/p^{m-1}$ and $D$ denotes a strongly self-absorbing $G$-$\Cstar$-algebra which is in the $G$-equivariant bootstrap class $\mathcal{B}_G$. We prove Theorem \ref{cyclic case} by induction on $m$. The case $m=0$ is the non-equivariant case: this is Toms and Winter's work \cite[Proposition 5.1]{Toms and Winter}. In what follows in this subsection, we assume the following induction hypothesis. 

\begin{hyp}\label{induction hypothesis} Theorem \ref{cyclic case} holds for cyclic groups of order $p^{m-1}$. \end{hyp}%
Let $\alpha^H_{01}$ be the element of $KK^G(\mathbb{C},C(G/H))$ given by the unital inclusion $\mathbb{C}\hookrightarrow C(G/H)$ and $\alpha^H_{10}$ be the element of $KK^G(C(G/H),\mathbb{C})$ given by the composition of the unital inclusion $C(G/H)\hookrightarrow B(l_2(G/H))$ with the Morita equivalence $B(\ell_2(G/H))\sim\mathbb{C}$. We first compute $\alpha^H_{10}\otimes \alpha^H_{01}$ and $\alpha^H_{01}\otimes\alpha^H_{10}$. Let $s_H\in KK^G(C(G/H),C(G/H))$ be the element given by the homomorphism $C(G/H)\ni f\mapsto f(\cdot+1)\in C(G/H)$ and $t\in R(G)$ be the element given by the irreducible representation %
\[G=\mathbb{Z}/p^{m}\ni n\mapsto\exp (\frac{2n\pi i}{p^{m}})\in U(1)\]%
of $G$. Then the representation ring $R(G)$ equals $\mathbb{Z}[t]/(t^{p^m}-1)$. It has the $p^m$-th cyclotomic polynomial $\Phi_{p^m}(t)=1+t^{p^{m-1}}+\cdots +t^{(p-1)p^{m-1}}$. \par%
The behavior of elements $\alpha^H_{01}, \alpha^H_{10}$ has been studied in \cite{Meyer} when $G=\mathbb{Z}/p$. Let $\pullback\colon KK^{\mathbb{Z}/p}\to KK^G$ be the functor that regards an action of $\mathbb{Z}/p$ as an action of $G$ via the natural projection $G\to G/H\cong\mathbb{Z}/p$. Namely, $\pullback$ is the pullback functor via $G\to \mathbb{Z}/p$. It preserves Kasparov products. If $f\colon A\to B$ is a $\mathbb{Z}/p$-equivariant $\ast$-homomorphism and $C_f$ is its mapping cone, then clearly %
\[\pullback(\Sigma B)\to \pullback(C_f)\to \pullback(A)\overset{\pullback(f)}{\to} \pullback(B)\]%
is also a mapping cone sequence. Hence, the functor $\pullback$ preserves exact triangles. Define $\alpha_{01}\in KK^{\mathbb{Z}/p}(\C,C(\mathbb{Z}/p))$, $\alpha_{10}\in KK^{\mathbb{Z}/p}(C(\mathbb{Z}/p),\C)$ and $s\in KK^{\mathbb{Z}/p}(C(\mathbb{Z}/p),C(\mathbb{Z}/p))$ similarly to $\alpha^H_{01}$, $\alpha^H_{10}$ and $s_H$, respectively. Then we have $\pullback(\alpha_{01})=\alpha^H_{01}$, $\pullback(\alpha_{10})=\alpha^H_{10}$ and $\pullback(s)=s_H$. 

\begin{lem}\label{composition of alpha}$\alpha^H_{10}\otimes\alpha^H_{01}=1+s_H+\cdots +s_H^{p-1}\in KK^G(C(G/H),C(G/H))$, and $\alpha^H_{01}\otimes\alpha^H_{10}=\Phi_{p^m}(t)\in R(G)$. \end{lem}
\begin{proof} Let $t_0\in R(\mathbb{Z}/p)$ be the element given by the irreducible representation %
\[\mathbb{Z}/p\ni n\mapsto \exp(\frac{2n\pi i}{p})\in U(1)\]%
of $\mathbb{Z}/p$. By \cite[Lemma 4.10 (4)]{Meyer}, $\alpha_{10}\otimes\alpha_{01}=1+s+\cdots +s^{p-1}$ and $\alpha_{01}\otimes\alpha_{10}=1+t_0+\cdots +t_0^{p-1}$. By the definition of $\pullback\colon KK^{\mathbb{Z}/p}\to KK^G$, $\pullback(t_0)\in R(G)$ is represented by the irreducible representation %
\[G=\mathbb{Z}/p^m\ni n\mapsto \exp(\frac{2n\pi i}{p})\in U(1).\]%
It implies $\pullback(t_0)=t^{p^{m-1}}$. Therefore %
\begin{align*}&\alpha^H_{10}\otimes\alpha^H_{01}=\pullback(1+s+\dots +s^{p-1})=1+s_H+\dots +s_H^{p-1}, \\
&\alpha^H_{01}\otimes\alpha^H_{10}=\pullback(1+t_0+\dots +t_0^{p-1})=1+t^{p^{m-1}}+\dots +t^{(p-1)p^{m-1}}=\Phi_{p^m}(t). \end{align*} \end{proof}

\begin{rem}\label{alpha02}We continue the argument of \cite{Meyer} in more detail. Let $C_u$ be the mapping cone of $\alpha_{01}$ (this is denoted by $D$ in \cite{Meyer}). There is the mapping cone sequence %
\[\Sigma C(\mathbb{Z}/p)\longrightarrow C_u\overset{\alpha_{20}}{\longrightarrow}\C\overset{\alpha_{01}}{\longrightarrow}C(\mathbb{Z}/p).\]%
Also, by \cite[Proposition 5.5]{Meyer}, there is some $\alpha_{02}\in KK^{\mathbb{Z}/p}(\C,C_u)$ and an exact triangle %
\[\Sigma\C\overset{\Sigma\alpha_{02}}{\longrightarrow}\Sigma C_u\longrightarrow C(\mathbb{Z}/p)\overset{\alpha_{10}}{\longrightarrow}\C.\]%
\cite[Lemma 5.52]{Meyer} shows that $\alpha_{02}$ and $\alpha_{20}$ satisfy $\alpha_{02}\otimes\alpha_{20}=1-t_0\in R(\mathbb{Z}/p)$ under $R(\mathbb{Z}/p)\cong\mathbb{Z}/(t_0^p-1)$. Thus, setting $\alpha^H_{02}\coloneqq\pullback(\alpha_{02})\in KK^G(\C,C_u)$ and $\alpha^H_{20}\coloneqq\pullback(\alpha_{20})\in KK^G(C_u,\C)$, we have $\alpha^H_{02}\otimes\alpha^H_{20}=\pullback(1-t_0)=1-t^{p^{m-1}}$. It will be used in the latter of this subsection. \end{rem}

$s_H$ acts on $KK^G(C(G/H),D)$ by the Kasparov product. By Proposition \ref{reciprocity} (2), the map %
\[\iota\colon \C\ni\lambda\mapsto\lambda\delta_{H}\in C(G/H)\]%
gives the isomorphism $KK^G(C(G/H),D)\overset{\cong}{\to}K^H_0(D),\ \widetilde{x}\mapsto\iota\otimes\Res^H_G(\widetilde{x})$. Via this isomorphism, $s_H$ acts on $K^H_0(D)$. Given $x,y\in K^H_0(D)$, let $\widetilde{x},\widetilde{y}$ be the elements of $KK^G(C(G/H),D)$ with $\iota\otimes\Res^H_G(\widetilde{x})=x$ and $\iota\otimes\Res^H_G(\widetilde{y})=y$, respectively. Then %
\[(\iota\otimes \Res^H_G(\widetilde{x}))\cupprod(\iota\otimes \Res^H_G(\widetilde{y}))=(\iota\cupprod\iota)\otimes\Res^H_G(\widetilde{x}\cupprod \widetilde{y})=\iota\otimes\Res^H_G(\mu\otimes(\widetilde{x}\cupprod \widetilde{y}))\]%
in $K^H_0(D\otimes D)$, where $\mu\in KK^G(C(G/H),C(G/H)\otimes C(G/H))$ is the element given by $C(G/H)\to C(G/H)\otimes C(G/H),\ \delta_{gH}\to\delta_{gH}\otimes\delta_{gH}$. This formula implies that the multiplication $xy=(\id\otimes\nu)_\ast^{-1}(x\cupprod y)\in K^H_0(D)$ corresponds to $(\id\otimes\nu)_\ast^{-1}(\mu\otimes (\widetilde{x}\cupprod\widetilde{y}))\in KK^G(C(G/H),D)$. Also, the action of $s_H$ maps $x$ to $s_H(x)\coloneqq\iota\otimes\Res^H_G(s_H\otimes\widetilde{x})$. Since $\mu\otimes(s_H\cupprod s_H)=s_H\otimes\mu$,  %
\begin{align*}(\id\otimes\nu)_\ast(s_H(x)\cdot s_H(y))&=(\iota\otimes \Res^H_G(s_H\otimes \widetilde{x}))\cupprod(\iota\otimes \Res^H_G(s_H\otimes \widetilde{y})) \\
&=\iota\otimes\Res^H_G(s_H\otimes\mu\otimes(\widetilde{x}\cupprod \widetilde{y})) \\
&=(\id\otimes\nu)_\ast(s_H(xy)).\end{align*}%
Thus, $s_H$ acts on $K^H(D)$ as a ring automorphism. In the same way, $s_H$ also gives a ring automorphism of $R(H)$. 

\begin{lem}\label{s=1}$s_H$ acts as the identity on $R(H)$ and on $K^H_0(D)$. Consequently, $\alpha^H_{10}\otimes\alpha^H_{01}$ acts as the multiplication by $p$ on $R(H)$ and on $K^H_0(D)$. \end{lem}
\begin{proof}By Hypothesis \ref{induction hypothesis}, $K^H_0(D)$ is isomorphic to a localization of $R(H)$. Thus, if $s_H$ acts trivially on $R(H)$, then $s_H$ also acts trivially on $K^H_0(D)$ by Remark \ref{localization}. Let $t_H\in R(H)$ be the irreducible representation %
\[H\cong\mathbb{Z}/p^{m-1}\ni n\mapsto\exp (\frac{2n\pi i}{p^{m-1}})\in U(1)\]%
of $H$. Since $R(H)$ is generated by $t_H$ as a ring, the action of $s_H$ on $R(H)$ is trivial if $s_H$ fixes $t_H$. The isomorphism $R(H)=KK^H(\mathbb{C},\mathbb{C})\cong KK^G(C(G/H),\mathbb{C})$ of Proposition \ref{reciprocity} (2) maps $t_H$ to the $G$-Hilbert space%
\[E=\qty{f\in \ell_2 (G)\mid f(g+n)=\exp (\frac{2n\pi i}{p^{m-1}})f(g),\ g\in G,\ n\in H}\]%
with a left $C(G/H)$-action by multiplication. We write it as $\widetilde{t_H}$. The element $s_H\otimes\widetilde{t_H}\in KK^G(C(G/H),\mathbb{C})$ is represented by the same Hilbert space $E$ but with the different left $C(G/H)$-action %
\[(\vp\cdot f)(g)=\vp(g+1)f(g)\quad (\vp\in C(G/H),\ f\in E,\ g\in G).\]%
The unitary $E\to E$, $f\mapsto f(\cdot -1)$ gives an equivalence of $s_H\otimes\widetilde{t_H}$ and $\widetilde{t_H}$. This completes the proof. \end{proof}

\begin{rem}\label{alpha and res-ind}By Proposition \ref{reciprocity} (2), the isomorphism $KK^G(C(G/H),A)\cong K^H_0(A)$ is given by the $H$-equivariant map %
\[\iota\colon\C\ni z\mapsto z\delta_{H}\in C(G/H).\]%
We also write the element of $KK^H(\C,C(G/H))$ given by $\iota$ as $\iota$. The composition of $\iota$ with the diagonal inclusion $C(G/H)\to B(\ell_2(G/H))$ equals the inclusion of $\C$ into a one-dimensional corner of $B(\ell_2(G/H))$. Hence %
\[\iota\otimes\Res^H_G(\alpha^H_{10})=\id_{\C}\in KK^H(\C ,\C),\]%
that is, the map $K^G_0(A)\to KK^G(C(G/H),A)$ induced by $\alpha^H_{10}$ is identified with the restriction $\res^H_G\colon K^G_\ast(A)\to K^H_\ast(A)$. Also, $I^G_H\colon K^H_\ast(A)\to K^G_\ast(A)$ is, by definition, just the composition of the isomorphism $K^H_\ast(A)\overset{\cong}\to KK^G(C(G/H),A)$ of Proposition \ref{reciprocity} (2) with the map $(\alpha^H_{01})^\ast\colon KK^G_\ast(C(G/H),A)\to K^G_\ast(A)$. Hence, $(\alpha^H_{01})^\ast$ is identified with the induction $I^G_H$. In particular, by Lemma \ref{composition of alpha} and Lemma \ref{s=1}, %
\begin{align}\label{indres=Phi}&I^G_H\circ\res^H_G=\Phi_{p^m}\colon K^G_\ast(D)\to K^G_\ast(D),\quad x\mapsto\Phi_{p^m}x, \\
\label{resind=p}&\res^H_G\circ I^G_H=p\colon K^H_0(D)\to K^H_0(D),\quad y\mapsto py.\end{align} \end{rem}

\begin{cor}\label{part comming from H}If $K_\ast(D)$ is uniquely $p$-divisible, then the restriction map $K^G_\ast(D)\to K^H_\ast(D)$ induces an isomorphism \[\Phi_{p^m}K^G_0(D)\cong K^H_0(D).\]\end{cor}
\begin{proof}Since the representation ring $R(H)$ is isomorphic to $\mathbb{Z}[t]/(t^{p^{m-1}}-1)$, we have $\res^H_G(\Phi_{p^m})=p$. If $\res^H_G(\Phi_{p^m} x)=0$ for $x\in K^G_0(D)$, then $p\res^H_G(x)=0$, so $\res^H_G(x)=0$ because $D$ has uniquely $p$-divisible $K$-theory by Corollary \ref{eq of p-div}. It implies $\Phi_{p^m}x=0$ by (\ref{indres=Phi}). Hence, the map $\res^H_G\colon \Phi_{p^m}K^G_0(D)\to K^H_0(D)$ is injective. To see surjectivity, take $y\in K^H_0(D)$. By $\res^H_G(\Phi_{p^m})=p$ and (\ref{resind=p}), %
\[\res^H_G(p^{-2}\Phi_{p^m} I^G_H(y))=p^{-2}\cdot p\cdot \res^H_G I^G_H(y)=y\]%
holds. This proves the surjectivity. \end{proof}
For the moment, we focus on the case that $K_\ast(D)$ is uniquely p-divisible. Since $D$ has uniquely $p$-divisible $K$-theory by Corollary \ref{eq of p-div}, $K^G_\ast(D)$ is decomposed as %
\[K^G_\ast(D)=\qty(1-\frac{\Phi_{p^m}}{p})K^G_\ast(D)\oplus \frac{\Phi_{p^m}}{p}K^G_\ast(D)\]%
by the idempotent $\Phi_{p^m}/p\in R(G)_P$, where $P$ is the multiplicative subset of $\mathbb{Z}$ generated by $1$ and $p$. The direct summand $(1-\Phi_{p^m}/p)K^G_\ast(D)$ equals $F^G_\ast(D)$ because this is the kernel of the multiplication by $\Phi_{p^m}$. (In fact, $1-\Phi_{p^m}/p$ equals $\psi_{p^m}$ defined in Remark \ref{decomposition by psi}.) In particular, $F^G_\ast(D)$ is a unital ring with the unity $1-\Phi_{p^m}/p$. Another component $(\Phi_{p^m}/p)K^G_\ast(D)$ is isomorphic to $K^H_\ast(D)$ by Corollary \ref{part comming from H}. Hence %
\begin{align}\label{K decompose}K^G_\ast(D)\cong F^G_\ast(D)\oplus K^H_\ast(D).\end{align}%
Similarly we have an isomorphism of rings %
\begin{align}\label{R decompose}R(G)_P\cong \mathbb{Z}[\zeta_{p^m},1/p]\oplus R(H)_P.\end{align}

\begin{thm}\label{F is a localization}If $K_\ast(D)$ is uniquely $p$-divisible, then $F^G_1(D)=0$, and $F^G_0(D)$ is zero or isomorphic to a localization of $\mathbb{Z}[\zeta_{p^m}]$ as a ring. \end{thm}
\begin{proof}$D$ is uniquely $p$-divisible by Corollary \ref{eq of p-div}. By Theorem \ref{Kunneth}, we get a K\"{u}nneth formula %
\[F^G_\ast(D)\otimes_{\mathbb{Z}[\zeta_{p^m}]} F^G_\ast(D)\rightarrowtail F^G_\ast(D\otimes D)\twoheadrightarrow\text{Tor}^{\mathbb{Z}[\zeta_{p^m}]}_1(F^G_\ast(D),F^G_{\ast+1}(D)).\]%
The first homomorphism is the ring multiplication of $F^G_\ast(D)$ under the isomorphism $D\otimes D\cong D$. Thus, it is also surjective and $\text{Tor}^{\mathbb{Z}[\zeta_{p^m}]}_1(F^G_\ast(D),F^G_{\ast+1}(D))$ vanishes. In particular we get an isomorphism %
\[\qty(F^G_0(D)\otimes_{\mathbb{Z}[\zeta_{p^m}]} F^G_1(D))\oplus\qty(F^G_1(D)\otimes_{\mathbb{Z}[\zeta_{p^m}]} F^G_0(D))\cong F^G_1(D),\]%
but this map is surjective even if we restrict its domain to one of direct summands. Hence $F^G_1(D)$ must be zero. 

Suppose $F^G_0(D)$ is nonzero. By Lemma \ref{Tor(M,M)}, $F^G_0(D)$ is a flat $\mathbb{Z}[\zeta_{p^m}]$-module and includes $\mathbb{Z}[\zeta_{p^m}]$ as a subring. The flatness also implies the inclusion $F^G_0(D)\subset F^G_0(D)\otimes_{\mathbb{Z}[\zeta_{p^m}]}\mathbb{Q}[\zeta_{p^m}]$. The above K\"{u}nneth formula gives an isomorphism of $\mathbb{Q}[\zeta_{p^m}]$-vector spaces and of rings%
\[\qty(F^G_0(D)\otimes_{\mathbb{Z}[\zeta_{p^m}]}\mathbb{Q}[\zeta_{p^m}])\otimes_{\mathbb{Q}[\zeta_{p^m}]}\qty(F^G_0(D)\otimes_{\mathbb{Z}[\zeta_{p^m}]}\mathbb{Q}[\zeta_{p^m}])\cong F^G_0(D)\otimes_{\mathbb{Z}[\zeta_{p^m}]}\mathbb{Q}[\zeta_{p^m}].\]%
Take a basis $(e_i)_{i\in I}$ of the $\mathbb{Q}[\zeta_{p^m}]$-vector space $F^G_0(D)\otimes_{\mathbb{Z}[\zeta_{p^m}]}\mathbb{Q}[\zeta_{p^m}]$ so that $e_{i_0}$ is the unit for some $i_0$ in $I$. The left hand side of the above formula has a basis $\qty(e_i\otimes e_j)_{(i,j)\in I\times I}$, but $e_i\otimes e_{i_0}-e_{i_0}\otimes e_i$ belongs to the kernel of this isomorphism for any $i$ in $I$. Hence, we get $I=\qty{i_0}$ and $F^G_0(D)\otimes_{\mathbb{Z}[\zeta_{p^m}]}\mathbb{Q}[\zeta_{p^m}]\cong\mathbb{Q}[\zeta_{p^m}]$. The flatness of $\mathbb{Q}[\zeta_{p^m}]$ implies $F^G_0(D)\subset\mathbb{Q}[\zeta_{p^m}]$. Now $F^G_0(D)$ is an overring of $\mathbb{Z}[\zeta_{p^m}]$. By Remark \ref{Z[zeta] is Dedekind} and Theorem \ref{overring}, $F^G_0(D)$ can be written as a localization of $\mathbb{Z}[\zeta_{p^m}]$. \end{proof}

\begin{cor}\label{cyclic p-div case}Theorem \ref{cyclic case} holds if $G=\mathbb{Z}/p^m$ and $K_\ast(D)$ is uniquely $p$-divisible. \end{cor}
\begin{proof}Let $P\subset\mathbb{Z}$ be the multiplicative subset generated by $1$ and $p$. Recall the isomorphisms of (\ref{K decompose}) and (\ref{R decompose}): %
\begin{align*}K^G_\ast(D)\cong F^G_\ast(D)\oplus K^H_\ast(D), \\
R(G)_P\cong \mathbb{Z}[\zeta_{p^m},1/p]\oplus R(H)_P.\end{align*}%
$K^G_1(D)=0$ since $F^G_1(D)=0$ and $K^H_1(D)=0$. Theorem \ref{F is a localization} and the Hypothesis \ref{induction hypothesis} implies that $K^G_0(D)$ is isomorphic to a localization of the ring $R(G)_P$. \end{proof}

We now consider the case that $K_\ast(D)$ is not uniquely $p$-divisible. This situation is more complicated because we cannot use Theorem \ref{Kunneth} directly. However, by Corollary \ref{embedding} (2), we can embed $D$ into another strongly self-absorbing $G$-$\Cstar$-algebra that has uniquely $p$-divisible $K$-theory. 

\begin{cor}\label{K1 vanish (cyclic not p-div)}If $K_\ast(D)$ is not uniquely $p$-divisible, then $K^G_1(D)=0$.\end{cor}
\begin{proof}This follows from Corollary \ref{cyclic p-div case} and Corollary \ref{embedding} (2). \end{proof}

\begin{lem}\label{[f] invertible}Suppose that $K_\ast(D)$ is not uniquely $p$-divisible. Let $[\cdot]$ be the natural projection %
\[R(H)=\mathbb{Z}[t]/(t^{p^{m-1}}-1)\ni f\mapsto [f]\in\mathbb{Z}[t]/(t^{p^{m-1}}-1,p).\]%
If $f\in R(H)$ is invertible in the ring $K^H_0(D)$, then $[f]\in\mathbb{Z}[t]/(t^{p^{m-1}}-1,p)$ is invertible. Thus the map $[\cdot]$ extends to a map from $K^H_0(D)$ to $\mathbb{Z}[t]/(t^{p^{m-1}}-1,p)$. \end{lem}
\begin{proof}Note that $K^H_0(D)$ includes $R(H)$ as a subring by Hypothesis \ref{induction hypothesis}. Since $f$ is invertible in $K^H_0(D)$, the element $f(1)\in\mathbb{Z}$ is invertible in $K_0(D)$ and so $[f(1)]\in\mathbb{Z}/p$ is invertible. If we write $f=\sum_{i=0}^{p^{m-1}-1}a_it^i$, then %
\[\qty[f^{p^{m-1}}]=\qty[\qty(\sum_{i=0}^{p^{m-1}-1}a_it^i)^{p^{m-1}}]=\qty[\sum_{i=0}^{p^{m-1}-1}a_i^{p^{m-1}}]=\qty[\sum_{i=0}^{p^{m-1}-1}a_i]=\qty[f(1)]\]%
holds in $\mathbb{Z}[t]/(t^{p^{m-1}}-1,p)$. Thus $[f]$ is invertible. Since $K^H_0(D)$ is a localization of $R(H)$ by Hypothesis \ref{induction hypothesis}, the last part of the statement follows from Remark \ref{localization}. \end{proof}

\begin{lem}\label{surjectivity of restriction}If $K_\ast(D)$ is not uniquely $p$-divisible, then the restriction map $K^G_0(D)\to K^H_0(D)$ is surjective. \end{lem}
\begin{proof}Its image contains $p\cdot K^H_0(D)$ by Lemma \ref{s=1} and Remark \ref{alpha and res-ind}, and also contains $R(H)$ since the restriction map $R(G)\to R(H)$ is surjective. Hypothesis \ref{induction hypothesis} implies that any element of $K^H_0(D)$ can be written as a fraction $f_0/f_1$ for some $f_0, f_1\in R(H)$ such that $f_1$ is invertible in $K^H_0(D)$. By Lemma \ref{[f] invertible}, we can take $h_0, h_1\in R(H)$ so that $h_0f_1-1=ph_1$ holds. Now we have %
\[\frac{f_0}{f_1}=h_0f_0-ph_1\frac{f_0}{f_1}\in R(H)+p\cdot K^H_0(D)\subset\Im\qty(\res^H_G\colon K^G_\ast(D)\to K^H_\ast(D)).\]%
 \end{proof}

Let $\alpha^H_{02},\alpha^H_{20}$ be the elements defined in Remark \ref{alpha02}. There are two long exact sequences %
\begin{align*}KK^G_1(C(G/H),D)\to KK^G(C_u,D)\overset{(\alpha^H_{02})^\ast}{\to} K^G(D)\overset{(\alpha^H_{10})^\ast}{\to} KK^G(C(G/H),D), \\
KK^G_1(C(G/H),D)\leftarrow KK^G(C_u,D)\overset{(\alpha^H_{20})^\ast}{\leftarrow} K^G(D)\overset{(\alpha^H_{01})^\ast}{\leftarrow} KK^G(C(G/H),D).\end{align*} %
Now $(\alpha^H_{10})^\ast$ is surjective by Remark \ref{alpha and res-ind} and Lemma \ref{surjectivity of restriction}. $(\alpha^H_{01})^\ast$ is injective by $(\alpha^H_{10})^\ast\circ(\alpha^H_{01})^\ast=p$ (Lemma \ref{s=1}) and Corollary \ref{embedding} (2). Also, by Proposition \ref{reciprocity} (2) and Hypothesis \ref{induction hypothesis}, $KK^G_1(C(G/H),D)\cong K^H_1(D)=0$. It implies that $(\alpha^H_{02})^\ast$ is injective and $(\alpha_{20})^\ast$ is surjective. Hence, we obtain the following short exact sequences %
\begin{align*}&KK^G(C_u,D)\overset{(\alpha^H_{02})^\ast}{\rightarrowtail}K^G_0(D)\overset{(\alpha^H_{10})^\ast}{\twoheadrightarrow}KK^G(C(G/H),D),\\
&KK^G(C_u,D)\overset{(\alpha^H_{20})^\ast}{\twoheadleftarrow}K^G_0(D)\overset{(\alpha^H_{01})^\ast}{\leftarrowtail}KK^G(C(G/H),D).\end{align*}%
By Remark \ref{alpha02}, $\alpha^H_{02}\otimes\alpha^H_{20}=1-t^{p^{m-1}}\in R(G)\cong \mathbb{Z}[t]/(t^{p^m}-1)$. Combining it with the above exact sequences, on $K^G_0(D)$, %
\begin{align}\label{Im Phi}&\Ker(1-t^{p^{m-1}})=\Ker(\alpha^H_{20})^\ast=\Im(\alpha^H_{01})^\ast=\Im(\alpha^H_{01})^\ast(\alpha^H_{10})^\ast=\Im\Phi_{p^m}, \\
\label{Ker Phi}&\Im(1-t^{p^{m-1}})=\Im(\alpha^H_{02})^\ast=\Ker(\alpha^H_{10})^\ast=\Ker(\alpha^H_{01})^\ast(\alpha^H_{10})^\ast=\Ker\Phi_{p^m},\end{align}%
where the last equalities of both equations follow from Lemma \ref{composition of alpha}. Thus there is an isomorphism of $\mathbb{Z}[\zeta_{p^m}]$-modules %
\[F^G_0(D)=\Ker\Phi_{p^m}=\Im(1-t^{p^{m-1}})\cong K^G_0(D)/\Ker(1-t^{p^{m-1}})=K^G_0(D)/\Im\Phi_{p^m},\]%
where the middle isomorphism is given by the multiplication map %
\[1-t^{p^{m-1}}\colon K^G_0(D)/\Ker(1-t^{p^{m-1}})\to \Im(1-{t^{p^{m-1}}}).\] %
Note that this isomorphism does not preserve the multiplicative structure. Indeed, $F^G_0(D)$ is not unital when $D$ does not have uniquely $p$-divisible $K$-theory. However, the module $K^G_0(D)/\Im\Phi_{p^m}$ always forms a unital ring. 

\begin{lem}\label{includes integers ring}If $K_\ast(D)$ is not uniquely $p$-divisible, then $K^G_0(D)/\text{Im}\Phi_{p^m}$ contains 
	$\mathbb{Z}[\zeta_{p^m}]$ as a subring. \end{lem}
\begin{proof}A similar argument as above gives $F^G_0(\mathbb{C})\cong\mathbb{Z}[\zeta_{p^m}]=R(G)/\operatorname{Im}\Phi_{p^m}$. Let $\nu\colon\C\to D$ be the unital inclusion. $\nu_\ast\colon R(G)\to K^G_0(D)$ induces a ring homomorphism $\mathbb{Z}[\zeta_{p^m}]\to K^G_0(D)/\operatorname{Im}\Phi_{p^m}$. Since the diagram %
\[\begin{tikzcd}\mathbb{Z}[\zeta_{p^m}] \arrow[r,"1-t^{p^{m-1}}"] \arrow[d] & F^G_0(\C) \arrow[d,"\nu_\ast"] \\
K^G_0(D)/\operatorname{Im}\Phi_{p^m} \arrow[r,"1-t^{p^{m-1}}"] & F^G_0(D)\end{tikzcd}\] %
commutes, it suffices to show that $\nu_\ast\colon F^G_0(\mathbb{C})\to F^G_0(D)$ is injective. If $F^G_0(D)\neq 0$, then $F^G_0(D\otimes M_{p^\infty})\neq 0$ by Corollary \ref{embedding} (2) and $F^G_0(D\otimes M_{p^\infty})$ includes $F^G_0(\mathbb{C})$ by Theorem \ref{F is a localization}. Since the map $F^G_0(\C)\to F^G_0(M_{p^\infty}\otimes D)$ factors through $F^G_0(D)$, $F^G_0(D)$ includes $F^G_0(\mathbb{C})$ if $F^G_0(D)$ is nonzero. \par %
Suppose $F^G_0(D)=0$. It implies $\text{Im}(1-t^{p^{m-1}})=0$, so $\Ker(1-t^{p^{m-1}})$ equals $K^G_0(D)$. Moreover, since $\Ker\Phi_{p^m}=F^G_0(D)=0$ and $\operatorname{Im}\Phi_{p^m}=\Ker(1-t^{p^{m-1}})=K^G_0(D)$, the multiplication by $\Phi_{p^m}$ on $K^G_0(D)$ is bijective. On the other hand, the equality $\Ker(1-t^{p^{m-1}})=K^G_0(D)$ gives that $\Phi_{p^m}=p$ on $K^G_0(D)$. It contradicts the assumption that $K^G_0(D)$ is not uniquely $p$-divisible. \end{proof}

Corollary \ref{embedding} (2) gives an inclusion %
\[K^G_0(D)/\text{Im}\Phi_{p^m}\subset K^G_0(D\otimes M_{p^\infty})/\text{Im}\Phi_{p^m}\cong F^G_0(D\otimes M_{p^\infty}).\]%
Hence by Theorem \ref{F is a localization} and Lemma \ref{includes integers ring} $K^G_0(D)/\text{Im}\Phi_{p^m}$ is an overring of $\mathbb{Z}[\zeta_{p^m}]$. It implies the following corollary. 

\begin{cor}\label{K/Im is a localization}If $K_\ast(D)$ is not uniquely $p$-divisible, then the ring $K^G_0(D)/\operatorname{Im}\Phi_{p^m}$ is isomorphic to a localizaton of $\mathbb{Z}[\zeta_{p^m}]$. \end{cor}

Recall that the $p^m$-th cyclotomic polynomial $\Phi_{p^m}$ is defined by $\Phi_{p^m}(t)=1+t^{p^{m-1}}+\cdots +t^{(p-1)p^{m-1}}$. The formula %
\[\Phi_{p^m}(t)-p=(t^{p^{m-1}}-1)(t^{(p-2)p^{m-1}}+2t^{{(p-3)}p^{m-1}}+\cdots +p-1)\]%
gives the natural projection $\mathbb{Z}[t]/(\Phi_{p^m}(t),p)\to\mathbb{Z}[t]/(t^{p^{m-1}}-1,p)$. By composing it with the projection $\mathbb{Z}[t]/(\Phi_{p^m}(t))\to\mathbb{Z}[t]/(\Phi_{p^m}(t),p)$, we obtain the map %
\[\pi\colon\mathbb{Z}[\zeta_{p^m}]=\mathbb{Z}[t]/(\Phi_{p^m}(t))\to\mathbb{Z}[t]/(\Phi_{p^m}(t),p)\to\mathbb{Z}[t]/(t^{p^{m-1}}-1,p).\]%
Note that the diagram %
\[\begin{tikzcd}R(G)=\mathbb{Z}[t]/(t^{p^m}-1) \arrow[r,"\res^H_G"] \arrow[d,twoheadrightarrow] & R(H)=\mathbb{Z}[t]/(t^{p^{m-1}}-1) \arrow[d,"{[\cdot]}"] \\
\mathbb{Z}[\zeta_{p^m}] \arrow[r,"\pi"] & \mathbb{Z}[t]/(t^{p^{m-1}}-1,p) \end{tikzcd}\]%
commutes, where $[\cdot]$ is the map defined in Lemma \ref{[f] invertible}. 

\begin{lem}\label{pi(f) invertivle}Suppose that $K_\ast(D)$ is not uniquely $p$-divisible. If $f\in\mathbb{Z}[\zeta_{p^m}]$ is invertible in $K^G_0(D)/\operatorname{Im}\Phi_{p^m}$, then $\pi(f)\in\mathbb{Z}[t]/(t^{p^{m-1}}-1,p)$ is also invertible. Thus the map $\pi$ extends to a map from $K^G_0(D)/\operatorname{Im}\Phi_{p^m}$ to $\mathbb{Z}[t]/(t^{p^{m-1}}-1,p)$. \end{lem}
\begin{proof}Let $\widetilde{f}, \widetilde{f^{-1}}\in K^G_0(D)$ be lifts of $f, f^{-1}\in K^G_0(D)/\text{Im}\Phi_{p^m}$, respectively. Since $\widetilde{f^{-1}}\widetilde{f}-1$ is in $\operatorname{Im}\Phi_{p^m}$, we may take $\widetilde{h}\in K^G_0(D)$ with $\widetilde{f^{-1}}\widetilde{f}-1=\Phi_{p^m}\widetilde{h}$. By restricting it to $K^H_0(D)$, we get %
\[\res^H_G(\widetilde{f^{-1}}\widetilde{f}-1)=p\cdot\res^H_G(\widetilde{h}).\]%
This equation and Lemma \ref{[f] invertible} imply $[\res^H_G(\widetilde{f^{-1}}\widetilde{f}-1)]=0$ in $\mathbb{Z}[t]/(t^{p^{m-1}}-1,p)$. Therefore $\pi(f)=[\res^H_G(\widetilde{f})]$ has the inverse $[\res^H_G(\widetilde{f^{-1}})]$. The last part of the statement follows from Corollary \ref{K/Im is a localization} and Remark \ref{localization}. \end{proof}

\begin{prp}\label{K includes R(G)}The unital inclusion $\nu\colon\mathbb{C}\to D$ gives an embedding $R(G)\hookrightarrow K^G_0(D)$. \end{prp}
\begin{proof}We already know that $\nu_\ast|_{F^G_0(\C)}\colon F^G_0(\C)\to K^G_0(D)$ is injective (Lemma \ref{includes integers ring}). Take $f\in\Ker\nu_\ast$. Then $(t^{p^{m-1}}-1)f\in\Ker\nu_\ast\cap F^G_\ast(D)=0$, so $f$ is in $\Ker(t^{p^{m-1}}-1)$. By (\ref{Im Phi}) and Remark \ref{alpha and res-ind}, %
\[\Ker(t^{p^{m-1}}-1)=\operatorname{Im}(\alpha^H_{01})^\ast=\operatorname{Im}\qty(I^G_H\colon R(H)\to R(G)).\] %
Hence, we can choose $h\in R(H)$ so that $f=I^G_H(h)$. (\ref{resind=p}) implies %
\[0=\res^H_G(\nu_\ast(f))=\nu_\ast(\res^H_G(f))=\nu_\ast(\res^H_G I^G_H(h))=\nu_\ast(p\cdot h)\]%
in $K^H_0(D)$. By Hypothesis \ref{induction hypothesis}, $\nu_\ast\colon R(H)\to K^H_0(D)$ is injective. Hence $p\cdot h=h=0$ in $R(H)$ and $f=I^G_H(h)=0$.  \end{proof}

\begin{lem}\label{cyclic lift} If $f\in\mathbb{Z}[\zeta_{p^m}]$ and $h\in R(H)$ satisfy $[f]=\pi(h)$, then there exists an element $\widetilde{f}\in R(G)$ such that $\widetilde{f}+\operatorname{Im}\Phi_{p^m}=f$ and $\res^H_G(\widetilde{f})=h$. Moreover, if $K_\ast(D)$ is not uniquely $p$-divisible, $f$ is invertible in $K^G_0(D)/\operatorname{Im}\Phi_{p^m}$ and $h$ is invertible in $K^H_0(D)$, then $\widetilde{f}$ is invertible in $K^G_0(D)$. \end{lem}

\begin{proof}Take $F\in R(G)$ with $F+\operatorname{Im}\Phi_{p^m}=f$.  Then $\pi(\res^H_G(F)-h)=0$ and so 
\[\res^H_G(F)-h=p\cdot h^\prime\]
for some $h^\prime\in R(H)$. Put $\widetilde{f_1}\coloneqq F-I^H_G(h^\prime)$. Since $\operatorname{Im} I^G_H=\operatorname{Im}\Phi_{p^m}$ by Remark \ref{alpha and res-ind} and (\ref{Im Phi}), we have $\widetilde{f}+\operatorname{Im}\Phi_{p^m}=f$. Also, by (\ref{resind=p}), $\res^H_G(\widetilde{f})=\res^H_G(F)-p\cdot h^\prime=h$. 

If $\widetilde{f^{-1}}\in K^G_0(D)$ is any lift of $f^{-1}\in K^G_0(D)/\operatorname{Im}\Phi_{p^m}$, then $\widetilde{f^{-1}}\widetilde{f}-1$ is in $\operatorname{Im}\Phi_{p^m}$. Remark \ref{alpha and res-ind} and (\ref{Im Phi}) implies that %
\[\widetilde{f^{-1}}\widetilde{f}-1=I^G_H(y) \]%
holds for some $y\in K^H_0(D)$. Then $\widetilde{f^{-1}}-I^G_H(y\cdot h^{-1})$ is the inverse of $\widetilde{f}$. Indeed, by Proposition \ref{Dell'ambrogio}, $I^G_H(y\cdot h^{-1})\cdot \widetilde{f}=I^G_H(y\cdot h^{-1}\cdot\res^H_G(\widetilde{f}))=I^G_H(y)$ and we have 
\[\left( \widetilde{f^{-1}}-I^G_H(y\cdot h^{-1}) \right) \widetilde{f}=1+I^G_H(y)-I^G_H(y)=1.\]\end{proof}

\begin{thm}\label{cyclic not p-div case}Theorem \ref{cyclic case} holds if $G=\mathbb{Z}/p^m$ and $K_\ast(D)$ is not uniquely $p$-divisible. \end{thm}
\begin{proof}Let $S=S_D$ be the set of invertible elements of $K^G_0(D)$. Thanks to Corollary \ref{K1 vanish (cyclic not p-div)} and Proposition \ref{K includes R(G)}, we only have to show that $K^G_0(D)$ is isomorphic to $R(G)_S$. Since $\nu_\ast\colon R(G)\to K^G_0(D)$ is injective by Proposition \ref{K includes R(G)}, $\nu_\ast$ extends to an injection $R(G)_S\to K^G_0(D)$ by Remark \ref{localization}. We show that it is surjective. 

Take $x\in K^G_0(D)$. By Corollary \ref{K/Im is a localization}, there are some $f_0,f_1\in\mathbb{Z}[\zeta_{p^m}]$ such that $f_1$ is invertible in $K^G_0(D)/\operatorname{Im}\Phi_{p^m}$ and %
\[x+\Im\Phi_{p^m}=\frac{f_0}{f_1}\] %
holds in $K^G_0(D)/\Im\Phi_{p^m}$. Hypothesis \ref{induction hypothesis} also allows us to take $h_0,h_1\in R(H)$ so that $h_1$ is invertible in $K^H_0(D)$ and %
\[\res^H_G(x)=\frac{h_0}{h_1}\]%
holds in $K^H_0(D)$. By replacing $(f_0,f_1)$ with $(f_0\cdot f_1^{n-1},f_1^n)$ and $(h_0,h_1)$ with $(h_0\cdot h_1^{n-1},h_1^n)$ for $n=\abs{(\mathbb{Z}[t]/(t^{p^{m-1}}-1,p))^\times}$, we  may assume that $f_1$ and $h_1$ satisfy $[f_1]=\pi(h_1)=1$ in $\mathbb{Z}[t]/(t^{p^{m-1}}-1,p)$. By Lemma \ref{cyclic lift}, there exists an element $\widetilde{f_1}\in S$ such that $\widetilde{f_1}+\Im\Phi_{p^m}=f_1$ and $\res^H_G(\widetilde{f_1})=h_1$. Since $x=\widetilde{f_1}x/\widetilde{f_1}$, it remains to show that $\widetilde{f_1}x$ is in $R(G)$. 

Let $C_\nu$ be the mapping cone of the unital inclusion $\nu\colon \C\to D$. By Proposition \ref{K includes R(G)}, there is a short exact sequence %
\[R(G)\overset{\nu_\ast}\rightarrowtail K^G_0(D)\twoheadrightarrow K^G_1(C_\nu).\]%
By Corollary \ref{embedding} (1), $K^G_\ast(C_\nu)$ is uniquely $p$-divisible. Thus $K^G_1(C_\nu)$ is decomposed by the idempotent $\Phi_{p^m}/p$ as %
\begin{align*}K^G_1(C_\nu) &=\qty(1-\frac{\Phi_{p^m}}{p})K^G_1(C_\nu)\oplus \frac{\Phi_{p^m}}{p}K^G_1(C_\nu) \\ &\cong F^G_1(C_\nu)\oplus\Im\qty(\Res^H_G\colon K^G_1(C_\nu)\to K^H_1(C_\nu)).\end{align*}%
Thus an element $z\in K^G_0(D)$ is in $\Im\nu_\ast$ if and only if $[z]$ is in $\mathbb{Z}[\zeta_{p^m}]$ and $\Res^H_G(z)$ is in $R(H)$. Since $[\widetilde{f_1}x]=f_0$ and $\Res^H_G(\widetilde{f_1}x)=h_0$, the element $\widetilde{f_1}x$ is in $\Im\nu_\ast$. \end{proof}

Theorem \ref{cyclic case} follows from Corollary \ref{cyclic p-div case} and Theorem \ref{cyclic not p-div case}.

\subsection{Actions of EPPO-groups}

The main theorem of this paper is as follows.

\begin{thm}\label{EPPO case}Let $G$ be an EPPO-group and $D$ be a strongly self-absorbing $G$-$\Cstar$-algebra that belongs to the $G$-equivariant bootstrap class. Then $D$ satisfies the localization condition with respect to $C(G/K)$ for every cyclic subgroup $K\subset G$. Moreover, there is a $KK^G$-equivalence between $M^{S_D}$ of Proposition \ref{model} and $D$ that maps $[1_{M^{S_D}}]\in K^G_0(M^{S_D})$ to $[1_D]\in K^G_0(D)$.\end{thm}

Before outlining the strategy of the proof of this theorem, we first recall the following notion. 

\begin{dfn}[cf. {\cite[10.1 and 10.5]{Serre}}]Let $G$ be a finite group. For a prime number $p$, a subgroup of $G$ is called \textit{$p$-elementary} if it is the direct product of a $p$-group with a cyclic group of order coprime to $p$. A subgroup $H$ of $G$ is called \textit{elementary} if $H$ is $p$-elementary for some $p$. \end{dfn}

When $G$ is an EPPO-group, an elementary subgroup of $G$ is just a subgroup of prime-power order. The reason for focusing on elementary subgroups is given by the following theorem (see \cite[Theorem 19 of 10.5]{Serre}, for example). 

\begin{thm}[Brauer's theorem]\label{Brauer}Let $G$ be a finite group and $X$ be the set of elementary subgroups of $G$. Then the map %
\[\bigoplus_{H\in X}I^G_H\colon \bigoplus_{H\in X}R(H)\to R(G)\]%
is surjective. \end{thm}

In what follows $G$ denotes an EPPO-group and $D$ denotes a strongly self-absorbing $G$-$\Cstar$-algebra in the $G$-equivariant bootstrap class. For each subgroup $H\subset G$, $\nu^H_\ast\colon R(H)\to K^H_0(D)$ denotes the map induced by the unital inclusion $\nu\colon\C\to D$, and $S^H_D$ denotes the set of elements $f\in R(H)$ with $\nu_\ast^H(f)\in K^H_0(D)$ invertible. The strategy of the proof of Theorem \ref{EPPO case} is as follows. By Theorem \ref{cyclic case}, $K^H_\ast(D)\cong K^H_\ast(\mathbb{C})_{S^H_D}$ for each cyclic subgroup $H\subset G$. We show that $S^H_D=\overline{\res^H_G(S^G_D)}$, where $\overline{\res^H_G(S^G_D)}$ is the saturation of $\res^H_G(S^G_D)$ (Corollary \ref{res(S)}). Then 
\[KK^G_\ast(C(G/H),D)\cong K^H_\ast(D)\cong K^H_\ast(\mathbb{C})_{\res^H_G(S^G_D)}\cong KK^G_\ast(C(G/H),\mathbb{C})_{S^G_D}.\]
Thus $D$ satisfies the localization condition with respect to $C(G/H)$. By Proposition \ref{loc KK}, it implies Theorem \ref{EPPO case}. 

Let $K\subset G$ be a cyclic subgroup. It is easy to see that $\res^K_G(S_D)\subset S^K_D$ To show that $S^K_D=\overline{\res^K_G(S^G_D)}$, for each $f\in S^K_D$, we want to get $\widetilde{f}\in S^G_D$ such that $\res^K_G(\widetilde{f})$ is a multiple of $f$. If such an element $\widetilde{f}$ exists, then the family $(f_H)_{H\in X}$ of elements $f_H\coloneqq \res^H_G(\widetilde{f})\in S^H_D$, where $X$ is the set of elementary subgroups of $G$, satisfies %
\begin{align}\label{res}& \res^L_H(f_H)=f_L\quad (L\subset H), \\
\label{con}& \con_{g,H}(f_H)=f_{{}^gH}\quad (g\in G).\end{align}%
and $f_K$ is a multiple of $f$ (note that $K\in X$ because every cyclic subgroup of an EPPO-group has prime-power order). Conversely, $\widetilde{f}$ is constructed by glueing such a family via the induction $I^G_H\colon R(H)\to R(G)$. The following theorem guarantees the existence of such $(f_H)_{H\in X}$, and will be proved in the next subsection. 

\begin{thm}\label{EPPO family} For each cyclic subgroup $K\subset G$ and each element $f\in S^K_D$, there is a family $(f_H)_{H\in X}$ of elements $f_H\in S^H_D$ satisfying (\ref{res})-(\ref{con}) such that $f_K$ is a multiple of $f$. \end{thm}

Admitting Theorem \ref{EPPO family}, we can show the following. 

\begin{thm}\label{lifting theorem} For each cyclic subgroup $K\subset G$ and each element $f\in S^K_D$, there exists $\widetilde{f}\in S^G_D$ such that $\res^K_G(\widetilde{f})$ is a multiple of $f$. \end{thm}

\begin{proof}Take a family $(f_H)_{H\in X}$ as in Theorem \ref{EPPO family}. By Theorem \ref{Brauer}, there is a family $(\phi_H)_{H\in X}$ of elements $\phi_H\in R(H)$ with $\sum_{H\in X}I^G_H(\phi_H)=1$. Using this, define %
\[\widetilde{f}\coloneqq \sum_{H\in X}I^G_H(\phi_H f_H).\]%
We show that $\res^L_G(\widetilde{f})=f_L$ for each cyclic subgroup $L\subset G$. Then, since $f_L\in S^D_L$, Corollary \ref{res invertible} implies $\widetilde{f}\in S^G_D$. Moreover, $\res^K_G(\widetilde{f})=f_K$ is a multiple of $f$. 

It remains to show that $\res^L_G(\widetilde{f})=f_L$ for each cyclic subgroup $L\subset G$. By Proposition \ref{Dell'ambrogio}, %
\begin{align*}\res^L_G(\widetilde{f})&=\sum_{H\in X}\sum_{g\in L\backslash G/H}I^L_{L\cap {}^gH}\circ\con_{g,L^g\cap H}\circ \res^{L^g\cap H}_H(\phi_Hf_H). \end{align*}%
By Lemma \ref{EPPO family}, we have $\con_{g,L^g\cap H}\circ \res^{L^g\cap H}_H(f_H)=f_{L\cap{}^gH}=\res^{L\cap{}^gH}_L(f_L)$. Thus %
\begin{align*}&\sum_{g\in L\backslash G/H}I^L_{L\cap {}^gH}\circ\con_{g,L^g\cap H}\circ \res^{L^g\cap H}_H(\phi_Hf_H) \\
&=\sum_{g\in L\backslash G/H}I^L_{L\cap {}^gH}\qty(\con_{g,L^g\cap H}\circ \res^{L^g\cap H}_H(\phi_H)\cdot \res^{L\cap{}^gH}_L(f_L)) \\
&=\sum_{g\in L\backslash G/H}I^L_{L\cap {}^gH}\circ\con_{g,L^g\cap H}\circ \res^{L^g\cap H}_H(\phi_H)\cdot f_L \\
&=\res^L_G\circ I^G_H(\phi_H)\cdot f_L,\end{align*}%
where the second and third equality follow from Proposition \ref{Dell'ambrogio}. Hence %
\begin{align*}\res^L_G(f_1)&=\sum_{H\in X}\res^L_G\circ I^G_H(\phi_H)\cdot f_L=f_L\end{align*}%
because $\sum_{H\in X}I^G_H(\phi_H)=1$. \end{proof}

The above theorem implies that the natural map $R(K)\to R(K)_{\res^K_G(S^D_G)}$ induces the map $R(K)_{S^D_K}\to R(K)_{\res^K_G(S^D_G)}$. Since $\res^K_G(S^D_G)\subset S^D_K$, there is also the natural map $R(K)_{\res^K_G(S^D_G)}\to R(K)_{S^D_K}$. They are inverses of each other. Therefore, the following statement holds. 

\begin{cor}\label{res(S)} For each cyclic subgroup $K\subset G$, we have $S^K_D=\overline{\res^K_G(S^G_D)}$. \end{cor}

\begin{proof}The preceding theorem implies that $\overline{\res^K_G(S^G_D)}$ includes $S^K_D$. Since $S^K_D$ is saturated and $\res^K_G(S^G_D)\subset S^K_D$, we obtain the desired conclusion. \end{proof}

Now we prove the main theorem. 

\begin{proof}[Proof of Theorem \ref{EPPO case}] Let $K$ be a cyclic subgroup of $G$. Proposition \ref{reciprocity} (2) gives the isomorphisms $KK^G_\ast(C(G/K),\mathbb{C})_{S^G_D}\cong K^K_\ast(\mathbb{C})_{\res^K_G(S^G_D)}$ and $KK^G_\ast(C(G/K),D)\cong K^K_\ast(D)$. By Theorem \ref{cyclic case} and Corollary \ref{res(S)}, there is an isomorphism $K^K_\ast(\mathbb{C})_{\res^K_G(S^G_D)} \cong K^K_\ast(D)$ such that its composition with the natural map $K^K_\ast(\mathbb{C})\to K^K_\ast(\mathbb{C})_{\res^K_G(S^G_D)}$ is $\nu_\ast\colon K^K_\ast(\mathbb{C})\to K^K_\ast(D)$. Thus, $D$ satisfies the localization condition with respect to $C(G/K)$. By \cite[corollary 3.3]{Meyer and Nadareishvili}, $\mathcal{B}_G$ is generated by the objects $C(G/K)$ for cyclic subgroups $K\subset G$. Therefore, Proposition \ref{loc KK} gives the desired $KK^G$-equivalence between $M^{S_D}$ and $D$. \end{proof}

As a consequence of Theorem \ref{EPPO case}, $S_D=S^G_D$ is the complete invariant of the identity-preserving $KK^G$-equivalence classes of a strongly self-absorbing $D\in\mathcal{B}_G$. 

\begin{cor}Let $G$ be an EPPO-group and $D_i=(D_i,\alpha_i)\ (i=1,2)$ be two strongly self-absorbing $G$-$\Cstar$-algebra that belong to $\mathcal{B}_G$. \begin{itemize} 
\item[$(1)$] There is an $KK^G$-equivalence between $D_1$ and $D_2$ that maps $[1_{D_1}]\in K^G_0(D_1)$ to $[1_{D_2}]\in K^G_0(D_2)$ if and only if $S_{D_1}=S_{D_2}$. 
\item[$(2)$] Suppose that $D_i$ are unital Kirchberg algebras and $\alpha_i$ are isometrically shift-absorbing for $i=1,2$. Then $D_1$ and $D_2$ are conjugate if and only if $S_{D_1}=S_{D_2}$. \end{itemize}\end{cor}

\begin{proof}\begin{itemize}
\item[(1)] Clearly $S_{D_1}=S_{D_2}$ if $D_1$ and $D_2$ are $KK^G$-equivalent preserving identities. The converse follows from Theorem \ref{EPPO case} and Proposition \ref{loc KK} (2). 
\item[(2)] By \cite[Corollary 6.4 (ii)]{Gabe and Szabo}, two isometrically shift-absorbing actions of a compact group $G$ are conjugate if and only if they are $KK^G$-equivalent preserving identities. Hence, the statement follows from (1). \end{itemize}\end{proof}

Combining this result with Proposition \ref{model}, there is the one-to-one correspondence between the set of saturated multiplicative subsets $S\subset R(G)$ and the identity-preserving $KK^G$-equivalence classes of strongly self-absorbing $G$-$\Cstar$-algebra $D$ in $\mathcal{B}_G$. This correspondence is given by $S\mapsto M^S$ and $D\mapsto S_D$. 

\subsection{Proof of Theorem \ref{EPPO family}}

The purpose of this subsection is to prove Theorem \ref{EPPO family}. As in the previous subsection, $G$ denotes an EPPO-group and $D$ is a strongly self-absorbing $G$-$\Cstar$-algebra in the $G$-equivariant bootstrap class $\mathcal{B}_G$. Fix a cyclic subgroup $K\subset G$ and an element $f\in S^K_D$. 

Let $\abs{G}=p_1^{r_1}\cdots p_m^{r_m}$ be the prime factorization of $\abs{G}$. After reindexing, we may assume that there exists $n\geq 0$ such that $K_\ast(D)$ is uniquely $p_i$-divisible for all $1\leq i\leq n$. If $K_\ast(D)$ is not uniquely $p_i$-divisible for any $i$, then we set $n=0$. Let $X_i$ be the set of $p_i$-subgroups of $G$ and $Y_i\subset X_i$ be the set of cyclic $p_i$-subgroups of $G$. There is a unique number $k$ with $K\in Y_k$. In this setting, the statement of Theorem \ref{EPPO family} is reformulated as follows. 

\begin{thm}\label{reformulation} There exist families $(f^{(i)}_H)_{H\in X_i}\ (i=1,\dots, m)$ of elements $f^{(i)}_H\in S^H_D$ such that each $(f^{(i)}_H)_{H\in X_i}$ satisfies (\ref{res})-(\ref{con}), $f^{(k)}_K$ is a multiple of $f$ and $f^{(1)}_{\{e\}} =\dots =f^{(m)}_{\{e\}}$. \end{thm}

Our first step is to obtain $(f^{(k)}_H)_{H\in X_k}$. We begin with the case where $K_\ast(D)$ is not uniquely $p_k$-divisible. 

\begin{lem}\label{pre family} Suppose that $n<k \leq m$. Then there exists a family $(f_{H})_{H\in Y_k}$ of elements $f_{H}\in S^H_D$ satisfying (\ref{res}) such that $f_K$ is a multiple of $f$, $f_{\{e\}}$ is a power of $\res^{\{e\}}_K(f)$ and $f_H-1\in p_k\cdot R(H)$ for each $H$. \end{lem}

\begin{proof}We obtain $f_{H}$ by induction of $r$, where $\abs{H}=p_k^r$. When $r=0$, the cyclic subgroup of order $p_k^0$ is only the trivial subgroup $\{e\}$. Note that $\res^{\{e\}}_K(f)$ is invertible in $K_0(D)$. Since $K_0(D)$ is not uniquely $p_k$-divisible, $\res^{\{e\}}_K(f)\in R(\{e\}) \cong \mathbb{Z}$ is coprime to $p_k$. Hence $\res^{\{e\}}_K(f)^{p_k-1}-1 \in p_k\cdot \mathbb{Z}$. Define $f_{\{e\}}\coloneqq \res^{\{e\}}_K(f)^{p_k-1}$. 

Suppose that we can take $f_{H}$ satisfying the conditions of the statement for all $H\in Y_k$ of order $\abs{H}\leq p_k^{r-1}$ and that there is some $N\in\mathbb{N}$ with $f_H=\res^H_K(f)^N$ for all $H\subset K$. Fix $L\in Y_k$ with $\abs{L}=p^r$. Note that 
\[[f_{p_kL}]=1\in\mathbb{Z}[t]/(t^{p_k^{r-1}}-1,p_k),\] 
where $[f_{p_kL}]$ is the image of $f_{p_kL}$ under the natural projection $R(p_kL)=\mathbb{Z}[t]/(t^{p_k^{r-1}}-1)\to\mathbb{Z}[t]/(t^{p_k^{r-1}}-1,p_k)$. If $L\subset K$, then set $\widetilde{f_L}\coloneqq \res^L_K(f)^N$. If not, then by Lemma \ref{cyclic lift}, we can take $\widetilde{f_L}\in S^L_D$ with $\res^{p_kL}_L(\widetilde{f_{L}})=f_{p_kL}$. Under $R(L)\cong\mathbb{Z}[t]/(t^{p_k^r}-1)$, we may write $\widetilde{f_{L}}=\sum_{i=0}^{p_k^r-1}a_it^i$. Then, modulo $p_k$, 
\[(\widetilde{f_{L}})^{p_k^r}\equiv\sum_{i=0}^{p_k^r-1}a_i^{p_k^r}t^{ip_k^r}=\sum_{i=0}^{p_k^r-1}a_i^{p_k^r}\equiv\sum_{i=0}^{p_k^r-1}a_i.\]
The rightmost integer equals $\res^{\{e\}}_L(\widetilde{f_{L}})=f_{\{e\}}$ because $\res^{\{e\}}_L\colon R(L)\to R(\{e\})$ is identified with the evaluation-at-one map $\mathbb{Z}[t]/(t^{p_k^r}-1)\to\mathbb{Z}$. It implies $(\widetilde{f_{L}})^{p_k^r}\equiv 1$ modulo $p_k$. Once $\widetilde{f_{L}}$ has been obtained for all $L$ with $\abs{L}=p^r$, we set $f_{L}\coloneqq(\widetilde{f_{L}})^{p_k^r}$. For $H\in Y_k$ with $\abs{H}\leq p_k^{r-1}$, replace $f_{H}$ with $(f_{H})^{p_k^r}$. 

By iterating this procedure, we get the desired family $(f_{H})_{H\in Y_k}$. \end{proof}

\begin{lem}\label{cyclic family} The family $(f_{H})_H$ in Lemma \ref{pre family} can be chosen so that $(f_{H})_H$ satisfies (\ref{con}) and $f_{H}-1\in p_k^{r_k}\cdot R(H)$ for all $H$. \end{lem}

\begin{proof}Take $(f_{H})_H$ as in Lemma \ref{pre family}. For $r\in\mathbb{Z}_{>0}$, 
\begin{align*}(f_{H})^{p_k^r}-1 &=((f_{H})^{p_k^{r-1}}-1+1)^{p_k}-1 \\
& =((f_{H})^{p_k^{r-1}}-1)^{p_k}+\sum_{i=1}^{p_k-1} \binom{p_k}{i}((f_{1,H})^{p_k^{r-1}}-1)^i.\end{align*}
Since $f_{H}-1\in p_k\cdot R(H)$, we have $(f_{H})^{p_k}-1\in p_k^2\cdot R(H)$. Inductively, the above formula implies $(f_{H})^{p_k^r}-1\in p_k^{r+1}\cdot R(H)$. Therefore, by replacing $f_{H}$ with $\prod_{g\in G}\con_{g,H^g}((f_{H^g})^{p_k^{r_k-1}})$, the family $(f_{H})_H$ forms the desired one. \end{proof}

We will obtain the desired family $(f_H)_{H\in X_k}$ by extending the family of the preceding lemma, and several lemmas are needed for that. 

For subgroups $H\subset E\subset G$, $N_E(H)$ denotes the normalizer of $H$ in $E$. 

\begin{lem}\label{f_1 lemma}Let $E\in X_k$, $H\subset E$ be a cyclic subgroup and $L\subset E$ be a subgroup. Let $(f_H)_{H\in Y_k}$ be a family obtained from Lemma \ref{cyclic family}. \begin{itemize}
\item[$(1)$]If $H\not\subset L$, then 
\[\res^{H\cap L}_H\qty(\frac{f_{H}-1}{\abs{N_E(H)/H}}-I^H_{p_kH}\qty(\frac{f_{p_kH}-1}{\abs{N_E(H)/p_kH}}))=0.\]
\item[$(2)$]
\begin{align*}&\res^L_E I^E_H\qty(\frac{f_{H}-1}{\abs{N_E(H)/H}}-I^H_{p_kH}\qty(\frac{f_{p_kH}-1}{\abs{N_E(H)/p_kH}})) \\
&=\sum_{g\in L\backslash E/H,\ {}^gH\subset L}I^L_{{}^gH}\qty(\frac{f_{{}^gH}-1}{\abs{N_E(H)/H}}-I^{{}^gH}_{{}^gp_kH}\qty(\frac{f_{{}^gpH}-1}{\abs{N_E(H)/p_kH}})).\end{align*}
\end{itemize}\end{lem}

\begin{proof}\begin{itemize}
\item[(1)]$H\not\subset L$ implies $H\cap L \subsetneqq H$. Since $H$ is a cyclic $p_k$-group, we have $H\cap L=p_kH\cap L$. Moreover, by Lemma \ref{s=1} and Remark \ref{alpha and res-ind}, the map $\res^{p_kH}_H\circ I^H_{p_kH}\colon R(p_kH)\to R(p_kH)$ equals the multiplication by $p_k$. Thus 
\begin{align*}\res^{H\cap L}_H\frac{f_{H}-1}{\abs{N_E(H)/H}} &=\frac{f_{H\cap L}-1}{\abs{N_E(H)/H}} \\
&=p_k\frac{f_{p_kH\cap L}-1}{\abs{N_E(H)/p_kH}} \\
&=\res^{p_kH\cap L}_{p_kH}\qty(p_k\frac{f_{p_kH}-1}{\abs{N_E(H)/p_kH}}) \\
&=\res^{p_kH\cap L}_{p_kH}\res^{p_kH}_H I^H_{p_kH}\qty(\frac{f_{p_kH}-1}{\abs{N_E(H)/p_kH}}) \\
&=\res^{H\cap L}_H I^H_{p_kH}\qty(\frac{f_{p_kH}-1}{\abs{N_E(H)/p_kH}}).\end{align*}%
\item[(2)]Note that for $g\in E,\ l\in L$ and $h\in H$, we have ${}^{lgh}H={}^l({}^gH)$. 
In particular, when $[g_1]=[g_2]$ in $L\backslash E/H$, we have  ${}^{g_1}H\subset L$ if and only if ${}^{g_2}H\subset L$. 
Recall that by Proposition \ref{Dell'ambrogio} and by Lemma \ref{cyclic family}, %
\begin{align*}&\res^L_E\circ I^E_H=\sum_{g\in L\backslash E/H}I^L_{{}^gH\cap L}\circ\con_{g,L^g\cap H}\circ\res^{L^g\cap H}_H, \\
&\con_{g,H}\circ I^H_{p_kH}=I^{{}^gH}_{{}^gp_kH}\circ\con_{g,p_kH}, \\
&\con_{g,H}(f_{H})=f_{{}^gH}.\end{align*}%
Therefore, (1) implies %
\begin{align*}&\res^L_E I^E_H\qty(\frac{f_{H}-1}{\abs{N_E(H)/H}}-I^H_{p_kH}\qty(\frac{f_{p_kH}-1}{\abs{N_E(H)/p_kH}})) \\
&=\sum_{g\in L\backslash E/H}I^L_{{}^gH\cap L}\con_{g,L^g\cap H}\res^{L^g\cap H}_H\qty(\frac{f_{H}-1}{\abs{N_E(H)/H}}-I^H_{p_kH}\qty(\frac{f_{p_kH}-1}{\abs{N_E(H)/p_kH}})) \\
&=\sum_{g\in L\backslash E/H,\ {}^gH\subset L}I^L_{{}^gH}\con_{g,H}\qty(\frac{f_{H}-1}{\abs{N_E(H)/H}}-I^H_{p_kH}\qty(\frac{f_{p_kH}-1}{\abs{N_E(H)/p_kH}})) \\
&=\sum_{g\in L\backslash E/H,\ {}^gH\subset L}I^L_{{}^gH}\qty(\frac{f_{{}^gH}-1}{\abs{N_E(H)/H}}-I^{{}^gH}_{{}^gp_kH}\qty(\frac{f_{{}^gp_kH}-1}{\abs{N_E(H)/p_kH}})).\end{align*}
\end{itemize}\end{proof}

\begin{prp}\label{f_E}For $E\in X_k$, let $C_E$ be a complete set of representatives of the conjugacy classes of cyclic subgroups of $E$. Let $(f_{H})_{H\in Y_k}$ be a family obtained from Lemma \ref{cyclic family}. For each $E\in X_k\setminus Y_k$, define $f_E\in R(E)$ by 
\[f_E\coloneqq 1+I^E_{\{e\}}\qty(\frac{f_{\{e\}}-1}{\abs{E}})+\sum_{H\in C_E\setminus\{e\}}I^E_H\qty(\frac{f_{H}-1}{\abs{N_E(H)/H}}-I^H_{p_kH}\qty(\frac{f_{p_kH}-1}{\abs{N_E(H)/p_kH}})).\]
Then $(f_E)_{E\in X_k}$ satisfies (\ref{res})-(\ref{con}). \end{prp}

\begin{proof}We first note that $f_E$ can be written by the same form as above even if $E\in Y_k$. Indeed, in this case, $C_E$ is just the set of subgroups of $E$ and so 
\begin{align*}& 1+I^E_{\{e\}}\qty(\frac{f_{\{e\}}-1}{\abs{E}})+\sum_{H\in C_E\setminus\{e\}}I^E_H\qty(\frac{f_{H}-1}{\abs{N_E(H)/H}}-I^H_{p_kH}\qty(\frac{f_{p_kH}-1}{\abs{N_E(H)/p_kH}})) \\
& =1+(f_E-1) \\
& =f_E,\end{align*}
where the first equality follows from $I^E_H\circ I^H_{p_kH}=I^E_{p_kH}$ (Proposition \ref{Dell'ambrogio}).

Let $E\in X_k$ and $L\subset E$ be a subgroup. We see $\res^L_E(f_E)=f_L$. By Proposition \ref{Dell'ambrogio} and Lemma \ref{f_1 lemma} (2), 
\begin{align*}\res^L_E(f_E) &=
\begin{aligned}[t] &1+\res^L_E I^E_{\{e\}}\qty(\frac{f_{\{e\}}-1}{\abs{E}}) \\
&+\sum_{H\in C_E\setminus\{e\}}\res^L_E I^E_H\qty(\frac{f_{H}-1}{\abs{N_E(H)/H}}-I^H_{p_kH} \qty(\frac{f_{p_kH}-1}{\abs{N_E(H)/p_kH}})) \end{aligned} \\
&=\begin{aligned}[t] &1+\sum_{g\in L\backslash E}I^L_{{}^g\{e\}} \con_{g,\{e\}} \qty(\frac{f_{\{e\}}-1}{\abs{E}}) \\
&+\sum_{H\in C_E\setminus\{e\}}\sum_{g\in L\backslash E/H,\ {}^gH\subset L}I^L_{{}^gH}\qty(\frac{f_{{}^gH}-1}{\abs{N_E(H)/H}}-I^{{}^gH}_{{}^gp_kH}\qty(\frac{f_{{}^gp_kH}-1}{\abs{N_E(H)/p_kH}}))\end{aligned}\end{align*}
Since $f_{\{e\}}=\con_{g,\{e\}}(f_{\{e\}})$ by Lemma \ref{cyclic family}, 
\[\sum_{g\in L\backslash E}I^L_{{}^g\{e\}} \con_{g,\{e\}} \qty(\frac{f_{\{e\}}-1}{\abs{E}})=I^L_{\{e\}} \qty(\frac{f_{\{e\}}-1}{\abs{L}}).\]%
In the sum $\sum_{H\in C_E\setminus\{e\}}\sum_{g\in L\backslash E/H,\ {}^gH\subset L}$, ${}^gH$ runs over all nontrivial cyclic subgroups of $L$. For a subgroup $H^\prime \subset L$, %
\[\qty{g\in L\backslash E/H^\prime \mid {}^gH^\prime=H^\prime}=N_L(H^\prime)\backslash N_E(H^\prime)/H^\prime=N_L(H^\prime)\backslash N_E(H^\prime).\]%
Thus %
\begin{align*}&\sum_{H\in C_E\setminus\{e\}}\sum_{g\in L\backslash E/H,\ {}^gH\subset L} I^L_{{}^gH} \qty(\frac{f_{{}^gH}-1}{\abs{N_E(H)/H}}-I^{{}^gH}_{{}^gp_kH}\qty(\frac{f_{{}^gp_kH}-1}{\abs{N_E(H)/p_kH}})) \\
&=\sum_{H^\prime \in C_L\setminus \{e\}} \abs{N_L(H^\prime)\backslash N_E(H^\prime)}\cdot I^L_{H^\prime} \qty(\frac{f_{H^\prime}-1}{\abs{N_E(H^\prime)/H^\prime}}-I^{H^\prime}_{p_kH^\prime} \qty(\frac{f_{p_kH^\prime}-1}{\abs{N_E(H^\prime)/p_kH^\prime}})) \\
&=\sum_{H^\prime \in C_L\setminus \{e\}} I^L_{H^\prime} \qty(\frac{f_{H^\prime}-1}{\abs{N_L(H^\prime)/H^\prime}}-I^{H^\prime}_{p_kH^\prime} \qty(\frac{f_{p_kH^\prime}-1}{\abs{N_L(H^\prime)/p_kH^\prime}})).\end{align*}%
Therefore, we have %
\begin{align*}\res^L_E(f_E)
&=\begin{aligned}[t] &1+I^L_{\{e\}}\qty(\frac{f_{\{e\}}-1}{\abs{L}}) \\
&+\sum_{H^\prime \in C_L\setminus \{e\}} I^L_{H^\prime} \qty(\frac{f_{H^\prime}-1}{\abs{N_L(H^\prime)/H^\prime}}-I^{H^\prime}_{p_kH^\prime} \qty(\frac{f_{p_kH^\prime}-1}{\abs{N_L(H^\prime)/p_kH^\prime}}))\end{aligned} \\
&=f_{L}. \end{align*}

(\ref{con}) follows from $\con_{g,E} \circ I^E_H = I^{{}^gE}_{{}^gH} \circ \con_{g,H}$ for all $g\in G$ (Proposition \ref{Dell'ambrogio}) and the fact that $(f_H)_{H\in Y_k}$ satisfies (\ref{con}). \end{proof}

Take $(f_H)_{H\in Y_k}$ as in Lemma \ref{cyclic family} and $f_E$ for $E\in X_k$ as in the previous lemma. Since $\res^H_E(f_E)=f_H$ and $f_H\in S^H_D$ for all cyclic subgroups $H\subset E$, we have $f_E\in S^E_D$ by Corollary \ref{res invertible}. Thus, we get the following results. 

\begin{cor}\label{f lift not p-div} Suppose that $n<k\leq m$. Then for any $f\in S^K_D$, there exists a family $(f_H)_{H\in X_k}$ of elements $f_H\in S^H_D$ satisfying (\ref{res})-(\ref{con}) such that $f_K$ is a multiple of $f$ and $f_{\{e\}}$ is a power of $\res^{\{e\}}_K(f)$. \end{cor}

Next, let us consider the case where $K_\ast(D)$ is uniquely $p_k$-divisible. 

\begin{rem}\label{p-group} Suppose that $G$ is a $p$-group and $K_\ast(D)$ is uniquely $p$-divisible. Then by Corollary \ref{eq of p-div}, we can apply Theorem \ref{UCT} for $KK^G(\mathbb{C},D)$. Since $F^H_1(\mathbb{C})=F^H_1(D)=0$ for a cyclic subgroup $H\subset G$ by Theorem \ref{cyclic case}, 
we have the isomorphism 
\[K^G_0(D)=KK^G(\C,D)\cong\prod_{H\subset G\ \text{cyclic}}\text{Hom}_{\mathbb{Z}[\zeta_{\abs{H}}]\rtimes W_H}(F^H_0(\C),F^H_0(D)).\]
$\text{Hom}_{\mathbb{Z}[\zeta_{\abs{H}}]\rtimes W_H}(F^H_0(\C),F^H_0(D))$ forms a ring with the multiplication induced from $K^G_0(D)\cong KK^G(\C,D)$. Given $f_0, f_1\in\text{Hom}_{\mathbb{Z}[\zeta_{\abs{H}}]\rtimes W_H}(F^H_0(\C),F^H_0(D))$, their product $f_0f_1$ is defined by 
\[f_0f_1\coloneqq(\id\otimes\nu)_\ast^{-1}\circ(f_0\otimes f_1)\colon F^H_0(\C)\to F^H_0(D\otimes D)\to F^H_0(D).\]
By Theorem \ref{F is a localization}, $\text{Hom}_{\mathbb{Z}[\zeta_{\abs{H}}]\rtimes W_H}(F^H_0(\C),F^H_0(D))$ is isomorphic to $F^H_0(D)^{W_H}$: the fixed point subring of $F^H_0(D)$ by the action of $W_H$. These results yields the isomorphism of rings 
\[K^G_0(D)\cong \prod_{H\subset G\ \text{cyclic}} F^H_0(D)^{W_H}.\]
Considering the case $D=M_{p^\infty}$, we also obtain 
\[R(G)_P \cong K^G_0(M_{p^\infty}) \cong \prod_{H\subset G\ \text{cyclic}} F^H_0(M_{p^\infty})^{W_H} \cong \prod_{H\subset G\ \text{cyclic}} \mathbb{Z}[\zeta_{\abs{H}},1/p]^{W_H},\]
where $P\subset\mathbb{Z}$ is a multiplicative subset generated by $1$ and $p$. 
\end{rem}

\begin{lem}\label{p-div lift} Suppose that $1\leq k\leq n$. Then there exists a family $(f_H)_{H\in X_k}$ of elements $f_H\in S^D_H$ satisfying (\ref{res})-(\ref{con}) such that $f_K$ is a multiple of $f$. \end{lem}

\begin{proof}Let $P_k\subset\mathbb{Z}$ be the multiplicative subset generated by $1$ and $p_k$. Note that $K^H_0(D)$ is uniquely $p_k$-divisible for all $H\in X_k$ by Theorem \ref{p-div}. 
Let $\widetilde{\nu_\ast^H}$ denote the map $R(H)_{P_k}\to K^H_0(D)$ induced by $\nu_\ast^H$. First, we construct a family $(f_H^\prime)_{H\in Y_k}$ of elements $f_H^\prime\in R(H)_{P_k}$ such that $\widetilde{\nu_\ast^H}(f_H^\prime)$ is invertible, $f^\prime_K=f$ and $(f_H^\prime)_H$ satisfies (\ref{res}). We obtain $f_H^\prime$ by induction of $r$, where $\abs{H}=p_k^r$. As a first step, set $f_{\{e\}}^\prime\coloneqq \res^{\{e\}}_K(f)$. Suppose that $(f_H^\prime)_{H\in Y_k}$ can be taken for $\abs{H}\leq p_k^{r-1}$. Take $L\in Y_i$ with $\abs{L}=p_k^r$. Corollary \ref{eq of p-div} implies that $D$ has uniquely $p_k$-divisible $K$-theory as a $L$-$\Cstar$-algebra. By (\ref{K decompose}) and (\ref{R decompose}), we have the isomorphisms of rings %
\begin{align*}K^L_0(D)\cong F^L_0(D)\oplus K^{p_kL}_0(D), \\
R(L)_{P_k}\cong \mathbb{Z}[\zeta_{\abs{L}},1/p_k]\oplus R(p_kL)_{P_k}.\end{align*}%
By Theorem \ref{F is a localization} and Theorem \ref{cyclic case}, $F^L_0(D)$ and $K^L_0(D)$ are localizations of $\mathbb{Z}[\zeta_{\abs{L}},1/p_k]$ and $R(L)_{P_k}$, respectively.  
If $L\subset K$, then set $f^\prime_L\coloneqq \res^L_K(f)$. If not, then let $f^\prime_L\in R(L)_{P_k}$ be the element corresponding to %
\[(1,f^\prime_{p_kL})\in \mathbb{Z}[\zeta_{\abs{L}},1/p_k]\oplus R(p_kL)_{P_k}.\] %
Since $(1,\widetilde{\nu_\ast^{p_kL}}(f_{p_kL}^\prime))\in F^L_0(D)\oplus K^{p_kL}_0(D)$ is invertible, $\widetilde{\nu_\ast^L}(f_L^\prime)\in K^L_0(D)$ is invertible. We also have $\res^{p_kL}_K(f_L^\prime)=f_{p_kL}^\prime$. Iterating this procedure, we obtain $f_H^\prime$ for all $H\in Y_k$. \par %
For $H\in Y_k$, replace $f_H^\prime$ with $\prod_{g\in G}\con_{g,H^g}(f_{H^g}^\prime)$. Take $L\in X_k\setminus Y_k$. Remark \ref{p-group} yields the isomorphisms of rings %
\begin{align*}& K^L_0(D)\cong\prod_{H\subset L\ \text{cyclic}} F^H_0(D)^{W_H}, \\
& R(L)_{P_k} \cong \prod_{H\subset L\ \text{cyclic}} \mathbb{Z}[\zeta_{\abs{H}},1/p_k]^{W_H}.\end{align*}%
Let $\overline{f_H^\prime}\in R(H)_{P_k}$ be the image of $f_H^\prime$ under the natural projection $R(H)_{P_k}\to R(H)_{P_k}/\Phi_{\abs{H}}\cong\mathbb{Z}[\zeta_{\abs{H}},1/p_k]$. Since $\con_{g,H}(f_H^\prime)=f_H^\prime$ for all elements $g$ of the normalizer of $H$, $\overline{f_H^\prime}$ is in $\mathbb{Z}[\zeta_{\abs{H}},1/p_k]^{W_H}$. Choose the element $f_L^\prime\in R(L)_{P_k}$ corresponding to $(\overline{f_H^\prime})_{H\subset L}$. Then $\widetilde{\nu_\ast^L}(f_L^\prime)\in K^K_0(D)$ is invertible because $\overline{f_H^\prime}$ is invertible in $F^H_0(D)$ for each $H$. If $L^\prime\subset L$, then $\res^{L^\prime}_L(f^\prime_L)\in R(L^\prime)_{P_k}$ corresponds to $(\overline{f^\prime_H})_{H\subset L^\prime}$. Also, $\con_{g,L}(f^\prime_L)\in R({}^gL)_{P_k}$ corresponds to $(\overline{f^\prime_{{}^gH}})_{H\subset L}$ for $g\in G$. Hence, the family $(f_L^\prime)_{L\in X_k}$ satisfies (\ref{res})-(\ref{con}). 

Pick $N\in\mathbb{N}$ sufficiently large so that $p_k^N\cdot f_H^\prime\in R(H)$ for all $H\in X_k$. Then $f_H\coloneqq p_k^N\cdot f_H^\prime$ forms the desired one. \end{proof}

\begin{lem}\label{integer lift} Suppose $n\geq 1$. Then for any $1\leq i\leq n$ and $a\in S^D_{\{e\}}$, there exists a family $(\psi_H)_{H\in X_i}$ of elements $\psi_H\in S^H_D$ satisfying (\ref{res})-(\ref{con}) such that $\psi_{\{e\}}=p_i^{r_i}a$. \end{lem}

\begin{proof}Let $P_i\subset\mathbb{Z}$ be the multiplicative subset generated by $1$ and $p_i$. We first define the family $(\psi_H^\prime)_{H\in Y_i}$ of elements $\psi_H^\prime\in R(H)_{P_i}$ as follows. Set $\psi_{\{e\}}^\prime\coloneqq a$. If $H\in Y_i$ and we have $\psi_{p_iH}^\prime$, then define $\psi_H^\prime\in R(H)_{P_i}$ as 
\[\psi_H^\prime\coloneqq 1-\frac{\Phi_{\abs{H}}}{p_i}+\frac{1}{p_i}I^H_{p_iH}(\psi_{p_iH}^\prime).\] 
Then $\widetilde{\nu_\ast^H}(\psi_H^\prime)\in K^H_0(D)$ is invertible, where $\widetilde{\nu_\ast^H}$ be the map $R(H)_{P_i}\to K^H_0(D)$ induced by $\nu_\ast^H$. Indeed, the isomorphism 
\[R(H)_{P_i}\cong\mathbb{Z}[\zeta_{\abs{H}},1/p_i]\oplus R(p_iH)_{P_i}.\]
of (\ref{R decompose}) is given by the projection $\pi\colon R(H)_{P_i}\to R(H)_{P_i}/\Phi_{\abs{H}}\cong \mathbb{Z}[\zeta_{\abs{H}},1/p_i]$ and the restriction $R(H)_{P_i}\to R(p_iH)_{P_i}$. By (\ref{indres=Phi}) and (\ref{resind=p}), 
\begin{align*}\pi(\psi_H^\prime)&=\pi\qty(1-\frac{\Phi_{\abs{H}}}{p_i}+\frac{1}{p_i^2}\cdot I^H_{p_iH}\circ\res^{p_iH}_H\circ I^H_{p_iH}(\psi_{p_iH}^\prime)) \\
&=\pi\qty(1-\frac{\Phi_{\abs{H}}}{p_i}+\frac{1}{p_i^2}\cdot \Phi_{\abs{H}}\cdot I^H_{p_iH}(\psi_{p_iH}^\prime)) \\
&=1, \end{align*}
\[\res^{p_iH}_H(\psi_H^\prime)=1-\frac{p_i}{p_i}+\frac{1}{p_i}\cdot p_i\cdot\psi_{p_iH}^\prime=\psi_{p_iH}^\prime. \]
Thus, the invertibility of $\widetilde{\nu_\ast^H}(\psi_H^\prime)$ follows inductively. Also, one easily checks that $(\psi_H^\prime)_{H\in Y_i}$ satisfies (\ref{res})-(\ref{con}) and $\abs{H}\cdot\psi_H^\prime\in R(H)$. 

Next, we define $\psi_L^\prime\in R(L)_{P_i}$ for $L\in X_i\setminus Y_i$ as 
\[\psi_L^{\prime}\coloneqq 1+I^L_{\{e\}}\qty(\frac{\psi_{\{e\}}^{\prime}-1}{\abs{L}})+\sum_{H\in C_L\setminus \{e\}}I^L_H\qty(\frac{\psi_H^{\prime}-1}{\abs{N_L(H)/H}}-I^H_{p_iH}\qty(\frac{\psi_{p_iH}^{\prime}-1}{\abs{N_L(H)/p_iH}})). \]%
As in the proof of Proposition \ref{f_E}, the family $(\psi_L^\prime)_{L\in X_i}$ satisfies (\ref{res})-(\ref{con}). Thus, $\widetilde{\nu_\ast^L}(\psi_L^\prime)\in K^L_0(D)$ is invertible by Corollary \ref{res invertible}. Furthermore, $\abs{H}\cdot\psi_H^\prime\in R(H)$ implies $\abs{L}\cdot\psi_L^\prime\in R(L)$. Therefore, the family of elements $\psi_L\coloneqq p_i^{r_i}\cdot\psi_L^\prime$ is the desired one. \end{proof}

We are now ready to prove Theorem \ref{reformulation}. 

\begin{proof}[Proof of Theorem \ref{reformulation}] The construction of families $(f^{(i)}_H)_{H\in X_i} \ (i=1,\dots, m)$ splits into two cases, according as $k\leq n$ or $k>n$. 

Suppose that $k\leq n$. Take $(f_H^{\prime(k)})_{H\in X_k}$ as in Lemma \ref{p-div lift} and take $(\psi_H^{(k)})_{H\in X_k}$ as in Lemma \ref{integer lift} with $\psi_{\{e\}}^{(k)}=p_1^{r_1}\cdots p_n^{r_n}$. Then set $f_H^{(k)}\coloneqq f_H^{\prime(k)}\psi_H^{(k)}$. Also, for $1\leq j\leq k$ with $j\neq k$, take $(f_H^{(j)})_{H\in X_j}$ as in Lemma \ref{integer lift} with 
\[f_{\{e\}}^{(j)}=p_1^{r_1}\cdots p_n^{r_n}\cdot f_{\{e\}}^{\prime(k)}=f_{\{e\}}^{(k)}. \]
Then the families $(f_H^{(j)})_{H\in X_j}\ (j=1,\dots, n)$ satisfy the desired conditions. If $n=m$, this finish the construction. Assume $n<m$. By Corollary \ref{f lift not p-div}, there is a family $(f_H^{(n+1)})_{H\in X_{n+1}}$ of elements $f_H^{(n+1)}\in S^H_D$ satisfying (\ref{res})-(\ref{con}) such that $f_{\{e\}}^{(n+1)}=(f_{\{e\}}^{(1)})^N$ for some $N\in\mathbb{N}$. By replacing $f_H^{(j)}$ with $(f_H^{(j)})^N$ for $j\leq n$, the desired relations hold for the families $(f_H^{(j)})_{H\in X_j}\ (j=1,\dots, n+1)$. By iterating this argument, we obtain the required families $(f_H^{(j)})_{H\in X_j}$ for $j=1,\dots,m$. 

Suppose that $k>n$. We may assume $k=n+1$. Take $(f_H^{\prime(n+1)})_{H\in X_{n+1}}$ as in Corollary \ref{f lift not p-div}. Again by Corollary \ref{f lift not p-div}, there is a family $(f_H^{\prime\prime(n+1)})_{H\in X_{n+1}}$ of elements $f_H^{\prime\prime(n+1)}\in S^H_D$ satisfying (\ref{res})-(\ref{con}) such that $f_{\{e\}}^{\prime\prime(n+1)}$ is a multiple of $p_1^{r_1}\cdots p_n^{r_n}$. Then set $f_H^{(n+1)}\coloneqq f_H^{\prime(n+1)}\cdot f_{\{e\}}^{\prime\prime(n+1)}$. Since $f_{\{e\}}^{(n+1)}$ is a multiple of $p_1^{r_1}\cdots p_k^{r_n}$, for $1\leq j\leq n$, we can take $(f_H^{(j)})_{H\in X_j}$ as in Lemma \ref{integer lift} with $f_{\{e\}}^{(j)}=f_{\{e\}}^{(n+1)}$. Now the families $(f_H^{(j)})_{H\in X_j}\ (j=1,\dots,n+1)$ satisfy the required conditions. The remainder is the same as in the argument of the previous case. \end{proof}

This completes the proof of Theorem \ref{EPPO family}.

\section{The case of the Lie groups $U(1)$ and $SU(2)$}

Although we cannot verify Conjecture \ref{conjecture} for $G=U(1)$ or $G=SU(2)$, the structure of $K^G_*(D)$ is rather restricted 
for a strongly self-absorbing $G$-$\Cstar$-algebra $D$ in $\mathcal{B}_G$. 

The main tool for studying $K^G_\ast(D)$ is the following. 

\begin{thm}[cf. {\cite[Theorem 5.1]{Rosenberg and Schochet 1986}}]\label{spec seq}
Let $G$ be a compact Lie group such that the fundamental group $\pi_1(G)$ is torsion-free. Then for every $A\in \mathcal{B}_G$ and $B\in KK^G$, there exists a strongly convergent spectral sequence of $R(G)$-modules
\[E^2_{p,q}={\Tor}^{R(G)}_p(K^G_\ast(A), K^G_{\ast+q}(B)) \Longrightarrow K^G_{\ast+p+q}(A\otimes B)\]
such that the edge homomorphism
\[K^G_\ast(A)\otimes_{R(G)} K^G_{\ast+q}(B) = E^2_{0,q} \to E^{\infty}_{0,q} \to K^G_{\ast+q}(A\otimes B)\]
is given by the cup product. \end{thm}

In what follows, let $G=U(1)$ or $SU(2)$. Then $R(G)=\mathbb{Z}[t,1/t]$ or $\mathbb{Z}[t]$, respectively. In either case, the global dimension of $R(G)$ is two. In other words, for all $p\geq 3$ and $R(G)$-modules $M,N$, we have $\Tor^{R(G)}_p(M,N)=0.$ Hence, in this case, $E^2_{p,q}=0$ if $p\geq 3$ or $p<0$. 

Let $d^r_{p,q} \colon E^r_{p,q} \to E^r_{p-r, q+r-1}$ be the differential. The $E^3$-page of this spectral sequence is computed as
\begin{align}\label{E3} E^3_{p,q}=\Ker d^2_{p,q}/\Im d^2_{p+2,q-1}=\begin{dcases*}
0 & if $p\geq 3$ or $p<0$, \\
\Ker d^2_{2,q} & if $p=2$, \\
E^2_{1,q} & if $p=1$, \\
E^2_{0,q}/\Im d^2_{2,q-1} & if $p=0$. \end{dcases*}\end{align}
This yields $d^3_{p,q}=0$ and $E^{\infty}_{p,q}=E^3_{p,q}$. Let 
\[0=F_{-1}H_n \subset F_0H_n \subset \dots \subset F_nH_n=K^G_{\ast+n}(A\otimes B)\]
be the filtration with $F_pH_n/F_{p-1}H_n=E^{\infty}_{p,n-p}=E^3_{p,n-p}$. Since $E^3_{p,q}=0$ for $p\geq 3$, we have $F_2H_n=F_nH_n=K^G_{\ast+n}(A\otimes B)$. Now there are two short exact sequences 
\begin{align}& E^3_{0,n} \rightarrowtail F_1H_n \twoheadrightarrow E^3_{1,n-1}, \\
& F_1H_n \rightarrowtail K^G_{\ast+n}(A\otimes B) \twoheadrightarrow E^3_{2,n-2}. \end{align}
Also, $E^2_{p,q}={\Tor}^{R(G)}_p(K^G_\ast(A), K^G_{\ast+q}(B))$ gives the following two short exact sequences
\begin{align}& \Ker d^2_{2,q} \rightarrowtail {\Tor}^{R(G)}_2(K^G_\ast(A), K^G_{\ast+q}(B)) \overset{d^2_{2,q}}{\twoheadrightarrow} \Im d^2_{2,q}, \\
& \Im d^2_{2,q} \rightarrowtail K^G_\ast(A)\otimes_{R(G)} K^G_{\ast+q+1}(B) \twoheadrightarrow E^2_{0,q+1}/\Im d^2_{2,q}. \end{align}
Combining these four exact sequences with (\ref{E3}), we obtain the following two exact sequences 
\begin{align}
& \begin{aligned}[t] \Ker d^2_{2,q} \rightarrowtail {\Tor}^{R(G)}_2(K^G_\ast(A), K^G_{\ast+q} & (B)) \overset{d^2_{2,q}}{\to} K^G_\ast(A)\otimes_{R(G)} K^G_{\ast+q+1}(B) \\ 
& \to F_1H_{q+1} \twoheadrightarrow \Tor^{R(G)}_1(K^G_\ast(A), K^G_{\ast+q}(B)), \end{aligned} \label{long} \\
& F_1H_{q+1} \rightarrowtail K^G_{\ast+q+1}(A\otimes B) \twoheadrightarrow \Ker d^2_{2,q-1}. \label{F_1} \end{align}
Note that the composition map $K^G_\ast(A)\otimes_{R(G)} K^G_{\ast+q+1}(B) \to F_1H_{q+1} \to K^G_{\ast+q+1}(A\otimes B)$ is the cup product. 

\begin{prp} Let $D$ be a strongly self-absorbing $G$-$\Cstar$-algebra in $\mathcal{B}_G$ and $C_\nu$ be the mapping cone of the unital inclusion $\nu\colon \mathbb{C} \to D$. Then 
\begin{align}& \Tor^{R(G)}_1(K^G_\ast(D),K^G_\ast(D))=0, \label{Tor_1(D,D)} \\
& \Tor^{R(G)}_1(K^G_\ast(D),K^G_\ast(C_\nu))=0, \label{Tor_1(D,C_nu)} \\
&\Tor^{R(G)}_2(K^G_\ast(D),K^G_{\ast-1}(C_\nu)) \cong K^G_\ast(D)\otimes_{R(G)}K^G_\ast(C_\nu) \label{Tor_2(D,C_nu)} \end{align}
and there is an exact sequence 
\begin{align} \Tor^{R(G)}_2(K^G_\ast(D),K^G_{\ast-1}(D)) \rightarrowtail K^G_\ast(D)\otimes_{R(G)} K^G_\ast(D) \twoheadrightarrow K^G_\ast(D\otimes D). \label{D tensor D} \end{align}\end{prp}

\begin{proof}First, apply (\ref{long}) and (\ref{F_1}) to the case $A=B=D$. Since the cup product $K^G_\ast(D)\otimes_{R(G)} K^G_\ast(D) \to K^G_\ast(D\otimes D)$ is surjective by Proposition \ref{ring structure}, $F_1H_{q+1} \cong K^G_\ast(D\otimes D)$ and the map $K^G_\ast(D)\otimes_{R(G)} K^G_{\ast+q+1}(D) \to F_1H_{q+1}$ is surjective. Hence $\Tor^{R(G)}_1(K^G_\ast(D),K^G_\ast(D)) = \Ker d^2_{2,q-1} =0$. Repeating this argument for $q+1$ instead of $q$, we also have $\Ker d^2_{2,q}=0$. Thus $d^2_{2,q}\colon \Tor^{R(G)}_2(K^G_\ast(D),K^G_{\ast-1}(D)) \to K^G_\ast(D)\otimes_{R(G)} K^G_\ast(D)$ is injective and so $(\ref{D tensor D})$ holds. 

Next, apply (\ref{long}) and (\ref{F_1}) to the case $A=D$ and $B=C_\nu$. Corollary \ref{multiplication} yields $F_1H_{q+1}=K^G_{\ast+q+1}(D\otimes C_\nu)=\Ker d^2_{2,q-1}=0$. Thus $\Tor^{R(G)}_1(K^G_\ast(D),K^G_{\ast+q}(C_\nu))=0$, $\Ker d^2_{2,q}=0$ and consequently $\Tor^{R(G)}_2(K^G_\ast(D),K^G_{\ast+q}(C_\nu)) \cong K^G_\ast(D)\otimes_{R(G)}K^G_{\ast+q+1}(C_\nu)$. \end{proof}

Since $R(G)\otimes_{\mathbb{Z}} \mathbb{Q}=\mathbb{Q}[t,1/t]$ or $\mathbb{Q}[t]$ is a PID, (\ref{Tor_1(D,D)}) and Lemma \ref{Tor(M,M)} imply that $K^G_\ast(D)\otimes_{\mathbb{Z}} \mathbb{Q}$ is flat. Also, by (\ref{D tensor D}), 
\[(K^G_\ast(D)\otimes_{\mathbb{Z}} \mathbb{Q}) \otimes_{R(G)\otimes_{\mathbb{Z}} \mathbb{Q}} (K^G_\ast(D)\otimes_{\mathbb{Z}} \mathbb{Q}) \cong K^G_\ast(D)\otimes_{\mathbb{Z}} \mathbb{Q}. \]
A similar argument to the proof of Theorem \ref{F is a localization} shows that $K^G_1(D)\otimes_{\mathbb{Z}} \mathbb{Q}=0$ and $K^G_0(D)\otimes_{\mathbb{Z}} \mathbb{Q}$ is a localization of $R(G)$. In particular, $K^G_1(D)$ is torsion as a $\mathbb{Z}$-module. 
This localization procedure can be realized by $D\otimes M_\Q$, which remains a strongly self-absorbing $G$-$\Cstar$-algebra, where $M_\Q$ is the UHF algebra satisfying $K_0(M_\Q)\cong \Q$.

\begin{prp} Let the notation be as above. Then, 
\begin{itemize}
\item[$(1)$] $K^G_1(D)$ is torsion as a $\Z$-module. 
\item[$(2)$] Unless $K^G_*(D)$ is trivial, the map $\nu_*:R(G)\to K_0^G(D)$ is injective and $K_0^G(C_\nu)\cong K_1^G(D)$.  
\end{itemize}
\end{prp}

\begin{proof}
We show only (2). 
If $\nu_*:R(G)\to K_0^G(D)$ is injective, the mapping cone sequence implies $K_0^G(C_\nu)\cong K_1^G(D)$. 

Assume that $\nu_*:R(G)\to K_0^G(D)$ is not injective. 
We first show that $K_0^G(D)$ is torsion as a $\Z$-module. 
Let $J_0$ be the kernel of $\nu_*$, which is an ideal of $R(G)$ as $\nu_*$ is an algebra homomorphism. 
Let $J_1$ be the kernel of $\nu_*\otimes \id_\Q:R(G)\otimes_\Z\Q\to K_0^G(D)\otimes_\Z\Q$. 
Then $J_1$ is an ideal of $R(G)\otimes_\Z \Q$ satisfying $J_0\otimes_\Z \Q\subset J_1$. 
Since $K_1^G(D\otimes M_\Q)=0$, the six-term exact sequence arising from 
\[\Sigma D\otimes M_\Q\rightarrowtail  C_\nu\otimes M_\Q
\twoheadrightarrow \C\otimes M_\Q\] 
implies $K_0^G(C_\nu)\otimes \Q\cong J_1$. 
Then (\ref{Tor_2(D,C_nu)}) shows $(K_0^G(D)\otimes \Q)\otimes_{R(G)\otimes\Q}J_1=0$. 
Since $R(G)\otimes \Q$ is a PID, the ideal $J_1$ is isomorphic to $R(G)\otimes_\Z \Q$ as an $R(G)\otimes_\Z \Q$-module, 
which means $K_0^G(D)\otimes_\Z\Q=0$. 
Thus $K_0^G(D)$ is torsion as a $\Z$-module. 

Since $R(G)$ is free as a $\Z$-module, the map $\nu_*$ is trivial, and we get $K_0^G(D)\cong K_1(C_\nu)$ and 
the split exact sequence 
\[K_1^G(D)\rightarrowtail K_0^G(C_\nu) \twoheadrightarrow R(G).\]
Thus 
\[\Tor_2^{R(G)}(K_{*-1}^G(D),K_{*}^G(D))\cong  \Tor_2^{R(G)}(K_{*}^G(C_\nu),K_*^G(D)),\]
\[K_0^G(D)\otimes_{R(G)}K_0^G(D)\cong K_1^G(C_\nu)\otimes_{R(G)}K_0^G(D),\]
and $K_1^G(D)\otimes_{R(G)}K_1^G(D)$ is a direct summand of $K_0^G(C_\nu)\otimes_{R(G)}K_1^G(D)$. 
Now (\ref{Tor_2(D,C_nu)}) and (\ref{D tensor D}) imply that $K_*^G(D)$ is trivial. 
\end{proof}

\begin{rem}When $K_*^G(D)$ is non-trivial, we can further show the following from (\ref{Tor_2(D,C_nu)}) and (\ref{D tensor D}).  
\begin{itemize}	
	\item [(1)] The following exact sequence splits: 
	\[K_0^G(D) \otimes_{R(G)} R(G)\rightarrowtail K_0^G(D)\otimes_{R(G)}K_0^G(D) 
	\twoheadrightarrow K_0^G(D)\otimes_{R(G)}K_1^G(C_\nu).\] 
	\item [(2)] $K_1^G(D)\otimes_{R(G)}K_1^G(C_\nu)=0$ and $1\otimes \nu_*:K_1^G(D)\otimes_{R(G)}R(G) \to K_1^G(D)\otimes_{R(G)} K_0^G(D)$ is an isomorphism. 
\item [(3)]$\operatorname{Tor}_2^{R(G)}(K_*^G(D),K_*^G(D))\cong K_1^G(D)$. 
\end{itemize}
Since we do not need these facts, we omit the proof. 
\end{rem}

Unfortunately, we are unable to prove that $K^G_1(D)=0$ itself. In what follows, we proceed under this assumption. 

In general, a ring homomorphism $R \to A$ is called \textit{epimorphic} if the multiplication map $A\otimes_R A \to A$ is an isomorphism. 

\begin{thm} \label{flat epi} Suppose that $K^G_1(D)=0$. Then $D$ satisfies the localization condition with respect to $\mathbb{C}$. \end{thm}

\begin{proof} The assumption $K^G_1(D)=0$ and (\ref{D tensor D}) give the isomorphism 
\[K^G_0(D) \otimes_{R(G)} K^G_0(D) \cong K^G_0(D). \]
Hence, $\nu_\ast\colon R(G)\to K^G_0(D)$ is epimorphic. By (\ref{Tor_1(D,D)}) and \cite[Theorem 1.1]{flat ring epi}, $K^G_0(D)$ is in fact a flat $R(G)$-module. Since $R(G)$ is a UFD, the associated divisor class group $\operatorname{Cl}(R(G))$ and the Picard group $\operatorname{Pic}(R(G))$ are trivial (see \cite[p.165--167]{Matsumura}, for example). Thus, by \cite[Theorem 1.3 (i)(ii)]{flat ring epi}, $K^G_0(D)$ is a localization of $R(G)$. \end{proof}

\begin{prp}\label{loc w.r.t C} If $D$ satisfies the localization condition with respect to $\mathbb{C}$, then $D$ satisfies it with respect to every compact object in $\mathcal{B}_G$. \end{prp}

\begin{proof}Let $B\in\mathcal{B}_G$ be a compact object. By \cite[Proposition 3.13]{Lefschetz}, there exists an object $B^\ast \in KK^G$ such that for every $A\in KK^G$, there is an isomorphism 
\[KK^G(B,A) \cong K^G_0(B^\ast \otimes A)\]
natural in $A$. Thus, to prove that $D$ satisfies the localization condition with respect to $B$, it suffices to show that $K^G_\ast(B^\ast \otimes D) \cong K^G_\ast(B^\ast)_{S_D}$. Since $K^G_0(D) \cong R(G)_{S_D}$ and $K^G_1(D)=0$, $K^G_\ast(D)$ is a flat $R(G)$-module. Hence $\Tor^{R(G)}_p (K^G_\ast(B^\ast), K^G_\ast(D))=0$ for all $p\geq 1$. This and Theorem \ref{spec seq} give an isomorphism 
\[K^G_\ast(B^\ast) \otimes_{R(G)} K^G_\ast(D) \cong K^G_\ast(B^\ast \otimes D). \]
The left-hand side is isomorphic to $K^G_\ast(B^\ast) \otimes_{R(G)} R(G)_{S_D} = K^G_\ast(B^\ast)_{S_D}$. This completes the proof. \end{proof}

\begin{rem}The proof of the preceding proposition does not use the assumption that $G=U(1)$ or $SU(2)$. In other words, this proposition holds for every compact Lie group $G$ with $\pi_1(G)$ torsion-free. \end{rem}

Combining Theorem \ref{flat epi}, Proposition \ref{loc w.r.t C} and Proposition \ref{loc KK}, we obtain the following result. 

\begin{cor}Let $D$ be a strongly self-absorbing $G$-$\Cstar$-algebra in $\mathcal{B}_G$ for $G=U(1)$ or $SU(2)$. If $K^G_1(D)=0$, then there is a $KK^G$-equivalence between $M^{S_D}$ and $D$ that maps $[1_{M^{S_D}}] \in K^G_0(M^{S_D})$ to $[1_D]\in K^G_0(D)$. \end{cor}

\begin{cor}Let $D$ be a strongly self-absorbing $G$-$\Cstar$-algebra in $\mathcal{B}_G$ for $G=U(1)$ or $SU(2)$. 
Then there is a $KK^G$-equivalence between $M^{S_{D\otimes M_\Q}}$ and $D\otimes M_\Q$ that maps 
$[1_{M^{S_{D\otimes M_\Q}}}] \in K^G_0(M^{S_{D\otimes M_\Q}})$ to $[1_{D\otimes M_\Q}]\in K^G_0(D\otimes M_\Q)$. \end{cor}

We conclude this paper with the following peculiar example. 

\begin{ex} For $G=U(1)$, let $S$ be the multiplicative set in $R(G)=\Z[t,t^{-1}]$ generated by $\{1-t^n\}_{n=1}^\infty$. 
Then $K_*^{\Z/n}(M^S)=0$ for all natural numbers $n$ while $K_0^G(M^S)\cong \Z[t,t^{-1}]_S\neq 0$. 
We can realize the $KK^G$-equivalence class of this model by an isometrically shift absorbing action $\alpha$ on the 
Cuntz algebra $\mathcal{O}_2$. 
This $\alpha$ satisfies the following property: the restriction $\beta_n$ of $\alpha$ to $\Z/n$ for any natural number is the unique $\Z/n$-action with the Rohlin property and $\mathcal{O}_2\rtimes_{\beta_n}\Z/n\cong \mathcal{O}_2$, 
while $\mathcal{O}_2\rtimes_\alpha U(1)$ is not isomorphic to $\mathcal{O}_2$.  
Indeed, we can show from the construction that $\beta_n$ is approximately representable, and its dual action has the Rohlin property (see \cite[Lemma 3.8]{Izumi2004}). 
Since $K_*^{\Z/n}(\mathcal{O}_2)=0$, the crossed product $\mathcal{O}_2\rtimes_{\beta_n}\Z/n$ is isomorphic to $\mathcal{O}_2$. 
Since the unqiue Rohlin action of $\Z/n$ on $\mathcal{O}_2$ is self-daul, we see that $\beta_n$ itself has the Rohlin prolerty. 
Choosing larger multiplicative sets $S$, we can construct a large variety of $U(1)$-actions with the same property.  
\end{ex}

\end{document}